\documentclass[11pt,oneside]{amsart}

\usepackage{latexsym,amssymb,bbm}
\usepackage{MnSymbol}
\usepackage{amsmath}
\usepackage{enumerate}
\usepackage[all]{xy}

\DeclareMathAlphabet\mathbb{U}{msb}{m}{n}

\newtheorem{thm}{Theorem}
\newtheorem{cor}[thm]{Corollary}

\newtheorem{ex}[thm]{Example}

\newtheorem{lem}[thm]{Lemma}
\newtheorem{prop}[thm]{Proposition}
\newtheorem{defn}[thm]{Definition}
\newtheorem{rem}[thm]{Remark}
\numberwithin{equation}{section}

\newcounter{nnn}  
\newcommand{\espai}{\vspace{0.3 truecm}}

\newcommand{\Aa}{\mathcal{A}}
\newcommand{\Bb}{\mathcal{B}}
\newcommand{\Cc}{\mathcal{C}}
\newcommand{\Dd}{\mathcal{D}}
\newcommand{\Ee}{\mathcal{E}}
\newcommand{\Ff}{\mathcal{F}}
\newcommand{\Gg}{\mathcal{G}}
\newcommand{\Hh}{\mathcal{H}}
\newcommand{\Ii}{\mathcal{I}}

\newcommand{\Kk}{\mathcal{K}}

\newcommand{\Mm}{\mathcal{M}}

\newcommand{\Pp}{\mathcal{P}}

\newcommand{\Ss}{\mathcal{S}}
\newcommand{\Tt}{\mathcal{T}}
\newcommand{\Uu}{\mathcal{U}}
\newcommand{\Vv}{\mathcal{V}}
\newcommand{\Ww}{\mathcal{W}}

\newcommand{\AAA}{\mathbb{A}}

\newcommand{\EE}{\mathbb{E}}
\newcommand{\FF}{\mathbb{F}}
\newcommand{\GG}{\mathbb{G}}
\newcommand{\HH}{\mathbb{H}}
\newcommand{\II}{\mathbb{I}}
\newcommand{\JJ}{\mathbb{J}}
\newcommand{\KK}{\mathbb{K}}

\newcommand{\MM}{\mathbb{M}}

\newcommand{\PP}{\mathbb{P}}

\newcommand{\SSS}{\mathbb{S}}
\newcommand{\TT}{\mathbb{T}}
\newcommand{\UU}{\mathbb{U}}

\newcommand{\WW}{\mathbb{W}}

\newcommand{\ZZ}{\mathbb{Z}}

\newcommand{\Asf}{\mathsf{A}}

\newcommand{\Dsf}{\mathsf{D}}

\newcommand{\Fsf}{\mathsf{F}}
\newcommand{\Gsf}{\mathsf{G}}
\newcommand{\Hsf}{\mathsf{H}}

\newcommand{\Ksf}{\mathsf{K}}
\newcommand{\Lsf}{\mathsf{L}}
\newcommand{\Msf}{\mathsf{M}}

\newcommand{\Ssf}{\mathsf{S}}

\newcommand{\Zsf}{\mathsf{Z}}

\newcommand{\Abf}{\mathbf{A}}

\newcommand{\Cbf}{\mathbf{C}}

\newcommand{\To}{\rightarrow}

\begin{document}

\title{Permutation 2-groups I: structure and splitness}
\author{Josep Elgueta}
\address{Applied Mathematics Department\\Polytechnical University of Catalonia}
\email{josep.elgueta@upc.edu}

\thanks{This work was partially supported by the Generalitat de Catalunya (Project: 2009 SGR 1284) and the Ministerio de Educaci\'on y Ciencia of Spain (Projects: MTM2009-14163-C02-02 and MTM2012-38122-C03-01).
}

\begin{abstract}
By a 2-group we mean a groupoid equipped with a weakened group structure. It is called split when it is equivalent to the semidirect product of a discrete 2-group and a one-object 2-group. By a permutation 2-group we mean the 2-group $\SSS ym(\Gg)$ of self-equivalences of a groupoid $\Gg$ and natural isomorphisms between them, with the product given by composition of self-equivalences. These generalize the symmetric groups $\Ssf_n$, $n\geq 1$, obtained when $\Gg$ is a finite discrete groupoid.

After introducing the wreath 2-product $\Ssf_n\wr\wr\ \GG$ of the symmetric group $\Ssf_n$ with an arbitrary 2-group $\GG$, it is shown that for any (finite type) groupoid $\Gg$ the permutation 2-group $\SSS ym(\Gg)$ is equivalent to a product of wreath 2-products of the form $\Ssf_n\wr\wr\ \SSS ym(\Gsf)$ for a group $\Gsf$ thought of as a one-object groupoid. This is next used to compute the homotopy invariants of $\SSS ym(\Gg)$ which classify it up to equivalence. Using a previously shown splitness criterion for strict 2-groups, it is then proved that $\SSS ym(\Gg)$ can be non-split, and that the step from the trivial groupoid to an arbitrary one-object groupoid is the only source of non-splitness. Various examples of permutation 2-groups are explicitly computed, in particular the permutation 2-group of the underlying groupoid of a (finite type) 2-group. It also follows from well known results about the symmetric groups that the permutation 2-group of the groupoid of all finite sets and bijections between them is equivalent to the direct product 2-group $\ZZ_2[1]\times\ZZ_2[0]$, where $\ZZ_2[0]$ and $\ZZ_2[1]$ stand for the group $\ZZ_2$ thought of as a discrete and a one-object 2-group, respectively.
\end{abstract}

\maketitle

\tableofcontents

\section{\large Introduction}

In the last two or three decades, a considerable effort has been made to categorify some parts of mathematics. Roughly, the idea is to take a theory whose objects are sets equipped with some structure, and to develop an analogous theory where the objects are categories equipped with a similar structure. For instance, this idea is carried out by Breen \cite{lB92} and in many subsequent works devoted to the categorification of the theory of group extensions, or by Bernstein, Frenkel and Khovanov \cite{BFK99} and Frenkel, Khovanov and Stroppel \cite{FKS06}, who categorify the representation theory of quantum groups with a view toward the construction of TQFT's in four dimensions \cite{CF94}.

In fact, it quickly became clear that the theory merits to be pushed on to the higher dimensional setting of $n$-categories for any $n\geq 1$, and to $\infty$-categories. Lurie's book \cite{jLbook09} lays the foundations for this ``infinite dimensional mathematics''. For example, Toen and Vezzosi (\cite{TV05}, \cite{TV08}) developed a theory of algebraic geometry on monoidal $\infty$-categories, generalizing algebraic geometry over a monoidal category as developed by Deligne \cite{pD90}, and Ben-Zvi, Francis and Nadler \cite{BFN10} worked on a geometric $\infty$-function theory, a sort of $\infty$-version of the matrix description of a linear map between vector spaces.

This paper should be viewed as part of this program. It is intended to be a first step toward a theory of {\em permutation $\infty$-groups}. By this I mean the ($\infty$-)groups of self-equivalences of an $\infty$-{\em groupoid}. Recall that an $\infty$-groupoid is an $\infty$-version of a set, where there exists morphisms between the elements or 1-morphisms, morphisms between the 1-morphisms or 2-morphisms, and so on, all $k$-morphisms being invertible for any $k\geq 1$, at least up to a (weakly) invertible $(k+1)$-morphism. According to the well known homotopy hypothesis first envisaged by Grothendieck in his famous letter to Quillen, an $\infty$-groupoid is equivalent to a topological space (modulo weak homotopy equivalences). As a first step, in this work I only consider {\em permutation 2-groups}, i.e. the (2-)groups of self-equivalences of arbitrary groupoids ($\infty$-groupoids with only identity $k$-morphisms for any $k>1$).

{\em 2-groups}, also called {\em categorical groups} or {\em gr-categories}, go back, in the disguised form of a crossed module of groups, to works by Whitehead \cite{jhcW49} and Eilenberg and MacLane \cite{EM47} in the late 40's. It is not surprising, then, that there is a considerable body of work devoted to the study of some aspect of 2-groups. Thus there are works on the relationship between $n$-groups and the homotopy types (for instance \cite{CC91},\cite{BCD93}), works on the theory of extensions of 2-groups (for instance \cite{lB92}, \cite{aR03}, \cite{BV02}) or works on the representation theory of 2-groups (for instance \cite{jE4}, \cite{jE6}, \cite{CY05}, \cite{BBFW12}), and, on the {\em 2-representation theory} of groups viewed as discrete 2-groups (for instance \cite{pD97}, \cite{GK08}).

Roughly, a 2-group is a groupoid equipped with a suitably weakened group structure. Its theory can be seen as a refinement of the classical theory of groups. In fact, any notion defined for 2-groups ultimately amounts to various ``classical'' objects interacting in the appropriate way. The notion of 2-group itself provides an example. According to Sinh's theorem \cite{hxSi75} (see also \S~\ref{invariants_homotopics} below), up to the appropriate notion of equivalence, a 2-group $\GG$ is completely given by its {\em first homotopy group} $\pi_0(\GG)$, an abelian group $\pi_1(\GG)$ on which $\pi_0(\GG)$ acts, called its {\em second homotopy group}, and a 3-cocycle of $\pi_0(\GG)$ with values in $\pi_1(\GG)$. The cohomology class of this 3-cocycle wil be denoted by $\alpha(\GG)$ and called the {\em Postnikov invariant} of $\GG$. I shall refer to $\pi_0(\GG),\pi_1(\GG),\alpha(\GG)$ as the {\em homotopy invariants} of $\GG$. 

2-groups naturally appear as 2-groups of self-equivalences of the objects in a 2-category. This is the case of the permutation 2-groups, defined as the 2-groups of self-equivalences of the objects in the 2-category $\mathbf{Gpd}$ of (small) groupoids, functors and natural transformations between these. For any groupoid $\Gg$ we shall denote the corresponding permutation 2-group by $\SSS ym(\Gg)$. The product is given by the composition of self-equivalences and the horizontal composition of natural isomorphisms. When $\Gg$ is the discrete groupoid associated to a set $X$, $\SSS ym(\Gg)$ reduces to the usual symmetric group $\mathsf{Sym}(X)$ viewed as a discrete 2-group (with only identity morphisms). 

There is a particularly simple class of 2-groups. It includes the discrete 2-groups, denoted $\Gsf[0]$, with $\Gsf$ any group, and the one-object (or connected) 2-groups, denoted $\Asf[1]$, with $\Asf$ any abelian group. I mean the 2-groups which, up to equivalence, are the semidirect product of a discrete 2-group $\Gsf[0]$ and a one-object 2-group $\Asf[1]$ for some abelian group $\Asf$ on which $\Gsf$ acts. They will be called {\em split 2-groups} by reasons which will become clear below (see Theorem~\ref{criteri_general_2}). They are characterized by the fact that their Postnikov invariant is zero (see Theorem~\ref{criteri_general}).

Split 2-groups are interesting for at least three reasons. Firstly, many important examples of 2-groups are split. For instance, the ``general linear 2-groups'' associated to some higher versions of a vector space are split (Example~\ref{exemple_2-grup_split} below). Secondly, split 2-groups are considerably much easier to study than arbitrary 2-groups. In fact, some authors, such as Crane and Yetter \cite{CY05}, directly restrict to split 2-groups when studying their representation theory. Finally, split 2-groups are almost generic. Indeed, it follows from Sinh's theorem that any 2-group is equivalent to a sort of twisted version of a split 2-group whose associator is no longer trivial (see Proposition~\ref{A[1]_rtimes_G[0]} below and the comment following it).   

The purpose of this work is to investigate the structure and split character of the permutation 2-groups $\SSS ym(\Gg)$, and to compute their homotopy invariants. To do that, we first introduce the {\em wreath 2-product} of the symmetric group $\Ssf_n$ with an arbitrary 2-group $\GG$. When $\GG$ is discrete, this product reduces to the usual wreath product for groups. The first main result of the paper is then the following (cf. Theorem~\ref{teorema_estructura} in the main text).

\medskip
\noindent
{\sc Theorem.} {\em Let $\{(n_i,\Gsf_i)\}_{i\in I}$ be any family of pairs consisting of a positive integer $n_i\geq 1$ and a group $\Gsf_i$, with $\Gsf_i\ncong\Gsf_{i'}$ for $i\neq i'$. Let $\Gg_i$ be the coproduct of $n_i$ copies of $\Gsf_i$ as a one-object groupoid, and let $\Gg=\coprod_{i\in I}\Gg_i$ (such a groupoid is called of {\em finite type}). Then there is an equivalence of 2-groups
$$
\SSS ym(\Gg)\cong\prod_{i\in I}\Ssf_{n_i}\wr\wr\ \SSS ym(\Gsf_i),
$$
where $\prod$ denotes the (2-)product in the (2-)category of 2-groups.}

\medskip
Here and throughout the entire paper I follow the usual convention in groupoid theory of regarding groups and one-object groupoids as the same thing, and just denote by $\Gsf$ the one-object groupoid associated to a group $\Gsf$.

The second main result has to do with the homotopy invariants of the permutation 2-groups of finite type. Given any group $\Gsf'$ and any $\Gsf'$-module $\Asf$ there is a canonical homomorphism
$$
\xi_n:\Hsf^\bullet(\Gsf',\Asf)\To\Hsf^\bullet(\Ssf_n\wr\Gsf',\Asf^n)
$$
for each $n\geq 1$. When $\Gsf'$ is the group $\mathsf{Out}(\Gsf)$ of outer automorphisms of some group $\Gsf$, and $\Asf$ is the center $\Zsf(\Gsf)$ viewed as a $\mathsf{Out}(\Gsf)$-module in the obvious way, this gives a canonical homomorphism
$$
\xi_n:\Hsf^\bullet(\mathsf{Out}(\Gsf),\Zsf(\Gsf))\To\Hsf^\bullet(\Ssf_n\wr\mathsf{Out}(\Gsf),\Zsf(\Gsf)^n)
$$
for each $n\geq 1$. On the other hand, let be given any family of groups $\{\Gsf_i\}_{i\in I}$ and a family of abelian groups $\{\Asf_i\}_{i\in I}$ with $\Gsf_i$ acting on $\Asf_i$, and let
$$
\zeta:\prod_{i\in I}\Hsf^\bullet(\Gsf_i,\Asf_i)\To\Hsf^\bullet(\prod_{i\in I}\Gsf_i,\prod_{i\in I}\Asf_i)
$$
be the canonical monomorphism mapping $([z_i])_{i\in I}$ to the cohomology class of the product cocycle $\prod_{i\in I}z_i$. Then the second main result may be stated as follows (cf. Theorem~\ref{invariants_homotopics_2-grup_permutacions_general} in the main text).

\medskip
\noindent
{\sc Theorem.} {\em 
Let $\{(n_i,\Gsf_i)\}_{i\in I}$ be and $\Gg=\coprod_{i\in i}\Gg_i$ be as before. Then:
\begin{itemize}
\item[(1)] $\pi_0(\SSS ym(\Gg))\cong\prod_{i\in I}\Ssf_{n_i}\wr\mathsf{Out}(\Gsf_i)$.
\item[(2)] $\pi_1(\SSS ym(\Gg))\cong\prod_{i\in I}\Zsf(\Gsf_i)^{n_i}$, and the $\pi_0(\SSS ym(\Gg))$-module structure is given componentwise by
$$
(\sigma,([\phi_1],\ldots,[\phi_n]))\lhd(z_1,\ldots,z_n)=(\phi_{\sigma^{-1}(1)}(z_{\sigma^{-1}(1)}),\ldots,\phi_{\sigma^{-1}(n)}(z_{\sigma^{-1}(n)})).
$$
\item[(3)] $\alpha(\SSS ym(\Gg))=\zeta(\{\xi_{n_i}(\alpha_i))\}_{i\in I})$, where $\alpha_i=\alpha(\SSS ym(\Gsf_i))$ for $i\in I$. Moreover, $\SSS ym(\Gg)$ is split if and only if $\SSS ym(\Gsf_i)$ is split for all $i\in I$. 
\end{itemize} }

\medskip
Finally, the third main result has to do with the split character of these permutation 2-groups. According to the previous result, non-splitness can only arise in the step from the trivial groupoid to an arbitrary one-object groupoid. A group $\Gsf$ will be called {\em permutationally split} when $\SSS ym(\Gsf)$ is split. It then follows from the previous Theorem that $\SSS ym(\Gg)$ is split if and only if the corresponding groups $\Gsf_i$ are permutationally split for all $i\in I$.  It will be shown that the permutationally split character of a group $\Gsf$ is related to the canonical exact sequence
$$
\xymatrix{
0\ar[r] & \mathsf{Z}(\Gsf)\ar[r] & \Gsf\ar[r]^{c\ \ \ \ } & \mathsf{Aut}(\Gsf)\ar[r]^{\pi\  } & \mathsf{Out}(\Gsf)\ar[r] & 1\ , }
$$
where $c$ is the map sending any element of $\Gsf$ to the corresponding inner automorphism. More precisely, $\Gsf$ turns out to be permutationally split if and only if there exists a set theoretic section of the projection $\pi$ satisfying an appropriate condition of a cocycle type (cf. Proposition~\ref{caracteritzacio_grup_permutacionalment_escindit} in the main text). In particular, all abelian groups, all centerless groups and all groups with only inner automorphisms are permutationally split. However, the last main result of the paper can be stated as follows.

\medskip
\noindent
{\sc Theorem.} {\em 
There exists permutation 2-groups which are non-split.}

\medskip
\noindent
More precisely, it will be shown that the dihedral groups $\Dsf_{8k}$ are non permutationally split for any $k\geq 1$ (cf. Proposition~\ref{2-grups_permutacions_no_escindits} in the main text). Hence the permutation 2-groups $\SSS ym(\Dsf_{8k})$ are non-split.

An important example of a permutation 2-group which can be explicitly computed is that of the underlying groupoid of a finite type 2-group, i.e. a 2-group $\GG$ such that $|\pi_0(\GG)|<\infty$. It will be shown that is is equivalent to a split 2-group of the form
$$
\Asf^n[1]\rtimes(\Ssf_n\wr\mathsf{Aut}(\Asf))[0]
$$
for some positive integer $n\geq 1$ and some abelian group $\Asf$ (cf. Proposition~\ref{2-grup_permutacions_grupoide_subjacent_2-grup} below). In a sequel to this paper we plan to discuss the analog for 2-groups of Cayley's theorem and the extent to which any finite type 2-group, split or not, can be seen as a sub-2-group of such a split 2-group. 

As said before, this work is just a first step to the general theory of permutation $\infty$-groups. It remains for a future work investigating the possible analogues of all of these results in the general setting. One reason by which this may be interesting has to be found in the above mentioned homotopy hypothesis. In the same way as $\infty$-groupoids provide purely algebraic descriptions of topological spaces through the associated fundamental $\infty$-groupoids, permutation $\infty$-groups should provide purely algebraic descriptions of the automorphism groups of topological spaces. Apart from this, investigating permutation 2-groups also constitutes a first step toward a categorification of the theory of permutation representation of groups or equivalently, of the theory of $\Gsf$-sets. This was actually my original motivation. 

Let us now briefly describe how the text is organized. Section~2 begins with a review of the basic facts about 2-groups needed in the sequel, so as to make the paper as self-contained as possible. Next, both the notion of 2-action of a 2-group on another one and the associated notion of semidirect product of 2-groups are discussed. In particular, the wreath 2-product of the symmetric groups with an arbitrary 2-group is introduced. Section~3 is devoted to the notion of split 2-group and to proving various splitness criteria. In particular, a necessary and sufficient condition for a {\em strict} 2-group to be split is shown which is used in the next section to prove the existence of non-split permutation 2-groups. Although some parts of this section may be known to experts, to my knowledge there exists no detailed discussion of this in the literature. Finally, in Section~4 the (finite type) permutation 2-groups are introduced and the result sketched above proved.      

\bigskip
\noindent
{\bf Notation and terminology.} By a 2-category I always mean a strict one. In fact, there appear no weak 2-categories in this work other than the monoidal categories (the one-object weak 2-categories). By a (2,1)-category I shall mean a category enriched over the category of groupoids, i.e. a (strict) 2-category such that all 2-morphisms are invertible. 

Categories are denoted by $\Aa,\Bb,\Cc,...$, and 2-categories by $\mathbf{A},\mathbf{B},\mathbf{C},...$. In particular, $\mathbf{Cat}$ and $\mathbf{Gpd}$ respectively denote the 2-category of (small)categories, functors and natural transformations, and the (2,1)-category of (small) groupoids, functors and natural transformations. The correspoding underlying categories are dentoed by $\Cc at$ and $\Gg pd$. 

The dual of a category $\Aa$ is denoted $\Aa^\vee$, and the set of morphisms between two objects $x,y$ in $\Aa$ is denoted by $Hom_\Aa(x,y)$, or just $Hom(x,y)$ when the category is clear from the context. $x\in\Aa$ means that $x$ is an object of $\Aa$. Similarly, the category of morphisms between two objects $x,y$ in a 2-category $\Abf$ is denoted by $\Hh om_\Abf(x,y)$ or just $\Hh om(x,y)$, and $x\in\Abf$ means that $x$ is an object of $\Abf$. 

Monoids (in particular, groups) with underlying sets $M,N,\ldots$ are denoted by $\mathsf{M},\mathsf{N},\ldots$ For instance, for any category $\Cc$ and any object $x\in\Cc$, $\mathsf{End}(x)$ denotes the endomorphism monoid $(End(x),\circ,id_x)$ and $\mathsf{Aut}(x)$ the automorphism group $(Aut(x),\circ,id_x)$. By the product of two endo- or automorphisms it is meant $\phi\cdot\phi'=\phi\circ\phi'$ in this order.

\section{\large Background of 2-groups}

We assume the reader is familiar with the notion of monoid in the more general setting of (non-discrete) categories, usually called a monoidal category, and with the notion of 2-category, its many objects version. However, to set up the notation we begin by recalling the definition of monoidal category and a few related notions. We next define 2-groups as a special class of monoidal categories, and describe the corresponding 2-category of 2-groups. The most basic facts about 2-groups are reviewed, in particular, the homotopy invariants which classify them up to equivalence. Finally, the notion of action of a 2-group on another 2-group and the associated notion of semidirect product of 2-groups are recalled. As an example, we introduce the {\em wreath 2-product} of the symmetric groups with an arbitrary 2-group.

\subsection{Monoidal categories}

For general references on monoidal categories and their many objects version, the 2-categories, we refer the reader to \cite{sM98}, \cite{fB94-1+2}, \cite{cK95}.

A {\em monoidal category} is a category $\Mm$ together with a functor $\otimes:\Mm\times\Mm\To\Mm$ (called the {\em tensor product}), a distinguished object $e\in\Mm$ (called the {\em unit object}), and natural isomorphisms $a,l,r$ (respectively called the {\it associator} and the {\it left} and {\it right unitors}) whose components
\begin{align*}
a_{x,y,z}&:x\otimes(y\otimes z)\stackrel{\cong}{\longrightarrow}(x\otimes y)\otimes z, \\ l_x&:e\otimes x\stackrel{\cong}{\longrightarrow}x,  \\ r_x&:x\otimes e\stackrel{\cong}{\longrightarrow}x
\end{align*}
satisfy the usual pentagon and triangle coherence conditions. The data $(\otimes,e,a,l,r)$ is called the {\em monoidal structure} on $\Mm$, and the category equipped with such a monoidal structure will be denoted by $\MM$.

A monoidal category $\MM$ is called {\em strict} when all the isomorphisms $a,l,r$ are identities (in particular, $\otimes$ is strictly associative on objects and $e$ is a strict unit for the tensor product). The simplest examples are the monoids $\Msf$ viewed as discrete categories, with $\otimes$ the product of $\Msf$ and $e$ the unit. The corresponding strict monoidal categories are denoted $\Msf[0]$. 

In all monoidal categories we have $l_e=r_e$. This is an easy consequence of the axioms and the naturality of the associator and left and right unitors (see for instance \cite{cK95}). This isomorphism will be denoted by $d:e\otimes e\To e$, and called the {\em unit isomorphism} of $\MM$.

In this work we are mostly interested in monoidal categories whose underlying category is a groupoid $\Gg$ (i.e. a category such that all morphisms are isomorphisms). They are called {\em monoidal groupoids} and denoted by $\GG$. 

\subsection{2-groups}

For a general reference on 2-groups, we refer the reader to \cite{BL03} and the references therein.

\begin{defn}
A {\em 2-group}, also called a {\em categorical group}, a {\em cat-group} or a {\em gr-category}, is a monoidal groupoid $\GG=(\Gg,\otimes,e,a,l,r)$ such that each object $x$ has a two-sided (generically weak) inverse for the tensor product, i.e. an object $x^\ast$ such that $x\otimes x^\ast$ and $x^\ast\otimes x$ are both isomorphic to $e$. A 2-group is called {\em strict} if the underlying monoidal groupoid is strict (i.e., if $a,l,r$ are identities), and each object $x$ has a (two-sided) strict inverse, i.e. an object $x^{-1}$ such that $x\otimes x^{-1}$ and $x^{-1}\otimes x$ are both equal to $e$.
\end{defn}
\noindent
Let us remark that the first condition of strict 2-group does not imply the second (see below).

Strict 2-groups are often described in terms of crossed modules of groups. However, we shall make no use of such a description because it lacks the desired invariance. See Remark~\ref{remarca_2Grp_s} below.

We shall denote by $G_0$ the set of objects, and by $G_1$ the set of morphisms of the underlying groupoid $\Gg$ of a 2-group $\GG$. In the strict case, $G_0$ is a group $\Gsf_0$ with the product given by the tensor product and with the unit object $e$ as unit, and $G_1$ is a group $\Gsf_1$ with the product given by the tensor product and with $id_e$ as unit.

Any quadruple $(x,x^\ast,\eta_x,\epsilon_x)$, with $\eta_x:e\To x\otimes x^\ast$ and $\epsilon_x:x^\ast\otimes x\To e$ isomorphisms satisfying the usual triangle axioms of an adjunction is called a {\em duality}. The choice of such a duality for each $x\in\Gg$ induces an equivalence $\Dd:\Gg^\vee\To\Gg$ mapping $x$ to its dual $x^\ast$, and a morphism $f:x\To y$ to its dual morphism $f^\ast:y^\ast\To x^\ast$ defined by the composite
$$
y^\ast\stackrel{\cong}{\longrightarrow}y^\ast\otimes e\stackrel{id\otimes\eta_x}{\longrightarrow}y^\ast\otimes(x\otimes x^\ast)\stackrel{id\otimes(f\otimes id)}{\longrightarrow}y^\ast\otimes(y\otimes x^\ast)\stackrel{\cong}{\longrightarrow}(y^\ast\otimes y)\otimes x^\ast\stackrel{\epsilon_y\otimes id}{\longrightarrow}e\otimes x^\ast\stackrel{\cong}{\longrightarrow}x^\ast.
$$ 
By precomposing it with the canonical equivalence $\Gg\simeq\Gg^\vee$ we get a self-equivalence $\Ii:\Gg\To\Gg$, called the {\em functor of inverses} of $\GG$.

In general, inverses in a 2-group are unique only up to isomorphism so that there are various functors of inverses, all of them isomorphic. In fact, for any two dualities $(x,x^\ast,\eta_x,\epsilon_x)$ and $(x,\hat{x}^\ast,\hat{\eta}_x,\hat{\epsilon}_x)$ there is a canonical natural isomorphism $x^\ast\cong\hat{x}^\ast$. By contrast, when they exist, strict inverses are always unique. Hence for strict 2-groups there is a canonical $\Ii$, obtained when all dualities are chosen to be trivial ($\eta_x=\epsilon_x=id_e$ for all $x\in\Gg$). The $\otimes$-inverse of a morphism $f:x\To y$ is then the morphism $\Ii(f)$, i.e. $id_{x^\ast}\otimes f^{-1}\otimes id_{y^\ast}$.

Notice that in an arbitrary 2-group all translation functors $-\otimes x,x\otimes -:\Gg\To\Gg$ are self-equivalences of $\Gg$, with respective pseudoinverses $-\otimes x^\ast$ and $x^\ast\otimes -$.

A 2-group $\GG$ is called {\em skeletal} if the underlying groupoid $\Gg$ is skeletal (i.e. if isomorphic objects are necessarily equal). In this case the tensor product is strictly associative on objects, the unit object is a strict unit and inverses are strict and unique. However, it is worth emphasizing that a skeletal 2-group need not be strict. The associator and left and right unitors may be non-trivial.

Among the simplest examples of 2-groups we have the following two families, both of strict and skeletal 2-groups. 

\begin{ex}{\rm
Any group $\Gsf$ is a 2-group when thought of as a discrete monoidal groupoid with the tensor product given by the group law. Such 2-groups are called {\em discrete 2-groups} and denoted by $\Gsf[0]$. When $\Gsf=\mathsf{1}$ we shall just write $\mathbbm{1}$, and call it the {\em trivial 2-group}.}
\end{ex}

\begin{ex}{\rm
Abelian groups $\Asf$ can be thought of as discrete monoidal groupoids but also as one-object monoidal groupoids. In the last case, the composition and the tensor product of morphisms are both given by the group law. The group needs to be abelian because of the functoriality of $\otimes$. As such, abelian groups are called {\em one-object} (or {\em connected}) {\em 2-groups}, and they are denoted by $\Asf[1]$.}
\end{ex}

In the same way as groups arise as groups of automorphism of objects in a category, 2-groups naturally appear as 2-groups of {\it self-equivalences} of objects in a 2-category (i.e. endomorphisms of the object which are invertible only up to a 2-isomorphism). Precisely, for any 2-category $\mathbf{C}$ and any object $x\in\Cbf$ the corresponding 2-group of self-equivalences is the monoidal groupoid
$$
\mathbb{E}quiv_\Cbf(x):=(\Ee quiv_\Cbf(x),\circ,id_x,a,l,r).
$$
$\Ee quiv_\Cbf(x)$ is the groupoid of self-equivalences of $x$ and 2-isomorphisms between them, $\circ$ is the appropriate restriction of the composition functor of $\Cbf$, and $a,l,r$ are given by the appropriate components of the associator $\alpha$ and left and right unitors $\lambda,\rho$ of $\Cbf$. When the 2-category $\Cbf$ is clear from the context we shall write $\EE quiv(x)$ and $\Ee quiv(x)$.

\begin{defn}\label{2-grup_permutacions_grupoide}
Let $\Gg$ be any groupoid. The {\em 2-group of permutations of $\Gg$}, denoted by $\SSS ym(\Gg)$, is the 2-group $\EE quiv_\mathbf{Gpd}(\Gg)$. The underlying groupoid will be denoted by $\Ss ym(\Gg)$.
\end{defn}
In fact, this paper is all about this kind of 2-groups. The simplest case is when $\Gg$ is a one-object groupoid $\Gsf$, i.e. a group. The corresponding 2-group of permutations has the following explicit description.
 
\begin{prop}\label{2-grup_autoequivalencies_grupoide}
Let $\Gsf$ be any group. Then $\SSS ym(\Gsf)$ is the strict 2-group given as follows:
\begin{itemize}
\item[(a)] The objects are in bijection with the automorphisms of $\Gsf$.  We shall denote by $\Ee(\phi)$ the object associated to the automorphism $\phi\in Aut(\Gsf)$.
\item[(b)] For any $\phi,\tilde{\phi}\in Aut(\Gsf)$ a morphism from $\Ee(\phi)$ to $\Ee(\tilde{\phi})$ is given by any $g\in G$ such that $\tilde{\phi}=c_g\circ\phi$, where $c_g$ denotes conjugation by $g$. We shall denote by 
$$
\tau(g;\phi,\tilde{\phi}):\Ee(\phi)\Rightarrow\Ee(\tilde{\phi}):\Gsf\To\Gsf
$$
the morphism defined by such a $g\in G$. It is the natural isomorphism whose unique component is given by $\tau(g;\phi,\tilde{\phi})_\ast=g$.
\item[(c)] The composition and tensor product are given by
\begin{align*}
\Ee(\phi)\otimes\Ee(\phi')&=\Ee(\phi\circ\phi'), \\ \tau(\tilde{g};\tilde{\phi},\tilde{\tilde{\phi}})\circ\tau(g;\phi,\tilde{\phi})&=\tau(\tilde{g}g;\phi,\tilde{\tilde{\phi}}), \\ 
\tau(g;\phi,\tilde{\phi})\otimes\tau(g';\phi',\tilde{\phi}')&=\tau(g\ \phi(g');\phi\circ\phi',\tilde{\phi}\circ\tilde{\phi}').
\end{align*}
\end{itemize}
\end{prop}
\begin{proof}
Left to the reader. 
\end{proof}

\bigskip
In general, $\EE quiv_\Cbf(x)$ is non strict even if $\Cbf$ is strict. Thus the underlying monoidal groupoid is strict when $\Cbf$ is strict. However, self-equivalences are invertible only to within 2-isomorphism in general. A strict 2-group for a strict $\Cbf$ is given by the 2-group $\AAA ut_\Cbf(x)$ of automorphisms of $x$, i.e. strictly invertible self-equivalences of $x$. Now $\Aa ut_\Cbf(x)$ is usually a proper subgroupoid of $\Ee quiv_\Cbf(x)$. However, the homotopically right ``2-group of symmetries'' of $x$ is $\EE quiv_\Cbf(x)$, not $\AAA ut_\Cbf(x)$. Indeed, equivalent objects in $\Cbf$ have equivalent 2-groups of self-equivalences (see Example~\ref{equivalencia_Equiv(X)_Equiv(Y)} below), but they may have non-equivalent 2-groups of automorphisms. 

Let us finally remind the reader that there are two possible ways of defining the opposite $\MM^{op}$ of an arbitrary monoidal category $\MM$. They differ in the underlying category, which may be either $\Mm$ or its dual category $\Mm^\vee$. However, for monoidal groupoids (in particular, for 2-groups) both versions are equivalent because groupoids are self-dual. To avoid working with the dual groupoid, we shall take as opposite $\GG^{op}$ of a 2-group $\GG$ the 2-group given by
$$
\GG^{op}=(\Gg,\otimes^{rev},e,a^{rev},r,l),
$$
where $\otimes^{rev}:\Gg\times\Gg\To\Gg$ denotes the functor given on objects by $(x,y)\mapsto y\otimes x$ and similarly on morphisms, and $a^{rev}$ is the natural isomorphism whose $(x,y,z)$-component is $(a_{z,y,x})^{-1}$. Notice that for discrete 2-groups this gives $(\Gsf[0])^{op}=\Gsf^{op}[0]$.

\subsection{The (2,1)-category of 2-groups}
\label{2-categoria_2Grp}

Given two monoidal categories $\MM$ and $\MM'$, a {\em monoidal functor} between them is a triple $\FF=(\Ff,\mu,\nu)$ with $\Ff:\Mm\To\Mm'$ a functor, $\mu$ a family of isomorphisms $\mu_{x,y}:\Ff x\otimes\Ff y\To\Ff(x\otimes y)$ natural in $x,y$, and $\nu:e'\To\Ff e$ an isomorphism such that
\begin{equation} \label{axioma_coherencia}
\xymatrix{
\Ff x\otimes'(\Ff y\otimes'\Ff z)\ar[d]_{id\otimes'\mu_{y,z}}\ar[rr]^{a'_{\Ff x,\Ff y,\Ff z}\ \ \ } && (\Ff x\otimes'\Ff y)\otimes'\Ff z\ar[d]^{\mu_{x,y}\otimes' id}
   \\ \Ff x\otimes'\Ff(y\otimes z)\ar[d]_{\mu_{x,y\otimes z}} &&  \Ff(x\otimes y)\otimes'\Ff z\ar[d]^{\mu_{x\otimes y,z}} \\ \Ff(x\otimes(y\otimes z))\ar[rr]_{\Ff(a_{x,y,z})} &&\Ff((x\otimes y)\otimes z)
}
\end{equation}
\begin{equation} \label{coherencia_iota}
\xymatrix{
(\Ff x)\otimes' e'\ar[d]_{id\otimes'\nu}\ar[r]^{r'_{\Ff x}} & \Ff x \\ \Ff x\otimes'\Ff e\ar[r]_{\mu_{x,e}} & \Ff(x\otimes e)\ar[u]_{\Ff r_x} }\qquad \xymatrix{
e'\otimes'(\Ff x)\ar[d]_{\nu\otimes' id}\ar[r]^{l'_{\Ff x}} & \Ff x \\ \Ff e\otimes'\Ff x\ar[r]_{\mu_{e,x}} & \Ff(e\otimes x)\ar[u]_{\Ff l_x}\  }
\end{equation}
commute for all $x,y,z\in\Mm$. We shall refer to the pair $(\mu,\nu)$ as the {\em monoidal structure} on $\Ff$, and to $\nu$ as the {\em unit isomorphism}.

\begin{defn}
For any 2-groups $\GG$ and $\GG'$, a {\em morphism of 2-groups} between them is a monoidal functor $\FF=(\Ff,\mu,\nu)$ between the underlying monoidal groupoids. It will be called a {\em normalized morphism} when the unit isomorphism $\nu$ is an identity (in particular, the unit object is strictly preserved), and a {\em strict morphism} when both $\mu$ and $\nu$ are identities.
\end{defn}

In fact, in the case of 2-groups the unit isomorphism turns out to be redundant. More precisely, we have the following.

\begin{prop}\label{existencia_iso_unitat}
Let $\GG,\GG'$ be any 2-groups, and let be given a functor $\Ff:\Gg\To\Gg'$ and natural isomorphisms $\mu_{x,y}:\Ff x\otimes'\Ff y\To\Ff(x\otimes y)$ such that (\ref{axioma_coherencia}) commutes for all $x,y,z\in\Gg$. Then there exists a unique isomorphism $\nu:e'\To\Ff e$ making (\ref{coherencia_iota}) commute for all $x\in\Gg$.
\end{prop}
\begin{proof}
All translations functors in a 2-group are equivalences, in particular fully faithful. Hence any functor $\Ff:\Gg\To\Gg'$ equipped with an isomorphism $\mu_{e,e}:\Ff e\otimes'\Ff e\To\Ff(e\otimes e)$ automatically preserves the unit objects up to a canonical isomorphism $\nu:e'\To\Ff e$ uniquely determined by the commutativity of any of the diagrams
$$
\xymatrix{
(\Ff e)\otimes' e'\ar[d]_{id\otimes'\nu}\ar[r]^{r'_{\Ff e}} & \Ff e \\ \Ff e\otimes'\Ff e\ar[r]_{\mu_{e,e}} & \Ff(e\otimes e)\ar[u]_{\Ff d} }\qquad \xymatrix{
e'\otimes'(\Ff e)\ar[d]_{\nu\otimes' id}\ar[r]^{l'_{\Ff e}} & \Ff e \\ \Ff e\otimes'\Ff e\ar[r]_{\mu_{e,e}} & \Ff(e\otimes e)\ar[u]_{\Ff d}\ . }
$$
We leave to the reader checking that this isomorphism also makes (\ref{coherencia_iota}) commute.
\end{proof}
\noindent
Notice that the unit isomorphism $\nu$ can be non trivial even if $\mu_{e,e}$ is an identity. 

Because of Proposition~\ref{existencia_iso_unitat}, henceforth we shall refer to any pair $(\Ff,\mu)$ as before also as a morphism of 2-groups, and we shall write $\FF=(\Ff,\mu)$. It further follows that any morphism of 2-group $\FF:\GG\To\GG'$ automatically preserves inverses. The isomorphisms $\Ff(x^\ast)\cong(\Ff x)^\ast$ are canonically given once the dualities in $\GG'$ have been fixed.

\begin{ex}\label{morfismes_entre_2-grups_discrets}{\rm
For any groups $\Gsf,\Gsf'$ a morphism of 2-groups from $\Gsf[0]$ to $\Gsf'[0]$ is the same thing as a group homomorphism $f:\Gsf\To\Gsf'$. The associated morphism of 2-groups will be denoted by $f[0]:\Gsf[0]\To\Gsf'[0]$. }
\end{ex}

\begin{ex}\label{morfismes_entre_2-grups_un_objecte}{\rm
For any abelian groups $\Asf,\Asf'$ a morphism of 2-groups from $\Asf[1]$ to $\Asf'[1]$ just amounts to a pair $(f,a')$, with $f:\Asf\To\Asf'$ a group homomorphism and $a'$ any element of $A'$. The homomorphism $f$ gives the underlying functor, which we shall denote $f[1]:\Asf[1]\To\Asf'[1]$, and $a'$ gives the monoidal structure. }
\end{ex}

\begin{ex}\label{equivalencia_induida_entre_2-grups_autoequivalencies}{\rm
Let $x,y$ be two equivalent objects in a 2-category $\Cbf$, and let $(h,h^\ast,\eta,\epsilon)$ be any adjoint equivalence, with $h:x\To y$, $h^\ast:y\To x$ morphisms in $\Cbf$, and $\eta:id_x\Rightarrow h^\ast\circ h$ and $\epsilon:h\circ h^\ast\Rightarrow id_y$ 2-isomorphisms satisfying the usual triangle axioms of an adjunction. Then there is an induced morphism of 2-groups
$$
\FF(h,h^\ast,\eta,\epsilon):\EE quiv(x)\To\EE quiv(y)
$$
mapping any self-equivalence $f$ of $x$ to $h\circ f\circ h^\ast$, and any 2-isomorphism $\tau$ to $1_{h}\circ\tau\circ 1_{h^{\ast}}$. The monoidal structure
$$
\mu_{f,f'}:h\circ f\circ h^\ast\circ h\circ f'\circ h^\ast\Rightarrow h\circ f\circ f'\circ h^\ast
$$
is given in the obvious way by the inverse of $\eta$. }
\end{ex}

\begin{ex}\label{functor_monoidal_I}{\rm
For any 2-group $\GG$, any functor of inverses $\Ii:\Gg\To\Gg$ canonically extends to a morphism of 2-groups $\II:\GG\To\GG^{op}$. The monoidal structure is given by the isomorphisms $\mu_{x,y}:y^\ast\otimes x^\ast\To(x\otimes y)^\ast$ canonically induced by the dualities chosen to define $\Ii$. }
\end{ex}  

\begin{defn}
Given two morphisms of 2-groups $\FF_1,\FF_2:\GG\To\GG'$, a {\em 2-morphism} between them it a natural transformation $\tau:\Ff_1\Rightarrow\Ff_2$ such that
$$
\xymatrix{ 
\Ff_1 x\otimes'\Ff_1 y\ar[r]^{\mu_{1;x,y}}\ar[d]_{\tau_x\otimes'\tau_y} & \Ff_1(x\otimes y)\ar[d]^{\tau_{xy}} \\ \Ff_2x\otimes'\Ff_2y\ar[r]_{\mu_{2;x,y}} & \Ff_2(x\otimes y) }
$$
commutes for all objects $x,y\in\Gg$.
\end{defn}
\noindent
If $\nu_1,\nu_2$ are the corresponding unit isomorphisms of $\FF_1,\FF_2$, it is shown that the diagram
\begin{equation}\label{monoidalitat_nu}
\xymatrix{ e'\ar[r]^{\nu_2}\ar[d]_{\nu_1} & \Ff_2 e \\ \Ff_1 e\ar[ru]_{\tau_e} & }
\end{equation}
also commutes. Hence a 2-morphism of 2-groups is nothing but a monoidal natural transformation between the corresponding monoidal functors. Notice that all 2-morphisms $\tau$ are invertible because the underlying category of $\GG'$ is a groupoid. When such a $\tau$ exists, $\FF_1$ and $\FF_2$ are said to be 2-isomorphic.

\begin{ex}\label{Equiv(G[0])}{\rm
For any groups $\Gsf,\Gsf'$ and any group homomorphisms $f_1,f_2:\Gsf\To\Gsf'$ there is no 2-(iso)morphism $f_1[0]\Rightarrow f_2[0]:\Gsf[0]\To\Gsf'[0]$ unless $f_1=f_2$, and in this case the 2-morphism is necessarily the identity. }
\end{ex}

\begin{ex}\label{Equiv(A[1])}{\rm
For any abelian groups $\Asf,\Asf'$, any group homomorphisms $f_1,f_2:\Asf\To\Asf'$ and any elements $a'_1,a'_2\in\Asf'$ there exists a 2-(iso)morphism $(f_1[1],a'_1)\Rightarrow(f_2[1],a'_2):\Asf[1]\To\Asf'[1]$ if and only if $f_1=f_2$, and in this case the morphism is also unique (but not an identity unless $a'_1=a'_2$).}
\end{ex}

Not every morphism of 2-groups is 2-isomorphic to a strict morphism. In fact, this is so in the more general context of monoidal categories and even for homomorphisms between bicategories (see \cite{BKP89}). However, it is easy to check that any morphism of 2-groups is at least 2-isomorphic to a normalized one.

Morphisms of 2-groups and 2-morphisms between them compose in the obvious way, and 2-groups are thus the objects of a (2,1)-category $\mathbf{2Grp}$. The corresponding hom-groupoids will be denoted by $\Hh om_\mathbf{2Grp}(\GG,\GG')$ or just $\Hh om(\GG,\GG')$, and the underlying category by $2\Gg rp$.

As in any 2-category, we may speak of {\em equivalent} 2-groups, i.e. 2-groups $\GG$ and $\GG'$ for which there exist morphisms $\FF:\GG\To\GG'$ and $\FF':\GG'\To\GG$, called {\em equivalences}, such that $\FF'\circ\FF$ and $\FF\circ\FF'$ are 2-isomorphic to the respective identities. In this case, we shall write $\GG\simeq \GG'$. This is weaker than the condition that $\GG$ and $\GG'$ be isomorphic as objects in $2\Gg rp$, in which case we shall write $\GG\cong\GG'$.

\begin{ex}\label{equivalencia_Equiv(X)_Equiv(Y)}{\rm
If $x,y$ are any equivalent objects in a 2-category $\Cbf$ the morphism of 2-groups $\FF(h,h^\ast,\eta,\epsilon):\EE quiv_\Cbf(x)\To\EE quiv_\Cbf(y)$ defined in Example~\ref{equivalencia_induida_entre_2-grups_autoequivalencies} is an equivalence. A pseudoinverse $\FF'(h,h^\ast,\eta,\epsilon):\EE quiv_\Cbf(y)\To\EE quiv_\Cbf(x)$ is given on objects by $g\mapsto h\circ(g\circ h^\ast)$. Thus equivalent objects in a 2-category always have equivalent 2-groups of self-equivalences.  }
\end{ex}

The canonical isomorphism $\Gsf\cong\Gsf^{op}$ for any group $\Gsf$ has the following generalization to arbitrary 2-groups. Later on, it allows us to equally think of left or right 2-actions of a 2-group on another 2-group.

\begin{prop}\label{equivalencia_G-op_G}
Any 2-group $\GG$ is equivalent to its opposite 2-group $\GG^{op}$ but non canonically.
\end{prop}
\begin{proof}
An equivalence is given by any of the morphisms of 2-groups $\II:\GG\To\GG^{op}$ of Example~\ref{functor_monoidal_I}. It is canonical only to within 2-isomorphism, in that it depends on the choice of a duality $(x,x^\ast,\eta_x,\epsilon_x)$ for each object $x\in\Gg$.
\end{proof}
Notice that this is false for arbitrary monoidal groupoids. In fact, it is already false in the discrete setting of monoids.

Equivalent 2-groups are also called {\em homotopic} because 2-groups provide an algebraic model of the pointed connected homotopy 2-types. Thus the theory of 2-groups has a topological counterpart in the theory of homotopy 2-types. Because of that, we shall often refer to the invariants of a 2-group up to equivalence as {\em homotopy invariants}.

\begin{rem}\label{remarca_2Grp_s}{\rm
Any 2-group is equivalent to a strict 2-group. In fact, this is true for arbitrary monoidal categories, and it is a consequence of MacLane's coherence theorem. However, the {\em locally full} sub-2-category $\mathbf{2Grp}_{str}$ of $\mathbf{2Grp}$ with objects the strict 2-groups and morphisms the strict ones is not a homotopically right version of $\mathbf{2Grp}$ because there are morphisms of strict 2-groups which are not 2-isomorphic to a strict one. Moreover, strict 2-groups together with the strict morphisms between them define a category $2\Gg rp_s$ equivalent to the category $\Cc ross\Mm od$ of crossed modules of groups and the morphisms between these (see \cite{BS76}). However, working with the more elementary category $\Cc ross\Mm od$ to study 2-groups is an evil option by the same reason.
}
\end{rem}

Being the objects in a 2-category, each 2-group $\GG$ has its own 2-group of self-equivalences $\EE quiv_\mathbf{2Grp}(\GG)$. This is to be distinguished from the 2-group $\SSS ym(\Gg)$ of self-equivalences of the underlying groupoid of $\GG$. There is a canonical strict morphism of 2-groups
$$
\UU_\GG=(\Uu_\GG,id):\EE quiv_\mathbf{2Grp}(\GG)\To\SSS ym(\Gg)
$$
mapping the self-equivalence $\EE=(\Ee,\mu)$ of $\GG$ to the self-equivalence $\Ee$ of $\Gg$, and any 2-morphism $\tau$ to itself. The underlying functor $\Uu_\GG$ is faithful, but non full in general.

\begin{ex}\label{autoequivalencies_G[0]_i_A[1]}{\rm
For any group $\Gsf$ we have
\begin{equation*}
\EE quiv_\mathbf{2Grp}(\Gsf[0])\cong\mathsf{Aut}(\Gsf)[0],
\end{equation*}
while $\SSS ym(G[0])$ is isomorphic to $\mathsf{Aut}(G)[0]$ (recall that $G$ stands for the underlying set of $\Gsf$). The canonical morphism $\UU_{\Gsf[0]}$ is nothing but the inclusion of groups $\mathsf{Aut}(\Gsf)\hookrightarrow\mathsf{Aut}(G)$. Similarly, for any abelian group $\Asf$ we have
\begin{equation*}
\EE quiv_\mathbf{2Grp}(\Asf[1])\simeq\mathsf{Aut}(\Asf)[0]
\end{equation*}
(now we have just an equivalence). By contrast, we shall see in Corollary~\ref{Sim(A)} below that the 2-group $\SSS ym(\Asf)$ is something more involved. In particular, it is not equivalent to a discrete discrete 2-group. However, it includes a copy of $\mathsf{Aut}(\Asf)[0]$. Hence the canonical morphism $\UU_{\Asf[1]}$ still is a sort of inclusion.}
\end{ex}  

Let us finally mention that the 2-category $\mathbf{2Grp}$ admits all PIE-limits, that is, all 2-limits~\footnote{The prefix 2- is used to emphasize that we are thinking of the two-dimensional notion of limit, in which the required universal property involves the {\em hom-categories} of morphisms; see \cite{mK89}. } constructible from 2-products, inserters, and equifiers (a simple characterisation of PIE-limits can be found in \cite{PR91}). Indeed, $\mathbf{2Grp}$ is equivalent to the 2-category $\mathbf{T}${\bf -}$\mathbf{Alg}$ of $\mathbf{T}$-algebras and weak morphisms between them for a certain 2-monad $\mathbf{T}$ on the 2-category of groupoids (see \cite{cSP11}, Corollary~60). Then the claim follows from the general theorem of Blackwell, Kelly and Power \cite{BKP89} according to which for any 2-monad $\mathbf{T}$ on $\mathbf{Gpd}$ the 2-category $\mathbf{T}${\bf -}$\mathbf{Alg}$ admits all PIE-limits.

However, apart from just mentioning cotensor products (see Remarks~\ref{accio_wreath_general} and \ref{wreath_2-product_general} below), in this work we only make use of 2-products. They are called {\em products} because they coincide with the products in the underlying category $2\Gg rp$. Explicitly, for any family of 2-groups $\{\GG_i\}_{i\in I}$, with $\GG_i=(\Gg_i,\otimes_i,e_i,a^{(i)},l^{(i)},r^{(i)})$, the {\em product 2-group} $\prod_{i\in I}\GG_i$ is the product groupoid $\prod_{i\in I}\Gg_i$ equipped with the monoidal structure defined componentwise. Thus the tensor product $\otimes$ is defined by the commutative diagram
$$
\xymatrix{
\left(\prod_{i\in I}\Gg_i\right)\times\left(\prod_{i\in I}\Gg_i\right)\ar[rr]^\otimes\ar[rd]_{\cong} & & \prod_{i\in I}\Gg_i \\ & \prod_{i\in I}(\Gg_i\times\Gg_i)\ar[ru]_{\prod_{i\in I}\otimes_i} & 
}
$$
the unit object is $\mathbf{e}=(e_i)_{i\in I}$ and the structural isomorphisms are given by
\begin{align*}
a_{\mathbf{x},\mathbf{y},\mathbf{z}}&=(a^{(i)}_{x_i,y_i,z_i})_{i\in I},\\ l_\mathbf{x}&=(l^{(i)}_{x_i})_{i\in I}, \\ r_\mathbf{x}&=(r^{(i)}_{x_i})_{i\in I}
\end{align*}  
for any $\mathbf{x}=(x_i)_{i\in I}$, $\mathbf{y}=(y_i)_{i\in I}$ and $\mathbf{z}=(z_i)_{i\in I}$. It follows that any object $\mathbf{x}\in\prod_{i\in I}\Gg_i$ is invertible with (weak) inverse $\mathbf{x}^\ast=(x^\ast_i)_{i\in I}$. In particular, $\prod_{i\in I}\GG_i$ is strict if and only if all 2-groups $\GG_i$ are strict. We shall just write $\GG^n$ when $|I|=n$ and all $\GG_i$ are equal to the same 2-group $\GG$. 

\begin{ex}{\rm
For any family of groups $\{\Gsf_i\}_{i\in I}$ we have $\prod_{i\in I}\Gsf_i[0]=(\prod_{i\in I}\Gsf_i)[0]$. Similarly, for any family of abelian groups $\{\Asf_i\}_{i\in I}$ we have $\prod_{i\in I}\Asf_i[1]=(\prod_{i\in I}\Asf_i)[1]$. (See Corollary~\ref{producte_2-grups_escindits} below for a generalization of these two facts.) }
\end{ex}

\subsection{Homotopy invariants, Sinh's theorem and Postnikov decomposition} 
\label{invariants_homotopics}

The automorphism group $\mathsf{Aut}(e)$ of the unit object in a monoidal groupoid $\GG$ is always abelian. This follows from the Eckmann-Hilton argument applied to the composition and tensor product. Moreover, for any object $x\in\Gg$, there are two canonical homomorphisms of groups $\gamma_x,\delta_x:\mathsf{Aut}(e)\To\mathsf{Aut}(x)$ given by
\begin{align*}
\gamma_x(u)&:=l_x\circ(u\otimes id_x)\circ l_x^{-1}, \\ \delta_x(u)&:=r_x\circ(id_x\otimes u)\circ r_x^{-1}
\end{align*}
for any $u\in Aut(e)$. These homomorphisms have the following useful properties. Actually, all of them are true for arbitrary monoidal {\em categories}, except that in this case $\gamma_x,\delta_x$ are homomorphisms of monoids $\mathsf{End}(e)\To\mathsf{End}(x)$. The proofs are left to the reader as a good exercise to become acquainted with this kind of structure.
\begin{itemize}
\item[(1)] They are compatible with the tensor product in the sense that
\begin{equation}\label{compatibilitat_amb_otimes}
\gamma_{x\otimes y}(u)=\gamma_x(u)\otimes id_y,\quad \delta_{x\otimes y}(u)=id_x\otimes\delta_y(u), \quad \delta_x(u)\otimes id_y=id_x\otimes\gamma_y(u)
\end{equation}
for any $x,y\in\Gg$ and $u\in Aut(e)$.
\item[(2)] They are ``natural'' in $x$, i.e. for any morphism $f:x\To y$ the diagrams
\begin{equation}\label{naturalitat_gamma_delta}
\xymatrix{
X\ar[r]^{\gamma_x(u)}\ar[d]_f & X\ar[d]^f \\ Y\ar[r]_{\gamma_y(u)} & Y }\qquad \xymatrix{
X\ar[r]^{\delta_x(u)}\ar[d]_f & X\ar[d]^f \\ Y\ar[r]_{\delta_y(u)} & Y }
\end{equation}
commute for all $u\in Aut(e)$.
\item[(3)] They are preserved by any equivalence of monoidal groupoids $\EE=(\Ee,\mu):\GG\To\GG'$, i.e. the diagrams
\begin{equation}\label{preservacio_gamma_delta}
\xymatrix{
\mathsf{Aut}(e)\ar[r]^{\Ee_{e,e}}\ar[d]_{\delta_x} & \mathsf{Aut}(\Ee e)\ar[r]^{\tilde{\nu}} & \mathsf{Aut}(e')\ar[d]^{\delta'_{\Ee x}} \\ \mathsf{Aut}(x)\ar[rr]_{\Ee_{x,x}} && \mathsf{Aut}(\Ee x) }\qquad \xymatrix{\mathsf{Aut}(e)\ar[r]^{\Ee_{e,e}}\ar[d]_{\gamma_x} & \mathsf{Aut}(\Ee e)\ar[r]^{\tilde{\nu}} & \mathsf{Aut}(e')\ar[d]^{\gamma'_{\Ee x}} \\ \mathsf{Aut}(x)\ar[rr]_{\Ee_{x,x}} && \mathsf{Aut}(\Ee x)  }
\end{equation}
commute for all $x\in\Gg$, where $\tilde{\nu}$ denotes conjugation by the unit isomorphism $\nu$ of $\EE$.
\end{itemize}

For arbitrary monoidal groupoids, $\gamma_x$ and $\delta_x$ are neither injective nor surjective in general. However, for 2-groups they are both isomorphisms because all translation functors $-\otimes x,\ x\otimes-$ are equivalences and hence, fully faithful functors. It follows that the underlying groupoid of a 2-group $\GG$ is of a very particular type, namely
$$
\Gg\ \simeq\coprod_{[x]\in\pi_0(\Gg)}\Asf_{[x]},
$$
with $\Asf_{[x]}$ a copy of $\mathsf{Aut}(e)$ and $\pi_0(\Gg)$ the set of isomorphism classes of objects in $\Gg$. This makes 2-groups much simpler structures than arbitrary monoidal groupoids. In particular, their classification up to equivalence, discussed below, becomes much easier (for the classification of arbitrary monoidal groupoids, see \cite{CCH13}).

Associated to any 2-group $\GG$ we have the corresponding {\em homotopy groups}, so called because they correspond to the homotopy groups of the associated homotopy 2-type. The {\em first homotopy group} is the set $\pi_0(\Gg)$ equipped with the group structure induced by $\otimes$, i.e.
$$
[x]\cdot[y]=[x\otimes y].
$$
It will be denoted by $\pi_0(\mathbb{G})$. The {\it second homotopy group} is the abelian group $\mathsf{Aut}(e)$. It will be denoted by $\pi_1(\GG)$. It comes equipped with a canonical left $\pi_0(\mathbb{G})$-module structure defined by~\footnote{There is an analogous canonical right action given by $u\rhd[x]=\delta_x^{-1}(\gamma_x(u))$. However, we shall make no use of this action.}
\begin{equation}\label{accio}
[x]\lhd u:=\gamma_x^{-1}(\delta_x(u))
\end{equation}
for any representative $x$ of $[x]$. This is a well defined left action because of (\ref{compatibilitat_amb_otimes}) and (\ref{naturalitat_gamma_delta}). Notice that in the strict case this gives
\begin{equation}\label{accio_cas_estricte}
[x]\lhd u=id_x\otimes u\otimes id_{x^{-1}}.
\end{equation}
It readily follows from (\ref{compatibilitat_amb_otimes}) that
\begin{equation}\label{propietat}
\gamma_{x\otimes y}(u\circ([x]\lhd v))=\gamma_x(u)\otimes\gamma_y(v)
\end{equation}
for any object $x,y\in\Gg$ and any automorphisms $u,v\in Aut(e)$.

Any equivalence of 2-groups $\EE=(\Ee,\mu):\GG\To\GG'$ induces both an isomorphism of groups $\pi_0(\EE):\pi_0(\GG)\To\pi_0(\GG')$ and an isomorphism of $\pi_0(\GG)$-modules $\pi_1(\EE):\pi_1(\GG)\To\pi_1(\GG')$ given by~\footnote{These morphisms are induced by any morphism of 2-groups $\FF:\GG\To\GG'$, but they are not isomorphisms in the generic case.}
\begin{align}
&\pi_0(\EE)([x]):=[\Ee(x)], \label{morfisme1}\\ &\pi_1(\EE)(u):=\nu^{-1}\circ\Ee(u)\circ\nu \label{morfisme2}
\end{align}
where $\nu:e'\To\Ee e$ is the unit isomorphism of $\EE$. Here $\pi_1(\GG')$ is viewed as a $\pi_0(\GG)$-module via its $\pi_0(\GG')$-module structure and the morphism $\pi_0(\GG)$. In particular, both homotopy groups are homotopy invariants.

In general, the pair $(\pi_0(\GG),\pi_1(\GG))$ is not enough to completely classify a generic $\GG$ up to equivalence. The missing item is a cohomology class $\alpha(\GG)\in H^3(\pi_0(\GG),\pi_1(\GG))$ coming from the associator of $\GG$. It is called the {\em Postnikov invariant} of $\GG$ because it is related to the Postnikov tower of the associated homotopy type (see \cite{BS10}).  Any representative of $\alpha(\GG)$ will be called a {\it classifying 3-cocycle} of $\mathbb{G}$. One such representative (and hence, the Postnikov invariant of $\GG$) can be obtained by the following procedure, due to Sinh \cite{hxSi75}. We first choose what Sinh calls an {\it \'epinglage} of $\mathbb{G}$. This is a pair $(s,\theta)$ with 
\begin{itemize}
\item[(i)] $s:\pi_0(\mathbb{G})\To G_0$ any {\em normalized} section of the projection $p:G_0\To\pi_0(\GG)$ mapping each object $x\in\Gg$ to its isomorphism class $[x]$ (normalized means such that $s[e]=e$), and
\item[(ii)] $\theta=\{\theta_x:s[x]\To x,\ x\in\Gg\}$ any family of isomorphisms satisfying the {\it normalization conditions}
\begin{align*}
 \theta_{e\otimes s[x]}&=\lambda^{-1}_{s[x]}, \\ \theta_{s[x]\otimes e}&=\rho^{-1}_{s[x]}.
\end{align*}
\end{itemize}
In fact, Sinh also requires $\theta_{s[x]}=id_{s[x]}$. This condition, however, is unnecessary and it comes into conflict with the normalization conditions when tensoring with $e$ is the identity but the unitors are non-trivial.

Then a classifying 3-cocycle $\mathsf{z}(\mathbb{G})$ is uniquely determined by the commutativity of the diagrams
\begin{equation} \label{definicio_3-cocicle}
\xymatrix{
s[x\otimes x'\otimes x'']\ar[d]_{\theta_{s[x\otimes x']\otimes s[x'']}}\ar[rrrr]^{\gamma_{s[x\otimes x'\otimes x'']}({\sf z}(\mathbb{G})([x],[x'],[x'']))} &&&& s[x\otimes x'\otimes x''] \\ s[x\otimes x']\otimes s[x'']\ar[d]_{\theta_{s[x]\otimes s[x']}\otimes id} &&&& s[x]\otimes s[x'\otimes x'']\ar[u]_{\theta^{-1}_{s[x]\otimes s[x'\otimes x'']}} \\ (s[x]\otimes s[x'])\otimes s[x'']\ar[rrrr]_{a^{-1}_{s[x],s[x'],s[x'']}} &&& & s[x]\otimes(s[x']\otimes s[x''])\ar[u]_{id\otimes\theta^{-1}_{s[x']\otimes s[x'']}}
}
\end{equation}
for all triples $([x],[x'],[x''])$. The map so defined is indeed a 3-cocycle thanks to the pentagon axiom. In fact, it follows from the triangle axiom and the above normalization conditions on the pair $(s,\theta)$ that it is a normalized 3-cocycle.

Unlike $\pi_0(\mathbb{G})$ and $\pi_1(\mathbb{G})$, which are uniquely determined by $\mathbb{G}$, ${\sf z}(\mathbb{G})$ depends on the chosen {\it \'epinglage} $(s,\theta)$. However, different {\it \'epinglages} lead to cohomologous normalized 3-cocycles, and we have a well defined cohomology class $\alpha(\GG)\in\Hsf^3(\pi_0(\GG),\pi_1(\GG))$. Moreover, it is an homotopy invariant of $\GG$. The isomorphism of groups $\Hsf^3(\pi_0(\GG),\pi_1(\GG))\To\Hsf^3(\pi_0(\GG'),\pi_1(\GG'))$ induced by any equivalence of 2-groups $\EE:\GG\To\GG'$ maps $\alpha(\GG)$ to $\alpha(\GG')$.

\begin{thm}[Sinh, 1975]
Let $\GG$ be any 2-group, and let $\mathsf{z}(\GG)$ any classifying normalized 3-cocycle of $\GG$. Then $\GG$ is equivalent to the skeletal 2-group $\hat{\GG}$ given as follows. It has the elements of $\pi_0(\GG)$ as objects and all pairs $(u,[x])\in\pi_1(\GG)\times\pi_0(\GG)$ as morphisms, with
$$
(u,[x]):[x]\To[x].
$$
The unit object is $[e]$ and the composition and tensor product are given by
\begin{align*}
(u,[x])\circ(u',[x])&=(u\circ u',[x]), \\
[x]\otimes[x']&=[x\otimes x'], \\ (u,[x])\otimes(u',[x'])&=(u\circ([x]\lhd u'),[x\otimes x']).
\end{align*}
Finally, the associator is given by
$$
a_{[x],[x'],[x'']}=(\mathsf{z}(\GG)([x],[x'],[x'']),[x\otimes x'\otimes x'']),
$$
and the left and right unitors are trivial.
\end{thm}

The pentagon axiom holds precisely because $\mathsf{z}(\GG)$ is a 3-cocycle, and the triangle axiom holds because $\mathsf{z}(\GG)$ is normalized. The monoidal category so defined is a 2-group with $(u^{-1},[x])$ as inverse of $(u,[x])$, and $[x^\ast]$ as $\otimes$-inverse of $[x]$. 

Finally, let us point out that for any 2-group $\GG$ there is a canonical sequence of 2-group morphisms given by
$$
\mathbbm{1}\longrightarrow\pi_1(\GG)[1]\stackrel{\JJ}{\longrightarrow}\GG\stackrel{\PP}{\longrightarrow}\pi_0(\GG)[0]\longrightarrow\mathbbm{1}.
$$
Here $\JJ$ denotes the inclusion of $\pi_1(\GG)[1]$ into $\GG$ as the group of automorphisms of the unit object $e$, and $\PP$ is given by the functor mapping each object to its isomorphism class. Both morphisms are strict. This sequence is sometimes called the {\em Postnikov decomposition} of $\GG$. Notice that it is ``2-exact'' in the sense that (1) the underlying functor of $\JJ$ is an embedding, (2) $\JJ$ defines an equivalence of 2-groups between $\pi_1(\GG)[1]$ and the homotopy fiber of $\PP$ over $[e]$ (i.e. the full subgroupoid of $\Gg$ generated by the objects isomorphic to $e$), and (3) the underlying functor of $\PP$ is surjective. In Section~3 we shall see conditions under which this sequence ``splits''.

\subsection{2-actions, semidirect product of 2-groups and wreath 2-products}
\label{2-accions}

The notion of left or right action of a group on another group have the following straightforward generalization to arbitrary 2-groups (see \cite{GI01}).

\begin{defn}
Let $\GG$ and $\HH$ be 2-groups. A {\em left 2-action} of $\mathbb{G}$ on $\mathbb{H}$, also called a {\em $\GG$-2-group structure} on $\HH$, is a morphism of 2-groups $\mathbb{F}=(\Ff,\mu):\mathbb{G}\To\mathbb{E}quiv_\mathbf{2Grp}(\mathbb{H})$. A {\it right 2-action} of $\mathbb{G}$ on $\mathbb{H}$ is a left 2-action of $\GG^{op}$ on $\HH$.
\end{defn}

\noindent
More explicitly, a $\GG$-2-group structure on $\HH$ is given by the following data and axioms.
\begin{itemize}
\item[$\bullet$] A self-equivalence $\mathbb{E}_x=(\Ee_x,\mu_x)$ of $\mathbb{H}$ for each object $x\in\Gg$ (the image of $x$ by $\Ff$). We shall write
\begin{align*}
\Ee_x(m)&\equiv x\lhd m \\ \Ee_x(f)&\equiv x\lhd f
\end{align*}
for any object $m$ and morphism $f$ in $\mathcal{H}$. With these notations, the functoriality of $\Ee_x$ says that
\begin{align*}
x\lhd (g\circ f)&=(x\lhd g)\circ(x\lhd f), \\ x\lhd id_m&=id_{x\lhd m}, 
\end{align*}
and the monoidal structure $\mu_x$ of $\Ee_x$ is given by a family of isomorphisms
$$
\mu_{x;m,n}:x\lhd(m\otimes n)\stackrel{\cong}{\To}(x\lhd m)\otimes(x\lhd n),\qquad m,n\in\Hh
$$
natural in $m,n$ and such that
$$
\xymatrix{
x\lhd((m\otimes n)\otimes p)\ar[rr]^{x\lhd a^{-1}_{m,n,p}}\ar[d]_{\mu_{x;m\otimes n,p}} && x\lhd(m\otimes(n\otimes p))\ar[d]^{\mu_{x;m,n\otimes p}} \\ [x\lhd(m\otimes n)]\otimes(x\lhd p)\ar[d]_{\mu_{x;m,n}\otimes id_{x\lhd p}} && (x\lhd m)\otimes[x\lhd(n\otimes p)]\ar[d]^{id_{x\lhd m}\otimes\mu_{x;n,p}} \\ [(x\lhd m)\otimes(x\lhd n)]\otimes(x\lhd p)\ar[rr]_{a^{-1}_{x\lhd m,x\lhd n,x\lhd p}} && (x\lhd m)\otimes[(x\lhd n)\otimes(x\lhd p)] 
}
$$
commute for all objects $m,n,p\in\mathcal{H}$.

\item[$\bullet$] A monoidal natural isomorphism $\tau_\phi:\mathbb{E}_x\Rightarrow\mathbb{E}_y$ for each morphism $\phi:x\To y$ in $\Gg$ (the image of $\phi$ by $\Ff$). Thus $\tau_\phi$ is given by a family of isomorphisms
$$
\tau_{\phi;m}\equiv\phi\lhd m:x\lhd m\stackrel{\cong}{\To}y\lhd m,\qquad m\in\mathcal{H}
$$
natural in $m$ and such that
$$
\xymatrix{
x\lhd(m\otimes n)\ar[rr]^{\mu_{x;m,n}}\ar[d]_{\phi\lhd(m\otimes n)} & & (x\lhd m)\otimes(x\lhd n)\ar[d]^{(\phi\lhd m)\otimes(\phi\lhd n)} \\ y\lhd(m\otimes n)\ar[rr]_{\mu_{y;m,n}} & & (y\lhd m)\otimes(y\lhd n)
}
$$
commutes for all objects $m,n\in\mathcal{H}$. Moreover, by the functoriality of $\Ff$ it must be
\begin{align*}
(\psi\circ\phi)\lhd m&=(\psi\lhd m)\circ(\phi\lhd m) \\ id_x\lhd m&=id_{x\lhd m}
\end{align*}
for all composable morphisms $\phi,\psi$ in $\Gg$ and all objects $x\in\Gg$ and $m\in\mathcal{H}$.
\item[$\bullet$] A monoidal natural isomorphism $\mu_{x,y}:\mathbb{E}_{x\otimes y}\Rightarrow\mathbb{E}_x\circ\mathbb{E}_y$ for each pair of objects $x,y\in\Gg$ (the monoidal structure of $\Ff$). Thus $\mu_{x,y}$ is given by a family of isomorphisms
$$
\mu_{x,y;m}:(x\otimes y)\lhd m\stackrel{\cong}{\To}x\lhd(y\lhd m),\qquad m\in\mathcal{H}
$$
natural in $m$ and such that
\begin{equation*}
\xymatrix{
(x\otimes y)\lhd(m\otimes n)\ar[rr]^{\mu_{x\otimes y;m,n}}\ar[d]_{\mu_{x,y;m\otimes n}} & & [(x\otimes y)\lhd m]\otimes[(x\otimes y)\lhd n]\ar[d]^{\mu_{x,y;m}\otimes\mu_{x,y;n}} \\ x\lhd(y\lhd(m\otimes n))\ar[rd]_{x\lhd\mu_{y;m,n}} & & [(x\lhd(y\lhd m)]\otimes[x\lhd(y\lhd n)] \\ & x\lhd[(y\lhd m)\otimes(y\lhd n)]\ar[ru]_{\mu_{x;y\lhd m,y\lhd n}} &
}
\end{equation*}
commutes for all objects $m,n\in\mathcal{H}$.
\end{itemize}
Such a set of data will be denoted by $(\Ee,\mu,\tau,\mu)$, and sometimes just $\lhd$. The data defining a right 2-action is exactly the same except that $\mu_{x,y}$ is a natural isomorphism $\mu_{x,y}:\Ee_{y\otimes x}\Rightarrow\Ee_x\circ\Ee_y$. For right 2-actions we shall use the notations
\begin{align*}
\Ee_x(m)&\equiv m\rhd x \\ \Ee_x(f)&\equiv f\rhd x.
\end{align*}
Then the remaining data defining a right 2-action look like
$$
\mu_{x;m,n}:(m\otimes n)\rhd x\stackrel{\cong}{\To}(m\rhd x)\otimes(n\rhd x)
$$   
$$
m\rhd\phi:m\rhd x\stackrel{\cong}{\To}m\rhd y
$$
$$
\mu_{x,y;m}:m\rhd (y\otimes x)\stackrel{\cong}{\To}(m\rhd y)\rhd x
$$

By taking the composite with the equivalence $\GG^{op}\simeq\GG$ or with a pseudoinverse, any left 2-action has an associated right 2-action, and conversely. In fact, there are various such right or left 2-actions associated to a given left or right 2-action, respectively. They depend on the chosen equivalence $\GG^{op}\simeq\GG$ (see Example~\ref{equivalencia_G-op_G} above). However, all of them are equivalent in the appropriate sense. Hence we can equally think of left or of right 2-actions.

For any left 2-action $\lhd$ and any morphisms $\phi:x\To y$, $f:m\To n$ in $\Gg$ and $\Hh$, respectively, the naturality of $\tau_{\phi;m}$ in $m$ implies the existence of a well defined morphism $\phi\lhd f:x\lhd m\To y\lhd n$ given by
\begin{equation}\label{phi_lhd_f}
\phi\lhd f:=(\phi\lhd n)\circ(x\lhd f)=(y\lhd f)\circ(\phi\lhd m),
\end{equation}
and similarly for right 2-actions.

Since the morphisms of 2-groups preserve the unit objects (see Proposition~\ref{existencia_iso_unitat}), hidden in any left 2-action of a 2-group $\mathbb{G}$ on a 2-group $\mathbb{H}$ we have two additional families of unit isomorphisms
\begin{align*}
\nu_x&:e\stackrel{\cong}{\To} x\lhd e,\qquad x\in\Gg \\ \nu_m&:m\stackrel{\cong}{\To} e\lhd m, \qquad m\in\Hh.
\end{align*}
The first family comes from the preservation of the unit object of $\HH$ by the self-equivalences $\mathbb{E}_x$. The second one just corresponds to the components of a unit isomorphism $\nu: id_\mathbb{H}\Rightarrow\mathbb{E}_e$ associated to the morphism $\mathbb{G}\To\mathbb{E}quiv(\mathbb{H})$. Both families are natural in the respective labels. Moreover, the diagram
\begin{equation*}
\xymatrix{
(x\otimes y)\lhd e\ar[rr]^{\mu_{x,y;e}} & & x\lhd(y\lhd e) \\ e\ar[u]^{\nu_{x\otimes y}}\ar[rr]_{\nu_x}  && x\lhd e\ar[u]_{x\lhd\nu_y}
}
\end{equation*}
commutes for any $x,y\in\Gg$. Indeed, the unit isomorphism of the morphism $\mathbb{E}_x\circ\mathbb{E}_y$ is given by
$$
\xymatrix{
\Ee_x(\Ee_y(e))\ar[rr]^{\Ee_x(\nu_y)} && \Ee_x(e)\ar[r]^{\nu_x} & e. }
$$
In our notations, this is the morphism $\nu_x\circ(x\lhd\nu_y)$. Hence the commutativity of the previous diagram follows from the monoidality of $\mu_{x,y}$. Finally, the two isomorphisms $\nu_{e}$, the unit isomorphism of $\EE_e$ and the $e$-component of $\nu$, coincide because of (\ref{monoidalitat_nu}).

Similarly, hidden in any right 2-action $\rhd$ we have unit isomorphisms
\begin{align*}
\nu_x&:x\stackrel{\cong}{\To} e\rhd x \\ \nu_m&:m\stackrel{\cong}{\To} m\rhd e
\end{align*}
satisfying analogous conditions.

Notice that any 2-action of $\GG$ on $\HH$ is equivalent to a {\em normalized} one in which the unit isomorphisms $\nu_m$ are all trivial. This is because any morphism of 2-groups is 2-isomorphic to a normalized one. Nevertheless, it seems that the isomorphisms $\nu_x$ can not be assumed to be all simultaneously trivial without loss of generality. Indeed, they refer to different morphisms of 2-groups, the self-equivalences $\mathbb{E}_x$. However, these morphisms are not independent. They are part of a bigger structure, namely, the morphism of 2-groups from $\mathbb{G}$ to $\mathbb{E}quiv_\mathbf{2Grp}(\mathbb{H})$.

A left or right 2-action will be called {\em strict} when the corresponding morphism of 2-groups is strict (i.e. all isomorphisms $\mu_{x,y;m}$ and all unit isomorphisms $\nu_m$ are identities), and it will be called {\em strongly strict} when it is strict and all self-equivalences $\EE_x$ are strict monoidal functors (i.e. all $\mu_{x;m,n}$ and all unit isomorphisms $\nu_x$ are also trivial). 

\begin{ex}\label{accio_G[0]_sobre_H[0]}{\rm
Any left action of a group $\Gsf$ on a group $\Hsf$ induces a strongly strict left 2-action of $\Gsf[0]$ on $\Hsf[0]$ given by the composition
$$
\Gsf[0]\longrightarrow\mathsf{Aut}(\Hsf)[0]\stackrel{\cong}{\longrightarrow}\EE quiv(\Hsf[0]).
$$
More interestingly, when $\Hsf$ is an abelian group $\Asf$, any left $\Gsf$-module structure on $\Asf$ induces an additional strongly strict left 2-action of $\Gsf[0]$ on $\Asf[1]$, given by the composition
$$
\Gsf[0]\longrightarrow\mathsf{Aut}(\Asf)[0]\stackrel{\simeq}{\longrightarrow}\EE quiv(\Asf[1])
$$
(see Example~\ref{autoequivalencies_G[0]_i_A[1]}). }
\end{ex}

\begin{ex}[{\em Wreath 2-action}]\label{accio_wreath}{\rm
For any 2-group $\GG$ and any $n\geq 1$, there are canonical left and right 2-actions of $\Ssf_n[0]$ on the product 2-group $\GG^n$ which generalize the usual wreath actions of $\Ssf_n$. In the right case, it is the strongly strict 2-action
$$
\WW_{n,\GG}:\Ssf_n^{op}[0]\To\EE quiv_\mathbf{2Grp}(\GG^n)
$$
defined as follows. It maps any $\sigma\in\Ssf_n$ to the permutation morphism $\PP_\sigma:\GG^n\To\GG^n$ with underlying self-equivalence $\Pp_\sigma:\Gg^n\To\Gg^n$ the automorphism of $\Gg^n$ defined by
\begin{align}
\Pp_\sigma(\mathbf{x})\equiv\mathbf{x}\rhd\sigma&:=(x_{\sigma(1)},\ldots,x_{\sigma(n)}), \label{def_P_sigma_1}\\ \Pp_\sigma(\mathbf{f})\equiv\mathbf{f}\rhd\sigma&:=(f_{\sigma(1)},\ldots,f_{\sigma(n)})\label{def_P_sigma_2}
\end{align}
for any object $\mathbf{x}=(x_1,\ldots,x_n)$ and morphism $\mathbf{f}=(f_1,\ldots,f_n)$. This automorphism is a strict automorphism of $\GG^n$. Indeed, the monoidal structure on $\GG^n$ is defined componentwise (see \S~\ref{2-categoria_2Grp}). Thus we have
\begin{align*}
\Pp_\sigma(\mathbf{x}\otimes\mathbf{x}')&=(x_1\otimes x'_1,\ldots,x_n\otimes x'_n)\rhd\sigma \\ &=(x_{\sigma(1)}\otimes x'_{\sigma(1)},\ldots,x_{\sigma(n)}\otimes x'_{\sigma(n)}) \\ &=(x_{\sigma(1)},\ldots,x_{\sigma(n)})\otimes(x'_{\sigma(1)},\ldots,x'_{\sigma(n)}) \\ &=((x_1,\ldots,x_n)\rhd\sigma)\otimes((x'_1,\ldots,x'_n)\rhd\sigma) \\ &=\Pp_\sigma(\mathbf{x})\otimes\Pp_\sigma(\mathbf{x}'),
\end{align*}
and similarly for morphisms. Moreover,
\begin{align*}
a_{\Pp_\sigma(\mathbf{x}),\Pp_\sigma(\mathbf{x}'),\Pp_\sigma(\mathbf{x}'')}&=(a_{x_{\sigma(1)},x'_{\sigma(1)},x''_{\sigma(1)}},\ldots,a_{x_{\sigma(n)},x'_{\sigma(n)},x''_{\sigma(n)}}) \\ &=(a_{x_1,x'_1,x''_1},\ldots,a_{x_n,x'_n,x''_n})\rhd\sigma \\ &=\Pp_\sigma(a_{\mathbf{x},\mathbf{x}',\mathbf{x}''}),
\end{align*}
so that (\ref{axioma_coherencia}) commutes when $\mu$ is equal to the identity. Notice that the corresponding unit isomorphism $\nu$ is also trivial because the unit isomorphism of $\GG^n$ as a monoidal groupoid is $\mathbf{d}=(d,\ldots,d)$ and hence, both $r_{\Pp_\sigma(\mathbf{e})}$ and $\Pp_\sigma(\mathbf{d})$ are equal to $\mathbf{d}$.

The 2-action so defined is indeed strict (hence, strongly strict) because $\EE quiv_{\mathbf{2Grp}}(\GG^n)$ is strict as a monoidal groupoid and
$$
\Ww_{n,\GG}(\sigma\sigma')=\Ww_{n,\GG}(\sigma')\circ\Ww_{n,\GG}(\sigma)
$$
for any permutations $\sigma,\sigma'\in S_n$. When $\GG=\Gsf[0]$, this 2-action reduces to the usual wreath action of $\Ssf_n$ on the group $\Gsf$. }
\end{ex}

\begin{rem}\label{accio_wreath_general}{\rm
This last Example can be further generalized to a canonical 2-action of $\SSS ym(\Kk)$ on the {\em cotensor product} 2-group $\HH^\Kk$ for any non necessarily discrete groupoid $\Kk$. However, we do not consider this general case because it is not needed in the sequel.}
\end{rem}

The following definition already appears in \cite{GI01}. It generalizes to arbitrary 2-groups the notion of semidirect product of groups.

\begin{defn}\label{producte_semidirecte}
Let $\rhd=(\Ee,\mu,\tau,\mu)$ be a right action of a 2-group $\mathbb{G}$ on a 2-group $\mathbb{H}$. Then the {\em semidirect product} of $\mathbb{G}$ and $\mathbb{H}$, denoted by $\mathbb{G}\ltimes\mathbb{H}$, is the groupoid $\Gg\times\mathcal{H}$ equipped with the following monoidal structure:
\begin{itemize}
\item[$\bullet$] the tensor product is given on objects $(x,m),(x',m')$ and morphisms $(\phi,f)$, $(\phi',f')$ by
$$
(x,m)\otimes(x',m'):=(x\otimes x',(m\rhd x')\otimes m'),
$$
$$
(\phi,f)\otimes(\phi',f'):=(\phi\otimes\phi',(f\rhd\phi')\otimes f');
$$ 
\item[$\bullet$] the unit object is $(e,e)$;
\item[$\bullet$] the associator is given by
$$
a_{(x,m),(x',m'),(x'',m'')}:=(a_{x,x',x''},\hat{a}_{x',x'';m,m',m''}),
$$
with $\hat{a}_{x',x'';m,m',m''}$ the isomorphism uniquely defined by the commutative diagram
\begin{equation}\label{hat_a}
\xymatrix{
(m\rhd(x'\otimes x''))\otimes ((m'\rhd x'')\otimes m'')\ar[d]_{a_{m\rhd(x'\otimes x''),m'\rhd x'',m''}}\ar[rr]^{\ \ \ \ \ \ \ \hat{a}_{x',x'';m,m',m''}\ \ \ \ }  & & [((m\rhd x')\otimes m')\rhd x'']\otimes m''\ar[d]^{\mu_{x'';m\rhd x',m'}\otimes id} \\  [(m\rhd(x'\otimes x''))\otimes(m'\otimes x'')]\otimes m''\ar[rr]_{\ (\mu_{x',x'';m}\otimes id)\otimes id\ \ } & & [((m\rhd x')\rhd x'')\otimes(m'\rhd x'')]\otimes m''
}
\end{equation}
for any objects $x',x''\in\Gg$ and $m,m',m''\in\Hh$;
\item[$\bullet$] the left and right unitors are given by
$$
l_{(x,m)}:=(l_{x},\hat{l}_{x,m}),
$$
$$
r_{(x,m)}:=(r_{x},\hat{r}_{m}),
$$
with $\hat{l}_{x,m}$ and $\hat{r}_{m}$ the isomorphisms defined by the commutative diagrams
\begin{equation}\label{hat_l_r}
\xymatrix{
(e\rhd x)\otimes m\ar[r]^{\ \ \ \ \hat{l}_{x,m}}\ar[d]_{\nu_x\otimes id}  & m \\ e\otimes m\ar[ru]_{l_m} &
} \qquad \xymatrix{
(m\rhd e)\otimes e\ar[r]^{\ \ \ \ \hat{r}_{m}}\ar[d]_{\nu_m\otimes id}  & m \\ m\otimes e\ar[ru]_{r_m} &
}
\end{equation}
for any $x\in\Gg$ and $m\in\Hh$.
\end{itemize}
Similarly, for any left action $\lhd=(\Ee,\mu,\tau,\mu)$ of $\GG$ on $\HH$, the semidirect product of $\GG$ and $\HH$, denoted by $\HH\rtimes\GG$, is the groupoid $\Hh\times\Gg$ equipped with the analogous monoidal structure.
\end{defn}
It is long but easy to check that the structure so defined is indeed a 2-group. In particular, a two-sided (weak) inverse in $\GG\ltimes\HH$ of $(x,m)\in\Gg\times\mathcal{H}$ is given by the pair
$$
(x,m)^\ast=(x^\ast,m^\ast\rhd x^\ast)
$$
for any two-sided inverses $x^\ast$ and $m^\ast$ of $x$ and $m$, respectively. Indeed, by definition we have
\begin{align*}
(x,m)^\ast\otimes(x,m)&=(x^\ast\otimes x,[(m^\ast\rhd x^\ast)\rhd x]\otimes m), \\  (x,m)\otimes(x,m)^\ast&=(x\otimes x^\ast,(m\rhd x^\ast)\otimes(m^\ast\rhd x^\ast)).
\end{align*}
Then any isomorphisms $\epsilon_x:x^\ast\otimes x\To e$, $\eta_x:e\To x\otimes x^\ast$ and $\epsilon_m:m^\ast\otimes m\To e$, $\eta_m:e\To m\otimes m^\ast$ induce isomorphisms
$$
\xymatrix{
[(m^\ast\rhd x^\ast)\rhd x]\otimes m\ar[rr]^{\mu^{-1}_{x^\ast,x;m^\ast}\otimes id} && [m^\ast\rhd(x^\ast\otimes x)]\otimes m\ar[rr]^{\ \ \ \ (m^\ast\rhd\epsilon_x)\otimes id} && (m^\ast\rhd e)\otimes m \ar[r]^{\ \ \ \ \nu_{m^\ast}\otimes id} & m^\ast\otimes m\ar[r]^{\ \ \ \ \epsilon_m} & e},
$$
$$
\xymatrix{
(m\rhd x^\ast)\otimes(m^\ast\rhd x^\ast)\ar[rr]^{\ \ \ \ \mu^{-1}_{x^\ast;m,m^\ast}} && (m\otimes m^\ast)\rhd x^\ast\ar[rr]^{\ \ \ \ \eta^{-1}_m\rhd x^\ast} & & e\rhd x^\ast\ar[r]^{\ \ \ \ \nu_{x^\ast}} & e
}.
$$

\espai
Notice that even if $\mathbb{G}$ and $\mathbb{H}$ are both strict 2-groups, $\GG\ltimes\HH$ need not be strict unless the action of $\GG$ on $\HH$ is {\em strongly} strict. Indeed, the isomorphisms $\mu_{x,y;m}$ and $\mu_{x;m,n}$ and the associated units $\nu_x,\nu_m$ both appear explicitly in the monoidal structure of the semidirect product, and also in the previous isomorphisms between $(x,m)\otimes(x,m)^\ast$ and $(e,e)$.

\begin{ex}[{\em Discrete case}]{\rm
Returning to Example~\ref{accio_G[0]_sobre_H[0]}, any right $\Gsf$-group $\Hsf$ gives rise to a semidirect product 2-group $\Gsf[0]\ltimes\Hsf[0]$. It readily follows from the above general definition that $\Gsf[0]\ltimes\Hsf[0]\cong(\Gsf\ltimes\Hsf)[0]$. }
\end{ex}

We also know from Example~\ref{accio_G[0]_sobre_H[0]} that any left $\Gsf$-module structure on an abelian group $\Asf$ induces a left action of $\Gsf[0]$ on $\Asf[1]$. The corresponding semidirect product 2-groups $\Asf[1]\rtimes\Gsf[0]$ play a basic role in what follows. They will be called (left) {\em elementary 2-groups}. They have the following very simple description.

\begin{prop}\label{A[1]_rtimes_G[0]}
Up to isomorphism, the elementary 2-group $\Asf[1]\rtimes\Gsf[0]$ is the strict 2-group with the elements of $G$ as objects, all pairs $(a,g)\in A\times G$ as morphisms, with $(a,g):g\To g$, and the composition and tensor product given by
\begin{align*}
(a',g)\circ(a,g)&=(a'+a,g), \\
g\otimes g'&=gg', \\ (a,g)\otimes(a',g')&=(a+g\lhd a',gg').
\end{align*}
The identity of $g$ is $(0,g)$, the inverse of $(a,g)$ is $(-a,g)$, the unit object is the unit $1\in G$ and the $\otimes$-inverse of $g$ is $g^{-1}$.
\end{prop}
\begin{proof}
It readily follows from the general description of a semidirect product as given in Definition~\ref{producte_semidirecte} and the action of $\Gsf[0]$ on $\Asf[1]$ as defined in Example~\ref{accio_G[0]_sobre_H[0]}. The details are left to the reader.
\end{proof}
 
In spite of their simplicity, elementary 2-groups are almost generic. Indeed, the 2-group $\hat{\GG}$ defined in \S~\ref{invariants_homotopics} is almost of this kind, with $\Gsf=\pi_0(\GG)$ and $\Asf=\pi_1(\GG)$ equipped with its canonical $\pi_0(\GG)$-module structure given by (\ref{accio}). $\hat{\GG}$ only differs from an elementary 2-group in that it comes equipped with a non-trivial associator. In the general case, the associator is given by
$$
a_{g,g',g''}=(\mathsf{z}(g,g',g''),gg'g'')
$$
for a given normalized 3-cocicle $\mathsf{z}$ of $\Gsf$ with values in $\Asf$. The 2-group so defined will be denoted by $\Asf[1]\rtimes_\mathsf{z}\Gsf[0]$. Such 2-groups will be called {\em special}. Sinh's theorem says then that any 2-group $\GG$ is equivalent to the special 2-group
\begin{equation*}
\mathbb{G}\simeq\pi_1(\mathbb{G})[1]\rtimes_{{\sf z}(\mathbb{G})}\pi_0(\mathbb{G})[0]
\end{equation*}
for any classifying (normalized) 3-cocycle $\mathsf{z}(\GG)$ of $\GG$.

Another family of examples of semidirect products which also plays a basic role in the sequel is the following.

\begin{defn}
For any $n\geq 1$ and any 2-group $\GG$ the {\em wreath 2-product} of $\Ssf_n$ with $\GG$, denoted by $\Ssf_n\wr\wr\ \GG$, is the semidirect product 2-group $\Ssf_n[0]\ltimes\GG^n$ associated the the wreath 2-action of $\Ssf_n[0]$ on $\GG^n$ described in Example~\ref{accio_wreath}.
\end{defn}
It has the following explicit description.

\begin{prop}\label{S_n_wreath_GG}
The wreath 2-product $\Ssf_n\wr\wr\ \GG$ is the groupoid $\Ssf_n[0]\times\Gg^n$ equipped with the following monoidal structure:
\begin{itemize}
\item[(a)] the tensor product is given by
\begin{align*}
(\sigma,\mathbf{x})\otimes(\sigma',\mathbf{x}')&=(\sigma\sigma',(\mathbf{x}\rhd\sigma')\otimes\mathbf{x}'), \\ (id_\sigma,\mathbf{f})\otimes(id_{\sigma'},\mathbf{f}')&=(id_{\sigma\sigma'},(\mathbf{f}\rhd\sigma')\otimes\mathbf{f}')
\end{align*}
for any permutations $\sigma,\sigma'\in S_n$ and any objects $\mathbf{x},\mathbf{x}'$ and morphisms $\mathbf{f},\mathbf{f}'$ in $\Gg^n$;
\item[(b)] the unit object is $(id_n,\mathbf{e})$, with $\mathbf{e}=(e,\ldots,e)$;
\item[(c)] the associator is given by
\begin{equation*}
a_{(\sigma,\mathbf{x}),(\sigma',\mathbf{x}'),(\sigma'',\mathbf{x}'')}:=(id_{\sigma\sigma'\sigma''},a_{\mathbf{x}\rhd(\sigma'\sigma''),\mathbf{x}'\rhd\sigma'',\mathbf{x}''})
\end{equation*}
for any permutations $\sigma,\sigma'\sigma''\in S_n$ and objects $\mathbf{x},\mathbf{x}',\mathbf{x}''$ in $\Gg^n$;
\item[(d)] the left and right unitors are given by
\begin{equation*}
l_{(\sigma,\mathbf{x})}:=(id_\sigma,l_{\mathbf{x}}),
\end{equation*}
\begin{equation*}
r_{(\sigma,\mathbf{x})}:=(id_\sigma,r_{\mathbf{x}}),
\end{equation*}
for any permutation $\sigma\in S_n$ and object $\mathbf{x}\in\Gg^n$.
\end{itemize}
In particular, $\Ssf_n\wr\wr\ \GG$ is strict (resp. skeletal) when $\GG$ is strict (resp. skeletal), and in this case the strict inverse of $(\sigma,\mathbf{x})$ is given by
\begin{equation*}
(\sigma,\mathbf{x})^{-1}=(\sigma^{-1},\mathbf{x}^{-1}\rhd\sigma^{-1})
\end{equation*}
where $\mathbf{x}^{-1}=(x_1^{-1},\ldots,x_n^{-1})$.
\end{prop}
\begin{proof}
The expression for the tensor product readily follows from (\ref{def_P_sigma_1})-(\ref{def_P_sigma_2}) and the analog of (\ref{phi_lhd_f}) for right actions, and the associator and left and right unitors follow from (\ref{hat_a}) and (\ref{hat_l_r}), respectively. 
\end{proof}
The reader may easily check that equivalent 2-groups have equivalent wreath 2-products with the same permutation group, i.e. we have
$$
\GG\simeq\GG'\ \Longrightarrow\ \Ssf_n\wr\wr\ \GG\simeq\Ssf_n\wr\wr\ \GG'.
$$
Later on, we shall compute the homotopy invariants of $\Ssf_n\wr\wr\ \GG$ in terms of $n$ and the homotopy invariants of  $\GG$ (see Proposition~\ref{invariants_producte_wreath}).

\begin{rem}\label{wreath_2-product_general}{\rm
More generally, for any 2-group $\GG$, any (non necessarily discrete) groupoid $\Kk$ and any right 2-action of a 2-group $\HH$ on $\Kk$ (i.e. any morphism of 2-groups $\HH^{op}\To\SSS ym(\Kk)$), one can also define the wreath 2-product $\HH\wr\wr_\Kk\GG$. Indeed the 2-action of $\HH$ on $\Kk$ together with the canonical 2-action of $\SSS ym(\Kk)$ on the cotensor product $\GG^\Kk$ of $\GG$ by $\Kk$ mentioned in Remark~\ref{accio_wreath_general} induces a 2-action of $\HH$ on $\GG^\Kk$. Then $\HH\wr\wr_\Kk\GG$ is the corresponding semidirect product 2-group $\HH\ltimes\GG^\Kk$. Details will be given in a future work. 
}\end{rem}

\section{\large Split 2-groups and splitness criteria}
\label{2-grups_escindits}

In this section we define split 2-groups and discuss various necessary and sufficient conditions of splitness. The first condition (Theorem~\ref{criteri_general}) works for any 2-group, and it is an easy consequence of Sinh's algorithm for constructing a classifying 3-cocycle of a 2-group from a given {\em \'epinglage} (see \S~\ref{invariants_homotopics}). The next two conditions (Theorem~\ref{criteri_general_2}) also work for arbitrary 2-groups, and justify the name given to these kind of 2-groups. Finally, the last condition (Theorem~\ref{criteri_split_2-grups_estrictes}) applies only to {\em strict} 2-groups, but in some cases it gives an easier way to establish the splitness of the 2-group. As a first application, we shall see that the wreath 2-product $\Ssf_n\wr\wr\ \GG$ is split if and only if $\GG$ is split. In Section~4, we shall use this last condition to prove the existence of non-split 2-groups of permutations.

\subsection{Split 2-groups.}

As said in the introduction, a split 2-group is basically an elementary 2-group, i.e. a semidirect product of a discrete and a one-object 2-group. However, the notion of elementary 2-group is not invariant by equivalences. Therefore, we take as definition the following.

\begin{defn}
A 2-group $\GG$ is {\em split} if it is equivalent to an elementary 2-group $\Asf[1]\rtimes\Gsf[0]$ for some group $\Gsf$ and some $\Gsf$-module $\Asf$.
\end{defn}
Observe that, as well as strict, elementary 2-groups are skeletal. This gives half of the next alternative characterization of split 2-groups.

\begin{prop}\label{G_split_sii_estricte_esqueletic}
A 2-group $\GG$ is split if and only if it is equivalent to a strict skeletal 2-group.
\end{prop}
\begin{proof}
It is enough to prove that for any strict skeletal 2-group $\KK$ we have
\begin{equation*}
\KK\cong\pi_1(\KK)[1]\rtimes\Ksf_0[0],
\end{equation*}
where $\Ksf_0$ denotes the group of objects of $\KK$ with the product given by the tensor product. The action of $\Ksf_0$ on $\Ksf_1$ is that given by
$$
k\lhd u=\gamma_k^{-1}(\delta_k(u)).
$$
Let $\Ff:\Kk\To\pi_1(\KK)[1]\times\Ksf_0[0]$ be the functor acting on objects as the identity and on a morphism $f:k\To k$ by
$$
\Ff f=(\gamma_k^{-1}(f),k).
$$
Since $\gamma_k:\pi_1(\KK)\To\mathsf{Aut}(k)$ is an isomorphism of groups, it is a fully faithful functor. Moreover, for any morphisms $f:k\To k$ and $f':k'\To k'$ we have
\begin{align*}
(\Ff f)\otimes(\Ff f')&=(\gamma_k^{-1}(f),k)\otimes(\gamma^{-1}_{k'}(f'),k') \\ &=(\gamma_k^{-1}(f)\circ(k\lhd\gamma^{-1}_{k'}(f')),k\otimes k').
\end{align*}
Now, it follows from (\ref{propietat}) that
$$
\gamma_{k\otimes k'}(\gamma_k^{-1}(f)\circ(k\lhd\gamma^{-1}_{k'}(f')))=\gamma_k(\gamma_k^{-1}(f))\otimes\gamma_{k'}(\gamma^{-1}_{k'}(f'))=f\otimes f'
$$
and hence
\begin{align*}
\Ff(f\otimes f')&=(\gamma_{k\otimes k'}^{-1}(f\otimes f'),k\otimes k') \\ &=(\gamma_k^{-1}(f)\circ(k\lhd\gamma^{-1}_{k'}(f')),k\otimes k') \\ &=(\Ff f)\otimes(\Ff f').
\end{align*}
Therefore, $\Ff$ is a strict monoidal functor and hence, an isomorphism of 2-groups.
\end{proof}

\begin{ex}\label{2-grup_escindit_grup_abelia}{\rm
For any abelian group $\Asf$, the 2-group $\SSS ym(\Asf)$ of self-equivalences of $\Asf$ as a one-object groupoid is split. Indeed, we shall see in Corollary~\ref{Sim(A)} below that
$$
\SSS ym(\Asf)\simeq\Asf[1]\rtimes\mathsf{Aut}(\Asf)[0],
$$
with $\mathsf{Aut}(\Asf)$ acting on $\Asf$ in the canonical way. }
\end{ex}

\begin{ex}\label{exemple_2-grup_split}{\rm
The previous example does not generalize to arbitrary non-abelian groups (see below). However, it remains true when the construction is appropriately linearized. More precisely, for any field $\Fsf$, let $\Vv ect_\Fsf$ the category of finite dimensional $\Fsf$-vector spaces, and $\mathbf{Cat}_\Fsf$ the 2-category of (small) categories enriched on $\Vv ect_\Fsf$, i.e. the 2-category of (small) $\Fsf$-linear categories, $\Fsf$-linear functors and natural transformations. For any group $\Gsf$, an object in $\mathbf{Cat}_\Fsf$ is given by the additive completion of the free $\Fsf$-linear category generated by the one-object groupoid $\Gsf$. Let us denote this category by $\Vv ect_\Fsf[\Gsf]$. For instance, when $\Gsf$ is trivial, it is nothing but the category $\Vv ect_\Fsf$ itself. More generally, the product category $\Vv ect_\Fsf[\Gsf]^n$ for any $n\geq 1$ is also an object in $\mathbf{Cat}_\Fsf$. Then when $\Fsf$ is algebraically closed the 2-group $\EE quiv(\Vv ect_\Fsf[\Gsf]^n)$ is split. Indeed, let $r$ be the number of conjugacy classes of $\Gsf$, $d_1,\ldots,d_s$ the dimensions of the finite dimensional irreducible representations of $\Gsf$, and $k_i\geq 1$ the number of non equivalent irreducible representations of dimension $d_i$ (in particular, $k_1+\cdots+k_s=r$). Then there exists a left action of the group $\Ssf_n\times(\Ssf_{k_1}\times\cdots\times\Ssf_{k_s})^n$ on the abelian group $(\Fsf^\ast)^{rn}$ such that
$$
\EE quiv(\Vv ect_\Fsf[\Gsf]^n)\simeq(\Fsf^\ast)^{rn}[1]\rtimes(\Ssf_n\times(\Ssf_{k_1}\times\cdots\times\Ssf_{k_s})^n)[0].
$$
For more details and a proof of this see \cite{jE5}. In particular, when $\Gsf$ is trivial it follows that
$$
\EE quiv_{\mathbf{Cat}_\Fsf}(\Vv ect_\Fsf^n)\simeq(\Fsf^\ast)^n[1]\rtimes\Ssf_n[0],
$$
where $(\Fsf^*)^n$ is equipped with the usual wreath product action of $\Ssf_n$
$$
\sigma\lhd(\lambda_1,\ldots,\lambda_n)=(\lambda_{\sigma^{-1}(1)},\ldots,\lambda_{\sigma^{-1}(n)})
$$
for all $(\lambda_1,\ldots,\lambda_n)\in(\Fsf^*)^n$.
}
\end{ex}

\begin{prop}\label{splitness_producte_wreath_1}
For any $n\geq 1$ and any split 2-group $\GG$, the wreath 2-product $\Ssf_n\wr\wr\ \GG$ is split.
\end{prop}
\begin{proof}
As pointed out before, equivalent 2-groups have equivalent wreath 2-products. Hence if $\GG$ is split and consequently, $\GG\simeq\KK$ for some strict skeletal 2-group $\KK$ (see Proposition~\ref{G_split_sii_estricte_esqueletic}) we have
$$
\Ssf_n\wr\wr\ \GG\simeq\Ssf_n\wr\wr\ \KK.
$$
Now the wreath 2-product of a strict skeletal 2-group is a strict skeletal 2-group (see Proposition~\ref{S_n_wreath_GG}). Therefore $\Ssf_n\wr\wr\ \GG$ is split. 
\end{proof}

\noindent
In fact, we shall see below that $\Ssf_n\wr\wr\ \GG$ is split only when $\GG$ is split (see \S~\ref{splitness_producte_wreath_2}).

Before going on, let us remark two easy facts about elementary 2-groups that we shall need in what follows. Firstly, their homotopy groups are as expected, i.e.
\begin{align*}
\pi_0(\Asf[1]\rtimes\Gsf[0])&=\Gsf, \\ \pi_1(\Asf[1]\rtimes\Gsf[0])&=\Asf\times\{1\}\cong\Asf
\end{align*}
with the given action of $\Gsf$ on $\Asf$. Thus for any $a\in A$ and any $g\in G$ we have
\begin{align}
\gamma_g(a,1)&=(a,1)\otimes(0,g)=(a,g), \label{gamma_split}\\  \delta_g(a,1)&=(0,g)\otimes(a,1)=(g\lhd a,g),\label{delta_split}
\end{align}
and hence, the canonical left action (\ref{accio}) gives
$$
\gamma_g^{-1}(\delta_g(a,1))=\gamma_g^{-1}(g\lhd a,g)=g\lhd a.
$$
Secondly, elementary 2-groups are completly classified by the homotopy groups. In fact, we have the following explicit description for the morphisms between two elementary 2-groups. It is a generalization of the descriptions given in Examples~\ref{morfismes_entre_2-grups_discrets} and \ref{morfismes_entre_2-grups_un_objecte} of the morphisms between discrete and one-object 2-groups, respectively. 

\begin{lem}
A morphism of elementary 2-groups $\FF:\Asf[1]\rtimes\Gsf[0]\To\Asf'[1]\rtimes\Gsf'[0]$ amounts to a triple $(\rho,\beta,\mathsf{z})$ with $\rho:\Gsf\To\Gsf'$ a group homomorphism, $\beta:\Asf\To\Asf'_\rho$ a morphism of $\Gsf$-modules and $\mathsf{z}:\Gsf\times\Gsf\To\Asf'_\rho$ any 2-cocycle of $\Gsf$ with values in $\Asf'_\rho$ (the abelian group $\Asf'$ equipped with the $\Gsf$-module structure induced by $\rho$ and its $\Gsf'$-module structure). 
\end{lem}
\begin{proof}
Left to the reader.
\end{proof}

\espai 
\noindent
We shall denote by $\FF(\rho,\beta,\mathsf{z})$, or $\FF(\rho,\beta)$ if $\mathsf{z}$ is trivial, the morphism of 2-groups associated to the triple $(\rho,\beta,\mathsf{z})$. It is given by the functor $\Ff(\rho,\beta):\Asf[1]\times\Gsf[0]\To\Asf'[1]\times\Gsf'[0]$ defined by
\begin{align*}
\Ff(\rho,\beta)(g)&:=\rho(g), \\ \Ff(\rho,\beta)(a,g)&:=(\beta(a),\rho(g)),
\end{align*}
and equipped with the monoidal structure $\mu(\rho,\mathsf{z})$ with components
$$
\mu(\rho,\mathsf{z})_{g_1,g_2}:=(\mathsf{z}(g_1,g_2),\rho(g_1g_2)).
$$

\begin{cor}
With the above notations, two elementary 2-groups $\Asf[1]\rtimes\Gsf[0]$ and $\Asf'[1]\rtimes\Gsf'[0]$ are equivalent if and only if there exists an isomorphism of groups $\rho:\Gsf\To\Gsf'$ and an isomorphism of $\Gsf$-modules $\beta:\Asf\To\Asf'_\rho$. In this case, the set of equivalences between both 2-groups is in bijection with the set of all triples $(\rho,\beta,\mathsf{z})$, where $\rho,\beta$ are as before and $\mathsf{z}$ is any 2-cocycle of $\Gsf$ with values in $\Asf'_\rho$. 
\end{cor}
\begin{proof}
The implication to the right follows from the invariance of the homotopy groups. Conversely, given $\rho$ and $\beta$ as in the statement, it is easy to check that the above functor $\Ff(\rho,\beta)$ together with the trivial monoidal structure is a strict monoidal equivalence. As for the last statement, the bijection maps the triple $(\rho,\beta,\mathsf{z})$ to the equivalence $\Ff(\rho,\beta)$ equipped with the non trivial monoidal structure $\mu(\rho,\mathsf{z})$.
\end{proof}

\subsection{First splitness criterion and splitness of the 2-group of permutations of a group}

We saw that an arbitrary 2-group $\GG$ is equivalent to the special 2-group $\pi_1(\GG)[1]\rtimes_{\mathsf{z}}\pi_0(\GG)[0]$, where $\mathsf{z}$ is any classifying normalized 3-cocycle $\mathsf{z}\in\alpha(\GG)$. Hence $\GG$ is split when the Postnikov invariant is $\alpha(\GG)=0$. In fact, the converse is also true. This gives our first criterion of splitness.

\begin{thm} \label{criteri_general}
A 2-group $\GG$ is split if and only if $\alpha(\GG)=0$.
\end{thm}
\begin{proof}
Since $\alpha(\GG)$ is a homotopy invariant of $\GG$, it remains to see that $\alpha(\GG)=0$ when $\GG$ is equal to $\Asf[1]\rtimes\Gsf[0]$ for some group $\Gsf$ and $\Gsf$-module $\Asf$. This follows from Sinh's algorithm for computing $\alpha(\GG)$ from a given {\em \'epinglage} (see \S~\ref{invariants_homotopics}). Indeed, for such a 2-group there is a unique {\em \'epinglage}, given by $s=id_G$ and $\theta_g=id_g$ for all $g\in G$, and the corresponding classifying 3-cocycle is $\mathsf{z}=0$.
\end{proof}

\begin{cor}\label{corolari}
A 2-group $\GG$ is split if and only if it is equivalent to the elementary 2-group $\pi_1(\GG)[1]\rtimes\pi_0(\GG)[0]$, with $\pi_0(\GG)$ acting on $\pi_1(\GG)$ according to (\ref{accio}).
\end{cor} 

As an application, let us consider the 2-group $\SSS ym(\Gsf)$ of permutations of a group $\Gsf$ as a one-object groupoid. Next result follows readily from the description of this 2-group given in Proposition~\ref{2-grup_autoequivalencies_grupoide}. The proof is left to the reader.
 
\begin{prop}\label{exemple_2-grup_split_bis}
Let $\GG=\SSS ym(\Gsf)$ for some group $\Gsf$. Then:
\begin{itemize}
\item[(a)] $\pi_0(\GG)$ is the group $\mathsf{Out}(\Gsf)$ of outer automorphisms of $\Gsf$. 
\item[(b)] $\pi_1(\GG)$ is the center $\Zsf(\Gsf)$ of $\Gsf$, and the action of $\mathsf{Out}(\Gsf)$ on $\Zsf(\Gsf)$ is the natural one, i.e.
$$
[\phi]\lhd z=\phi(z).
$$
\item[(c)] An {\em \'epinglage} of $\GG$ is given by a set theoretic section $s:Out(\Gsf)\To Aut(\Gsf)$ such that $s[id_\Gsf]=id_\Gsf$, together with a map $t:Aut(\Gsf)\To G$ such that
\begin{itemize}
\item[(i)] $t(s[\phi])=e$ for all $[\phi]\in Out(\Gsf)$, and
\item[(ii)] $\phi=c_{t(\phi)}\circ s[\phi]$ for all $\phi\in Aut(\Gsf)$.
\end{itemize}
Moreover, the associated classifying 3-cocycle $\mathsf{z}_{s,t}:Out(\Gsf)^3\To Z(\Gsf)$ is given by
\begin{align}\label{3-cocicle_classificador}
\mathsf{z}_{s,t}&([\phi],[\phi'],[\phi''])= \\ &\ s[\phi](t(s[\phi']\circ s[\phi'']))t(s([\phi][\phi'])\circ s[\phi''])^{-1}t(s[\phi]\circ s([\phi'][\phi'']))t(s[\phi]\circ s[\phi'])^{-1}.\nonumber
\end{align}
\end{itemize}
\end{prop}

Despite the appearance of this formula, it is not clear whether this is a 2-coboundary or not, because $t$ takes values in $G$ but not necessarily in $Z(\Gsf)$. However, there are two particular cases where the 2-group $\SSS ym(\Gsf)$ is clearly split. 

\begin{cor}\label{Sim(A)}
Let $\Asf$ be an abelian group. Then $\SSS ym(\Asf)$ is split. More precisely, we have
$$
\SSS ym(\Asf)\simeq\Asf[1]\rtimes\mathsf{Aut}(\Asf)[0],
$$
the action of $\mathsf{Aut}(\Asf)$ on $\Asf$ being the canonical one.
\end{cor} 
\begin{proof}
If $\Asf$ is abelian, both $\mathsf{Out}(\Asf)$ and $\mathsf{Aut}(\Asf)$ coincide and we can take as section $s$ the identity. Then $t$ takes values in $Z(\Asf)$, and $\mathsf{z}_{id,t}$ is the coboundary of the 2-cochain given by $\mathsf{c}(\phi,\phi')=t(\phi\circ\phi')$. Hence $\alpha(\SSS ym(\Asf))=0$.
\end{proof}

\begin{cor}\label{splits_no_trivials_1}
Let $\Gsf$ be a group with a trivial center. Then $\SSS ym(\Gsf)$ is split.
\end{cor}
\begin{proof}
If $\Zsf(\Gsf)$ is trivial, the 3-cocycle (\ref{3-cocicle_classificador}) is also trivial. Hence $\alpha(\SSS ym(\Gsf))=0$.
\end{proof}
Next Example shows that $\SSS ym(\Gsf)$ may also be a split 2-group even for non-abelian and non-centerless groups $\Gsf$. 

\begin{ex}\label{splits_no_trivials_2}{\rm
Let us consider the dihedral groups
$$
\Dsf_n=\langle r,s\ |\ r^n=s^2=e,sr=r^{-1}s\rangle
$$
with $n$ even. These are non-abelian groups with a non-trivial center $\Zsf(\Dsf_n)=\{e,r^{n/2}\}$. Let us see that $\SSS ym(\Dsf_4)$ and $\SSS ym(\Dsf_6)$ are both split. Indeed, for any $n\geq 2$ the group $\mathsf{Aut}(\Dsf_n)$ is isomorphic to the semidirect product $\ZZ_n\rtimes(\ZZ_n)^\ast$ with $(\ZZ_n)^\ast$ acting on $\ZZ_n$ by multiplication. The pair $(p,q)\in\ZZ_n\times(\ZZ_n)^\ast$ has to be identified to the automorphism $\phi_{p,q}:\Dsf_n\To\Dsf_n$ defined on the generators by
$$
\phi_{p,q}(r)=r^q, \qquad \phi_{p,q}(s)=sr^p.
$$
Moreover, $\mathsf{Out}(\Dsf_n)$ is of order 2 for $n=4,6$. Therefore, up to congugation, there is a unique non inner automorphism in these dihedral groups, for instance $\phi_{1,1}$. Its square is trivial in $\mathsf{Out}(\Dsf_n)$. Since the classifying 3-cocycle $\mathsf{z}:Out(\Dsf_n)^3\To Z(\Dsf_n)$ is normalized, this means that there is a unique triple where $\mathsf{z}$ can take a non-zero value, namely $([\phi_{1,1}],[\phi_{1,1}],[\phi_{1,1}])$. For such a triple, Equation~(\ref{3-cocicle_classificador}) gives
$$
\mathsf{z}_{s,t}([\phi_{1,1}],[\phi_{1,1}],[\phi_{1,1}])=s[\phi_{1,1}](t(s[\phi_{1,1}]^2))t(s[\phi_{1,1}]^2)^{-1}.
$$
It is then an easy direct computation checking that for any $\phi\in[\phi_{1,1}]$ and any map $t$ as before we have $\phi(t(\phi^2))=t(\phi^2)$. Therefore the 3-cocycle (\ref{3-cocicle_classificador}) is always trivial for any choices of $s$ and $t$. For instance, in the case of $\Dsf_6$ we have
$$
[\phi_{1,1}]=\{\phi_{1,1},\phi_{3,1},\phi_{5,1},\phi_{1,5},\phi_{3,5},\phi_{5,5}\}.
$$
In all but the first and third automorphisms, the corresponding square is equal to the identity, so that $t(\phi^2)=e$, and this clearly remains fixed by $\phi$. On the other hand, $\phi_{1,1}^2$ is conjugation by $r^2$ or $r^5$, and $\phi_{5,1}^2$ is conjugation by $r$ or $r^3$, and all powers of $r$ remain fixed both by $\phi_{1,1}$ and by $\phi_{5,1}$.  }
\end{ex}

\subsection{More splitness criteria.}

For any 2-group $\GG$ we denote by $p:G_0\To\pi_0(\GG)$ the projection mapping each object of $\Gg$ to its isomorphisms class. It is a group homomorphism when $\GG$ is strict. Let us start with the following observations.

\begin{lem} \label{equivalencia_induida}
For any 2-group $\GG$, any set theoretic section $s:\pi_0(\GG)\To G_0$ of $p$ induces an equivalence of groupoids
$$
\Ee_s:\pi_1(\GG)[1]\times\pi_0(\GG)[0]\To\Gg
$$
given on objects $[x]$ and morphisms $(u,[x])$ by
\begin{align}
\Ee_s[x]&:=s[x], \label{def_E_s_1} \\
\Ee_s(u,[x])&:=\gamma_{s[x]}(u).\label{def_E_s_2}
\end{align}
\end{lem}
\begin{proof}
Straightforward.
\end{proof}

\noindent
There is a similar induced equivalence $\Ee'_s$ acting on objects as $\Ee_s$ and on morphisms by
$$
\Ee_s'(u,[x])=\delta_{s[x]}(u).
$$
However, for our purposes it is enough to consider the ``left'' equivalences $\Ee_s$.

\begin{lem}\label{splita_si_E_s_monoidal}
A 2-group $\GG$ is split if and only if for some set theoretic section $s:\pi_0(\GG)\To G_0$ of $p$ the induced equivalence $\Ee_s$ extends to an equivalence of 2-groups $\EE_s:\pi_1(\GG)[1]\rtimes\pi_0(\GG)[0]\To\GG$.
\end{lem}
\begin{proof}
The implication to the left is obvious. Conversely, let us suppose that $\GG$ is split. By Corollary~\ref{corolari}, $\GG$ is necessarily equivalent to $\pi_1(\GG)[1]\rtimes\pi_0(\GG)[0]$. Let us choose any equivalence of 2-groups $\EE=(\Ee,\mu):\pi_1(\GG)[1]\rtimes\pi_0(\GG)[0]\To\GG$, and let $\rho=\pi_0(\EE)$ and $\beta=\pi_1(\EE)$ be the associated automorphisms of $\pi_0(\GG)$ and $\pi_1(\GG)$, respectively given by (\ref{morfisme1}) and (\ref{morfisme2}). Then a new equivalence
$$
\EE'=(\Ee',\mu'):\pi_1(\GG)[1]\rtimes\pi_0(\GG)[0]\To\GG
$$
is obtained by precomposing $\EE$ with the self-equivalence $\EE(\rho^{-1},\beta^{-1})$ of $\pi_1(\GG)[1]\rtimes\pi_0(\GG)[0]$. Let us check that $\Ee'=\Ee_s$ for some set theoretic section $s$ of $p$. Indeed, on objects $\Ee'$ is given by $[x]\mapsto\Ee(\rho^{-1}[x])$, and by definition of $\rho$
$$
[\Ee(\rho^{-1}[x])]=\rho(\rho^{-1}[x])=[x].
$$
Hence $\Ee'$ maps each class $[x]$ to some representative of it. Let $s:\pi_0(\GG)\To G_0$ be the set theoretic section so defined, i.e.
$$
s[x]:=\Ee'[x]=\Ee(\rho^{-1}[x]).
$$ 
As for the action on morphisms, if $\{\gamma_x\}_{x\in\Gsf_0}$ denote the canonical isomorphisms of $\GG$ and $\{\hat{\gamma}_{[x]}\}_{[x]\in\pi_0(\GG)}$ those of $\pi_1(\GG)[1]\rtimes\pi_0(\GG)[0]$ we have
\begin{align*}
\Ee'_{[x],[x]}(u,[x])&\ \ =\ \ \Ee_{\rho^{-1}[x],\rho^{-1}[x]}(\beta^{-1}(u),\rho^{-1}[x]) \\ &\stackrel{(\ref{preservacio_gamma_delta})}{=}(\gamma_{\Ee(\rho^{-1}[x])}\circ\beta\circ\hat{\gamma}^{-1}_{\rho^{-1}[x]})(\beta^{-1}(u),\rho^{-1}[x]) \\ &\ \stackrel{(\ref{gamma_split})}{=}\ (\gamma_{s[x]}\circ\beta)(\beta^{-1}(u)) \\ &\ \ \ =\ \ \gamma_{s[x]}(u) \\ &\ \ \ =\ \ \Ee_s(u,[x]).
\end{align*}
Therefore, $\Ee'=\Ee_s$ and $\Ee_s$ indeed extends to the equivalence of 2-groups $\EE_s=\EE'$.
\end{proof}

\begin{lem}\label{independencia_seccio}
If for some set theoretic section $s:\pi_0(\GG)\To G_0$ of $p$ the equivalence $\Ee_{s}$ extends to an equivalence of 2-groups $\EE_s:\pi_1(\GG)[1]\rtimes\pi_0(\GG)[0]\To\GG$, then for any other section $s'$ the equivalence $\Ee_{s'}$ also extends to an equivalence of 2-groups $\EE_{s'}:\pi_1(\GG)[1]\rtimes\pi_0(\GG)[0]\To\GG$.
\end{lem}
\begin{proof}
Let $\mu_{[x],[y]}:s[x]\otimes s[y]\To s[x\otimes y]$ be a monoidal structure on $\Ee_s$ which makes it an equivalence of 2-groups between $\pi_1(\GG)[1]\rtimes\pi_0(\GG)[0]$ and $\GG$. For any other section $s'$ of $p$ and any isomorphisms $\theta_{[x]}:s[x]\To s'[x]$, let us consider the isomorphisms
$$
\mu'_{[x],[y]}:=\theta_{[x\otimes y]}\circ\mu_{[x],[y]}\circ(\theta^{-1}_{[x]}\otimes\theta^{-1}_{[y]}).
$$
For any morphisms $(u,[x])$ and $(v,[y])$ we have
\begin{align*}
\Ee_{s'}((u,[x])\otimes(v,[y]))\circ\mu'_{[x],[y]}&\ \ =\ \Ee_{s'}(u\circ([x]\lhd v),[x\otimes y])\circ\mu'_{[x],[y]} \\ &\stackrel{(\ref{def_E_s_2})}{=}\gamma_{s'[x\otimes y]}(u\circ([x]\lhd v))\circ\theta_{[x\otimes y]}\circ\mu_{[x],[y]}\circ(\theta^{-1}_{[x]}\otimes\theta^{-1}_{[y]}) \\ &\stackrel{(\ref{naturalitat_gamma_delta})}{=}\theta_{[x\otimes y]}\circ\gamma_{s[x\otimes y]}(u\circ([x]\lhd v))\circ\mu_{[x],[y]}\circ(\theta^{-1}_{[x]}\otimes\theta^{-1}_{[y]}) \\ &\stackrel{(\ref{def_E_s_2})}{=}\theta_{[x\otimes y]}\circ\Ee_s(u\circ([x]\lhd v),[x\otimes y])\circ\mu_{[x],[y]}\circ(\theta^{-1}_{[x]}\otimes\theta^{-1}_{[y]})  \\ &\ \ =\ \theta_{[x\otimes y]}\circ\mu_{[x],[y]}\circ(\Ee_s(u,[x])\otimes\Ee_s(v,[y]))\circ(\theta^{-1}_{[x]}\otimes\theta^{-1}_{[y]}) \\ &\stackrel{(\ref{def_E_s_2})}{=}\theta_{[x\otimes y]}\circ\mu_{[x],[y]}\circ(\gamma_{s[x]}(u)\otimes\gamma_{s[y]}(v))\circ(\theta^{-1}_{[x]}\otimes\theta^{-1}_{[y]}) \\ &\stackrel{(\ref{naturalitat_gamma_delta})}{=}\theta_{[x\otimes y]}\circ\mu_{[x],[y]}\circ(\theta^{-1}_{[x]}\otimes\theta^{-1}_{[y]})\circ(\gamma_{s'[x]}(u)\otimes\gamma_{s'[y]}(v)) \\ &\stackrel{(\ref{def_E_s_2})}{=}\mu'_{[x],[y]}\circ(\Ee_{s'}(u,[x])\otimes\Ee_{s'}(v,[y]))
\end{align*}
where in the first and fifth equalities we have used the formula for the tensor product of morphisms in an elementary 2-group (see Proposition~\ref{A[1]_rtimes_G[0]}) and the naturality of $\mu_{[x],[y]}$, respectively. Hence, the isomorphisms $\mu'_{[x],[y]}$ are natural in $[x],[y]$. Moreover, the functoriality of $\otimes$ along with the coherence of $\mu_{[x],[y]}$ imply that they are coherent:
\begin{align*}
\mu'_{[x],[y\otimes z]}\circ(id_{s'[x]}\otimes\mu'_{[y],[z]})&=\theta_{[x\otimes y\otimes z]}\circ\mu_{[x],[y\otimes z]}\circ(\theta^{-1}_{[x]}\otimes\theta^{-1}_{[y\otimes z]}) \\ &\hspace{1truecm}\circ(id_{s'[x]}\otimes\theta_{[y\otimes z]})\circ(id_{s'[x]}\otimes\mu_{[y],[z]})\circ(id_{s'[x]}\otimes\theta^{-1}_{[y]}\otimes\theta^{-1}_{[z]}) \\ &=\theta_{[x\otimes y\otimes z]}\circ\mu_{[x],[y\otimes z]}\circ(id_{s[x]}\otimes\mu_{[y],[z]})\circ(\theta^{-1}_{[x]}\otimes\theta^{-1}_{[y]}\otimes\theta^{-1}_{[z]})  \\ &=\theta_{[x\otimes y\otimes z]}\circ\mu_{[x\otimes y],[z]}\circ(\mu_{[x],[y]}\otimes id_{s[z]})\circ(\theta^{-1}_{[x]}\otimes\theta^{-1}_{[y]}\otimes\theta^{-1}_{[z]}) \\ &=\mu'_{[x\otimes y],[z]}\circ(\theta_{[x\otimes y]}\otimes\theta_{[z]})\circ(\mu_{[x],[y]}\otimes id_{s[z]})\circ(\theta^{-1}_{[x]}\otimes\theta^{-1}_{[y]}\otimes\theta^{-1}_{[z]}) \\ &=\mu'_{[x\otimes y],[z]}\circ(\theta_{[x\otimes y]}\otimes id_{s'[z]})\circ(\mu_{[x],[y]}\otimes id_{s'[z]})\circ(\theta^{-1}_{[x]}\otimes\theta^{-1}_{[y]}\otimes id_{s'[z]}) \\ &=\mu'_{[x\otimes y],[z]}\circ(\mu'_{[x],[y]}\otimes id_{s'[z]}).
\end{align*}
Hence the isomorphisms $\mu'_{[x],[y]}$ provide a monoidal structure on $\Ee_{s'}$ making it an equivalence of 2-groups as in the statement.  
\end{proof}

It follows from the previous two Lemmas that we can restrict the attention to {\em normalized} sections, i.e. sections $s$ such that $s[e]=e$. Moreover, we have the following.

\begin{cor}\label{cor_criteri_escissio}
Let $\GG$ be an arbitrary 2-group. Then $\GG$ is split if and only if for any (normalized) set theoretic section $s:\pi_0(\GG)\To G_0$ of $p$ the associated equivalence $\Ee_s$ extends to an equivalence of 2-groups $\EE_s:\pi_1(\GG)[1]\rtimes\pi_0(\GG)[0]\To\GG$.
\end{cor}

We may now prove the following two additional splitness criteria. The second one justifies the name ``split'' for this kind of 2-groups.

\begin{thm}\label{criteri_general_2}
Let $\GG$ be any 2-group. The following conditions are equivalent.
\begin{itemize}
\item[(1)] $\GG$ is split.
\item[(2)] There exists a (normalized) section $s:\pi_0(\GG)\To G_0$ of $p$ and a collection of isomorphisms
$$
\{\mu_{[x],[x']}:s[x]\otimes s[x']\To s[x\otimes x']\}_{[x],[x']\in\pi_0(\GG)}
$$
such that
\begin{equation} \label{axioma_coherencia_amb_s}
\xymatrix{
s[x]\otimes(s[x']\otimes s[x''])\ar[dd]_{id\otimes'\mu_{[x'],[x'']}}\ar[dr]^{\ \ \ \ a_{s[x],s[x'],s[x'']}\ \ \ } & \\ & (s[x]\otimes s[x'])\otimes s[x'']\ar[dd]^{\mu_{[x],[x']}\otimes id}
   \\ s[x]\otimes s[x'\otimes x'']\ar[dd]_{\mu_{[x],[x'\otimes x'']}} &  \\ & s[x\otimes x']\otimes s[x'']\ar[dl]^{\ \ \ \ \mu_{[x\otimes x'],[x'']}} \\ s[x\otimes x'\otimes x''] & }
\end{equation}
commutes for all $[x],[x'],[x'']\in\pi_0(\GG)$.
\item[(3)] The Postnikov exact sequence of $\GG$ 
$$
\mathbbm{1}\longrightarrow\pi_1(\GG)[1]\stackrel{\JJ}{\longrightarrow}\GG\stackrel{\PP}{\longrightarrow}\pi_0(\GG)[0]\longrightarrow\mathbbm{1}
$$
``splits'', i.e. there exists a (unitary) morphism of 2-groups $\SSS:\pi_0(\GG)[0]\To\GG$ such that $\PP\circ\SSS\cong id_{\pi_0(\GG)[0]}$. 
\end{itemize}
\end{thm}
\begin{proof}
It easily follows from (\ref{propietat}) that any family of isomorphisms $s[x]\otimes s[x']\cong s[x\otimes x']$ defines a natural isomorphism between $\otimes\circ(\Ee_s\times\Ee_s)$ and $\Ee_s\circ\otimes$. Moreover, (\ref{axioma_coherencia_amb_s}) is just (\ref{axioma_coherencia}) with $\Ff=\Ee_s$. Hence, a family as in item (2) is the same thing as a monoidal structure on $\Ee_s$ making it an equivalence of 2-groups between $\pi_1(\GG)[1]\rtimes\pi_0(\GG)[0]$ and $\GG$. The equivalence (1)$\Leftrightarrow$(2) follows then from Lemma~\ref{splita_si_E_s_monoidal} (and Lemma~\ref{independencia_seccio} in the normalized case). As for the equivalence (2)$\Leftrightarrow$(3), it is enough to observe that a (unitary) morphism $\SSS$ as in item (3) is exactly the same thing as a (normalized) section $s$ of $p$ together with a family of isomorphisms as in item (2). 
\end{proof}

\begin{cor}\label{criteri_suficient_split}
Let $\GG$ be a strict 2-group. If there exists a set theoretic section $s:\pi_0(\GG)\To G_0$ of the projection map $p:\Gsf_0\To\pi_0(\GG)$ which is a homomorphism of groups, then $\GG$ is split.
\end{cor}
\begin{proof}
Let $s$ be a homomorphism of groups, and take as $\mu_{[x],[x']}$ the corresponding identity. Then (\ref{axioma_coherencia_amb_s}) clearly commutes when the associator is trivial.
\end{proof}

\medskip
It is worth emphasizing that this corollary gives a sufficient but not necessary condition for a strict 2-group $\GG$ to be split. In fact, the condition corresponds to the existence of a set theoretic section $s$ such that $\Ee_s$ is a {\em strict} monoidal functor. However, a functor between monoidal groupoids may be a monoidal functor without being strictly monoidal.

\begin{ex}{\rm
Since $\Dsf_5$ is centerless, we know from Corollary~\ref{splits_no_trivials_1} that $\SSS ym(\Dsf_5)$ is split. However, the projection map $\mathsf{Aut}(\Dsf_5)\To\mathsf{Out}(\Dsf_5)$ has no section which is a group homomorphism. Indeed, a direct computation shows that the group of outer automorphisms is
$$
\mathsf{Out}(\Dsf_5)=\{[\phi_{0,1}],[\phi_{0,2}]\}\cong\ZZ_2
$$
with
\begin{align*}
[\phi_{0,1}]&=\{\phi_{0,1},\phi_{1,1},\phi_{2,1},\phi_{3,1},\phi_{4,1},\phi_{0,4},\phi_{1,4},\phi_{2,4},\phi_{3,4},\phi_{4,4}\}, \\
[\phi_{0,2}]&=\{\phi_{0,2},\phi_{1,2},\phi_{2,2},\phi_{3,2},\phi_{4,2},\phi_{0,3},\phi_{1,3},\phi_{2,3},\phi_{3,3},\phi_{4,3}\}. 
\end{align*}
Hence, if there exists a section $s$ which is a group homomorphism, it sends $[\phi_{0,2}]$ to some representative of order two. But there is no such representative.}
\end{ex}

Later on we shall give a slightly weakened version of this condition which is also necessary (see Theorem~\ref{criteri_split_2-grups_estrictes}).

\subsection{Splitness of a wreath 2-product}
\label{splitness_producte_wreath_2}
We know from Proposition~\ref{splitness_producte_wreath_1} that $\Ssf_n\wr\wr\ \GG$ is split when $\GG$ is split. Let us now prove the converse using Theorem~\ref{criteri_general_2}.

Let us first observe that the first homotopy group of $\Ssf_n\wr\wr\ \GG$ is
\begin{equation}\label{pi0_producte_wreath}
\pi_0(\Ssf_n\wr\wr\ \GG)\cong\Ssf_n\wr\pi_0(\GG).
\end{equation}
This is an immediate consequence of Proposition~\ref{S_n_wreath_GG} and the fact that two objects $(\sigma,\mathbf{x}),(\sigma',\mathbf{x}')$ are isomorphic if and only if $\sigma=\sigma'$ and $\mathbf{x}\cong\mathbf{x}'$. We shall denote by
$$
\mathbf{p}:S_n\times G_0^n\To S_n\times\pi_0(\GG)^n
$$
the corresponding projection map. We have
$$
\mathbf{p}=id_{S_n}\times p\times\cdots\times p,
$$
where $p:G_0\To\pi_0(\GG)$ is the projection map for $\GG$. On the other hand, the automorphism group of the unit object is
$$
\mathsf{Aut}(id_n,\mathbf{e})=\{id_{id_n}\}\times\mathsf{Aut}(\mathbf{e})\cong\mathsf{Aut}(e)^n.
$$
Hence
\begin{equation}\label{pi1_producte_wreath}
\pi_1(\Ssf_n\wr\wr\ \GG)\cong\pi_1(\GG)^n.
\end{equation}
Although we shall not need it right now, let us remark that for a strict 2-group $\GG$ the canonical left action of $\Ssf_n\wr\pi_0(\GG)$ on this group is given by
\begin{equation}\label{accio_producte_wreath}
(\sigma,[\mathbf{x}])\lhd\mathbf{u}=(id_{\mathbf{x}}\otimes\mathbf{u}\otimes id_{\mathbf{x}^{-1}})\rhd\sigma^{-1}.
\end{equation}
Indeed, Proposition~\ref{S_n_wreath_GG} implies that $\Ssf_n[0]\wr\wr\ \GG$ is strict when $\GG$ is strict. Moreover, the explicit description of $\Ssf_n[0]\wr\wr\ \GG$ given in this Proposition also implies that
\begin{align*}
(\sigma,[\mathbf{x}])\lhd(id_{id_n},\mathbf{u})\ &\stackrel{(\ref{accio_cas_estricte})}{=}id_{(\sigma,\mathbf{x})}\otimes(id_{id_n},\mathbf{u})\otimes id_{(\sigma,\mathbf{x})^{-1}} \\  &\ \ =\ (id_\sigma,id_\mathbf{x})\otimes(id_{id_n},\mathbf{u})\otimes (id_{\sigma^{-1}},id_{\mathbf{x}^{-1}\rhd\sigma^{-1}}) \\ &\ \ =\ (id_\sigma,(id_\mathbf{x}\rhd id_n)\otimes\mathbf{u})\otimes (id_{\sigma^{-1}},id_{\mathbf{x}^{-1}\rhd\sigma^{-1}}) \\ &\ \ =\ (id_{id_n},((id_\mathbf{x}\otimes\mathbf{u})\rhd\sigma^{-1})\otimes id_{\mathbf{x}^{-1}\rhd\sigma^{-1}})  \\ &\ \ =\ (id_{id_n},(id_\mathbf{x}\otimes\mathbf{u}\otimes id_{\mathbf{x}^{-1}})\rhd\sigma^{-1}). 
\end{align*}
In the last equality we have used that the wreath 2-action is strict and the appropriate naturality diagram.

Let us now suppose that $\Ssf_n\wr\wr\ \GG$ is split. According to item (2) in Theorem~\ref{criteri_general_2}, this means that there exists a section $\mathbf{s}$ of $\mathbf{p}$, necessarily of the form 
$$
\mathbf{s}=id_{S_n}\times s_1\times\cdots\times s_n
$$
for some sections $s_1,\ldots,s_n$ of $p$, and a family of isomorphisms in $S_n[0]\times\Gg^n$
$$
\{\mu_{(\sigma,[\mathbf{x}]),(\sigma',[\mathbf{x}'])}:\mathbf{s}(\sigma,[\mathbf{x}])\otimes\mathbf{s}(\sigma,[\mathbf{x}])\To\mathbf{s}(\sigma\sigma',[(\mathbf{x}\rhd\sigma')\otimes\mathbf{x}'])\}_{(\sigma,[\mathbf{x}]),(\sigma',[\mathbf{x}'])}
$$
such that
\begin{equation*}
\xymatrix{
\mathbf{s}(\sigma,[\mathbf{x}])\otimes(\mathbf{s}(\sigma',[\mathbf{x}'])\otimes \mathbf{s}(\sigma'',[\mathbf{x}'']))\ar[dd]_{id\otimes'\mu_{(\sigma',[\mathbf{x}']),(\sigma'',[\mathbf{x}''])}}\ar[dr]^{\hspace{1.8truecm}a_{\mathbf{s}(\sigma,[\mathbf{x}]),\mathbf{s}(\sigma',[\mathbf{x}']),\mathbf{s}(\sigma'',[\mathbf{x}''])}\ \ \ } & \\ & (\mathbf{s}(\sigma,[\mathbf{x}])\otimes \mathbf{s}(\sigma',[\mathbf{x}']))\otimes \mathbf{s}(\sigma'',[\mathbf{x}''])\ar[dd]^{\mu_{(\sigma,[\mathbf{x}]),(\sigma',[\mathbf{x}'])}\otimes id}
   \\ \mathbf{s}(\sigma,[\mathbf{x}])\otimes \mathbf{s}(\sigma'\sigma'',[(\mathbf{x}'\rhd\sigma'')\otimes \mathbf{x}''])\ar[dd]_{\mu_{(\sigma,[\mathbf{x}]),(\sigma'\sigma'',[(\mathbf{x}'\rhd\sigma'')\otimes \mathbf{x}''])}} & \\ &  \mathbf{s}(\sigma\sigma',[(\mathbf{x}\rhd\sigma')\otimes \mathbf{x}'])\otimes \mathbf{s}(\sigma'',[\mathbf{x}''])\ar[dl]^{\hspace{1.8truecm}\mu_{(\sigma\sigma',[(\mathbf{x}\rhd\sigma')\otimes \mathbf{x}']),(\sigma'',[\mathbf{x}''])}} \\ \mathbf{s}(\sigma\sigma'\sigma'',[(\mathbf{x}\rhd(\sigma'\sigma''))\otimes((\mathbf{x}'\rhd\sigma'')\otimes \mathbf{x}'')]) & }
\end{equation*} 
commutes for all $(\sigma,[\mathbf{x}]),(\sigma',[\mathbf{x}']),(\sigma'',[\mathbf{x}''])\in S_n\times\pi_0(\GG)^n$. Let us focus the attention on the subfamily $\{\mu_{(id_n,[\mathbf{x}]),(id_n,[\mathbf{x}'])}\}_{[\mathbf{x}],[\mathbf{x}']}$. For short, we shall write
\begin{align*}
\tilde{\mathbf{s}}&\equiv s_1\times\cdots\times s_n,
\\ \tilde{\mathbf{s}}[\mathbf{x}]&\equiv (s_1[x_1],\ldots,s_n[x_n]).
\end{align*}
Then it follows from Proposition~\ref{S_n_wreath_GG} that these isomorphisms are of the form
$$
\mu_{(id_n,[\mathbf{x}]),(id_n,[\mathbf{x}'])}=(id_{id_n},\tilde{\mu}_{[\mathbf{x}],[\mathbf{x}']})
$$
for some isomorphisms $\tilde{\mu}_{[\mathbf{x}],[\mathbf{x}']}:\tilde{\mathbf{s}}[\mathbf{x}]\otimes\tilde{\mathbf{s}}[\mathbf{x}']\To\tilde{\mathbf{s}}[\mathbf{x}\otimes\mathbf{x}']$ in $\Gg^n$ making
\begin{equation*}
\xymatrix{
\tilde{\mathbf{s}}[\mathbf{x}]\otimes(\tilde{\mathbf{s}}[\mathbf{x}']\otimes \tilde{\mathbf{s}}[\mathbf{x}''])\ar[dd]_{id\otimes'\tilde{\mu}_{[\mathbf{x}'],[\mathbf{x}'']}}\ar[dr]^{\ \ \ \ a_{\tilde{\mathbf{s}}[\mathbf{x}],\tilde{\mathbf{s}}[\mathbf{x}'],\tilde{\mathbf{s}}[\mathbf{x}'']}\ \ \ } & \\ & (\tilde{\mathbf{s}}[\mathbf{x}]\otimes \tilde{\mathbf{s}}[\mathbf{x}'])\otimes \tilde{\mathbf{s}}[\mathbf{x}'']\ar[dd]^{\tilde{\mu}_{[\mathbf{x}],[\mathbf{x}']}\otimes id}
   \\ \tilde{\mathbf{s}}[\mathbf{x}]\otimes \tilde{\mathbf{s}}[\mathbf{x}'\otimes \mathbf{x}'']\ar[dd]_{\tilde{\mu}_{[\mathbf{x}],[\mathbf{x}'\otimes \mathbf{x}'']}} &  \\ & \tilde{\mathbf{s}}[\mathbf{x}\otimes \mathbf{x}']\otimes \tilde{\mathbf{s}}[\mathbf{x}'']\ar[dl]^{\ \ \ \ \tilde{\mu}_{[\mathbf{x}\otimes \mathbf{x}'],[\mathbf{x}'']}} \\ \tilde{\mathbf{s}}[\mathbf{x}\otimes \mathbf{x}'\otimes \mathbf{x}''] & }
\end{equation*}
commute for all $[\mathbf{x}],[\mathbf{x}'],[\mathbf{x}'']\in\pi_0(\GG)^n$. Since the monoidal structure on $\Gg^n$ is given componentwise, their components
$$
\tilde{\mu}^{(i)}_{[\mathbf{x}],[\mathbf{x}']}:s_i[x_i]\otimes s_i[x'_i]\To s_i[x_i\otimes x'_i],\qquad i=1,\ldots,n
$$
make the analogous diagrams commute in $\Gg$. By item (2) of Theorem~\ref{criteri_general_2}, we conclude that $\GG$ is split. Hence we have proved the following.

\begin{prop}\label{caracter_split_producte_wreath}
Let $\GG$ be any 2-group. Then for any $n\geq 1$ the wreath 2-product $\Ssf_n\wr\wr\ \GG$ is split if and only if $\GG$ is split.
\end{prop}

\subsection{Splitness criterion for strict 2-groups.}
\label{criteri_2-grups_estrictes}

In spite of their generality, the previous conditions of splitness are sometimes difficult to check. For instance, showing that a given cocycle is cohomologically trivial is in general a difficult task. In this subsection we state and prove a new criterion valid only for strict 2-groups. This criterion is used in the next section to prove the existence of non split permutation 2-groups. Unless otherwise indicated, throughout this subsection $\GG$ stands for a strict 2-group.

Theorem~\ref{criteri_general_2} reduces the problem of determining when a 2-group is split to that of determining when a (normalized) set theoretic section $s$ of $p$ admits a family of isomorphisms
$$
\mu_{[x],[x']}:s[x]\otimes s[x']\To s[x\otimes x']
$$
making commutative some diagrams. The answer we shall now give involves a canonical exact sequence of groups associated to any strict 2-group $\GG$. It is the sequence obtained as follows.

Let $H_0$ be the set of all morphisms of $\Gg$ whose domain is the unit object $e$. Since $\GG$ is strict, it becomes a group $\Hsf_0$ with the product given by
$$
(e\stackrel{f}{\longrightarrow}x)\cdot(e\stackrel{f'}{\longrightarrow}x'):=
(e\stackrel{f\otimes f'}{\longrightarrow}x\otimes x).
$$
Moreover, the map $t:H_0\To G_0$ sending any $f:e\To x$ to its codomain is a group homomorphism. Its kernel is the homotopy group $\pi_1(\GG)$, and its image is the subgroup of all objects $x\in G_0$ isomorphic to $e$, i.e. the kernel $\Ksf_0$ of the projection homomorphism $p:\Gsf_0\To\pi_0(\GG)$. Thus we get the exact sequence of groups  
\begin{equation}\label{successio_exacta}
\xymatrix{
0\ar[r] & \pi_1(\GG)\ar[r] & \Hsf_0\ar[r]^t & \Gsf_0\ar[r]^{p\ \ \ } & \pi_0(\GG)\ar[r] & 1 \ .}
\end{equation}
We shall call it the {\em exact 4-sequence} of $\GG$.

\begin{ex}{\rm
We know that $\Ssf_n\wr\wr\ \GG$ is a strict 2-group when $\GG$ is strict. If (\ref{successio_exacta}) is the exact 4-sequence of $\GG$, the exact 4-sequence of $\Ssf_n\wr\wr\ \GG$ is
\begin{equation*}
\xymatrix{
0\ar[r] & \pi_1(\GG)^n\ar[r] & \Hsf^n_0\ar[r]^{\mathbf{t}\ \ \ } & \Ssf_n\wr\Gsf_0\ar[r]^{\mathbf{p}\ \ \ } & \Ssf_n\wr\pi_0(\GG)\ar[r] & 1\ , }
\end{equation*} 
where $\mathbf{t}$ and $\mathbf{p}$ denote the group homomorphisms given by
\begin{align*}
&\mathbf{t}(\mathbf{f}):=(id_n,(t(f_1),\ldots,t(f_n))), \\ &\mathbf{p}(\sigma,\mathbf{x})):=(\sigma,([x_1],\ldots,[x_n]).
\end{align*}
Indeed, we know from (\ref{pi0_producte_wreath}) and (\ref{pi1_producte_wreath}) that the first and second homotopy groups of $\Ssf_n\wr\wr\ \GG$ are $\Ssf_n\wr\pi_0(\GG)$ and $\pi_1(\GG)^n$, respectively. Moreover, it follows from Proposition~\ref{S_n_wreath_GG} that the group of objects of $\Ssf_n\wr\wr\ \GG$ is the wreath product $\Ssf_n\wr\Gsf_0$. Finally, if $K_0$ is the subset of objects of $\Gg$ isomorphic to $e$, the set of morphisms in $S_n[0]\times\Gg^n$ with domain $(id_n,\mathbf{e})$ is
\begin{align*}
\coprod_{\mathbf{x}\in K_0^n}\mathrm{Hom}((id_n,\mathbf{e}),(id_n,\mathbf{x}))&\cong\coprod_{\mathbf{x}\in K_0^n}\left(\prod_{i=1}^n\mathrm{Hom}(e,x_i)\right)
\\ &\cong\left(\coprod_{x\in K_0}\mathrm{Hom}(e,x)\right)^n \\ &\cong H_0^n.
\end{align*}
By Proposition~\ref{S_n_wreath_GG}, this is an isomorphism of groups. }
\end{ex}

\begin{ex}\label{s.e._Equiv(G[1])}{\rm
The exact 4-sequence of the strict 2-group $\SSS ym(\Gsf)$ is isomorphic to the sequence
\begin{equation*}
\xymatrix{
0\ar[r] & \Zsf(\Gsf)\ar[r] & \Gsf\ar[r]^{c\ \ \ \ \ } & \mathsf{Aut}(\Gsf)\ar[r]^p & \mathsf{Out}(\Gsf)\ar[r] & \mathsf{1}\ ,}
\end{equation*}
where $c:\Gsf\To\mathsf{Aut}(\Gsf)$ is the map given by $g\mapsto c_g$. Indeed, by Propositions~\ref{2-grup_autoequivalencies_grupoide} and \ref{exemple_2-grup_split_bis} its group of objects is isomorphic to $\mathsf{Aut}(\Gsf)$, and the first and second homotopy groups are respectively isomorphic to $\mathsf{Out}(\Gsf)$ and $\Zsf(\Gsf)$. As for the group of morphisms with source the unit object $id_{\Gsf}$, its underlying set is
$$
\coprod_{\phi\in\mathrm{Inn}(\Gsf)}\mathrm{Hom}(\Ee(id_G),\Ee(\phi))=\coprod_{\phi\in\mathrm{Inn}(\Gsf)}\{\tau(g;id_G,\phi),\ \forall g\in G\ \mbox{s.t.}\ \phi=c_g\}\cong G.
$$
By Proposition~\ref{2-grup_autoequivalencies_grupoide}, this is an isomorphism of groups. We leave to the reader checking that the map $t$ reduces to the above map $c$.
}
\end{ex}

Let us now consider any set theoretic section $s$ of $p$. We can assume without loss of generality that $s$ is normalized. Then we have the induced map
$$
\hat{s}:\pi_0(\GG)\times\pi_0(\GG)\To K_0
$$
defined by
\begin{equation}\label{definicio_hat_s}
\hat{s}([x],[x']):=s[x]\otimes s[x']\otimes s[x\otimes x']^{-1},
\end{equation}
and measuring the failure of $s$ to be a homomorphism of groups. It is also normalized, i.e. such that
$$
\hat{s}([x],[x'])=e
$$
when $[x]$, $[x']$ or both are equal to $[e]$. Since $\hat{s}$ takes values in $K_0$, it can always be lifted to a map
$$
\psi_s:\pi_0(\GG)\times\pi_0(\GG)\To H_0.
$$
Such a lifting will be called {\em normalized} when
$$
\psi_s([e],[x])=\psi_s([x],[e])=id_e
$$
for all $[x]\in\pi_0(\GG)$. Moreover, the section $s$ induces a left ``action'' $\lhd_s$ of $\pi_0(\GG)$ on $H_0$ given by
\begin{equation*}
[x']\lhd_s f:=id_{s[x']}\otimes f\otimes id_{s[x']^{-1}}.
\end{equation*}
It is an action only in a weak sense. Indeed, $[e]\lhd_s f$ is indeed equal to $f$ but $[x']\lhd_s([x'']\lhd_s f)$ is equal to $[x'\otimes x'']\lhd_s f$ only up to composition with an isomorphism.

Then we have the following splitness criterion for strict 2-groups.

\begin{thm}\label{criteri_split_2-grups_estrictes}
Let $\GG$ be a strict 2-group. Then $\GG$ is split if and only if there exists a normalized set theoretic section $s:\pi_0(\GG)\To G_0$ of $p$ such that the induced map $\hat{s}:\pi_0(\GG)\times\pi_0(\GG)\To K_0$ defined by (\ref{definicio_hat_s}) admits a normalized lifting $\psi_s:\pi_0(\GG)\times\pi_0(\GG)\To H_0$ satisfying the ``2-cocycle condition''
\begin{equation}\label{condicio_2-cocicle}
\psi_s([x],[x'])\cdot\psi_s([x\otimes x'],[x''])=([x]\lhd_s\psi_s([x'],[x'']))\cdot\psi_s([x],[x'\otimes x''])
\end{equation}
for all $[x],[x'],[x'']\in\pi_0(\GG)$.
\end{thm}
\begin{proof}
Let $\GG$ be split. By Theorem~\ref{criteri_general_2}, there exists a normalized section $s$ and a family of isomorphisms $\mu_{[x],[x']}:s[x]\otimes s[x']\To s[x\otimes x']$ such that (\ref{axioma_coherencia_amb_s}) commutes. We can assume that this family is normalized, i.e. such that
$$
\mu_{[e],[x]}=\mu_{[x],[e]}=id_{s[x]}
$$
for all $[x]\in\pi_0(\GG)$. Indeed, since $s$ is normalized and $\GG$ strict, all objects $s[e]\otimes s[x]$, $s[x]\otimes s[e]$, $s[e\otimes x]$ and $s[x\otimes e]$ are equal to $s[x]$ for any $[x]$. Hence, if the family is not normalized, it is enough to consider the new family defined by
$$
\mu'_{[x],[x']}:=\left\{\begin{array}{ll} \mu_{[x],[x']}, & \mbox{if $[x],[x']\neq[e]$} \\ id_{s[x]}, & \mbox{if $[x']=[e]$} \\  id_{s[x']}, & \mbox{if $[x]=[e]$.} \end{array}\right.
$$
This new family also makes (\ref{axioma_coherencia_amb_s}) commute for all $[x],[x'],[x'']\in\pi_0(\GG)$. Then let us take as lifting $\psi_s$ of $\hat{s}$ the map given by
$$
\psi_s([x],[x']):=\mu^{-1}_{[x],[x']}\otimes id_{s[x\otimes x']^{-1}}.
$$
It is clearly normalized. W claim that it satisfies the ``2-cocycle condition'' (\ref{condicio_2-cocicle}). Indeed, the product in $\Hsf_0$ is the tensor product of morphisms in $\Gg$ and $\GG$ is strict. Therefore the condition takes the form
\begin{align*} 
\mu^{-1}_{[x],[x']}\otimes id_{s[x\otimes x']^{-1}}\otimes&\mu^{-1}_{[x\otimes x'],[x'']}\otimes id_{s[x\otimes x'\otimes x'']^{-1}} \\ &=id_{s[x]}\otimes\mu^{-1}_{[x'],[x'']}\otimes id_{s[x'\otimes x'']^{-1}}\otimes id_{s[x]^{-1}}\otimes\mu^{-1}_{[x],[x'\otimes x'']}\otimes id_{s[x\otimes x'\otimes x'']^{-1}}
\end{align*}
Now, tensoring with any object is a self-equivalence of $\Gg$. Hence this is equivalent to 
\begin{align*} 
\mu^{-1}_{[x],[x']}\otimes id_{s[x\otimes x']^{-1}}\otimes&\mu^{-1}_{[x\otimes x'],[x'']} \\ &=id_{s[x]}\otimes\mu^{-1}_{[x'],[x'']}\otimes id_{s[x'\otimes x'']^{-1}\otimes s[x]^{-1}}\otimes\mu^{-1}_{[x],[x'\otimes x'']}
\end{align*}
Using again that $\GG$ is strict, we can rewrite the left hand side as
\begin{align*} 
\mu^{-1}_{[x],[x']}\otimes id_{s[x\otimes x']^{-1}}&\otimes\mu^{-1}_{[x\otimes x'],[x'']} \\ &=\mu^{-1}_{[x],[x']}\otimes id_{s[x'']}\otimes id_{s[x'']^{-1}\otimes s[x\otimes x']^{-1}}\otimes\mu^{-1}_{[x\otimes x'],[x'']} \\ &=[(\mu^{-1}_{[x],[x']}\otimes id_{s[x'']})\circ id_{s[x\otimes x']\otimes s[x'']}]\otimes[id_e\circ(id_{s[x'']^{-1}\otimes s[x\otimes x']^{-1}}\otimes\mu^{-1}_{[x\otimes x'],[x'']})] \\ &=(\mu^{-1}_{[x],[x']}\otimes id_{s[x'']}\otimes id_e)\circ(id_e\otimes\mu^{-1}_{[x\otimes x'],[x'']}) \\ &=(\mu^{-1}_{[x],[x']}\otimes id_{s[x'']})\circ\mu^{-1}_{[x\otimes x'],[x'']}
\end{align*}
Similarly, it is shown that the right hand side is
$$
id_{s[x]}\otimes\mu^{-1}_{[x'],[x'']}\otimes id_{s[x'\otimes x'']^{-1}\otimes s[x]^{-1}}\otimes\mu^{-1}_{[x],[x'\otimes x'']}=(id_{s[x]}\otimes\mu^{-1}_{[x'],[x'']})\circ\mu^{-1}_{[x],[x'\otimes x'']}.
$$
Therefore (\ref{condicio_2-cocicle}) is equivalent to
$$
(\mu^{-1}_{[x],[x']}\otimes id_{s[x'']})\circ\mu^{-1}_{[x\otimes x'],[x'']}=(id_{s[x]}\otimes\mu^{-1}_{[x'],[x'']})\circ\mu^{-1}_{[x],[x'\otimes x'']},
$$
and this is nothing but the commutativity of (\ref{axioma_coherencia_amb_s}).

Conversely, let us assume that there exists a normalized section $s$ and a normalized lifting $\psi_s$ of the induced map $\hat{s}$ satisfying the 2-cocycle condition. An argument similar to the previous one shows that the isomorphisms
$$
\mu_{[x],[x']}:=\psi_s^{-1}([x],[x'])\otimes id_{s[x\otimes x']^{-1}}
$$
are such that (\ref{axioma_coherencia_amb_s}) commutes. Hence $\GG$ is split by Theorem~\ref{criteri_general_2}.
\end{proof}

\medskip
Notice that this condition is indeed a weakened version of the condition in Corollary~\ref{criteri_suficient_split}. Indeed, when $s$ is a group homomorphism (in paticular, normalized), $\hat{s}$ is trivial. Hence we can take as normalized lifting $\psi_s$ the map sending all pairs $([x],[x'])$ to the identity of the unit object, which clearly satisfies (\ref{condicio_2-cocicle}).

\section{\large Permutation 2-groups}

In this section the symmetric groups are generalized to the (non necessarily discrete) groupoid setting. The new structures will be called {\em permutation 2-groups}, to distinguish them from what are usually called symmetric 2-groups (i.e. 2-groups whose tensor product is commutative up to a fixed coherent natural isomorphism). Their structure is investigated and their homotopy invariants computed. In particular, it is shown that they are non-split in general, and we explicitly identify the source of non-splitness. Various examples of permutation 2-groups are explicitly computed.

\subsection{Permutation 2-groups and finite type permutation 2-groups}

The symmetric groups $\Ssf_n$ generalize in an obvious way to the context of 2-groups. Indeed, we can identify $\Ssf_n$ with the discrete 2-group $\Ssf_n[0]$, and this is nothing but the 2-group of self-equivalences of the discrete groupoid $X[0]$ associated to the set $X=\{1,\ldots,n\}$. This suggests the following general definition, which is an invariant version of Definition~\ref{2-grup_permutacions_grupoide}. 

\begin{defn}
A {\em permutation 2-group} is a 2-group equivalent to the 2-group of permutations $\SSS ym(\Gg)$ of some groupoid $\Gg$.
\end{defn}

We are mainly interested in the permutation 2-groups of the {\em finite type} groupoids. To explain what we mean by this, let us recall that any groupoid $\Gg$ decomposes as a disjoint union of its connected components. In the skeletal case, these are necessarily one-object groupoids $\Gsf$. Hence any groupoid $\Gg$ is of the form
$$
\Gg\simeq\coprod_{j\in J}\Gsf_j
$$
for some family of groups $\{\Gsf_j\}_{j\in J}$. Now, there may exists non-isomorphic objects (hence, objects in different connected components) having isomorphic automorphism groups, i.e. we may have $\Gsf_j\cong\Gsf_{j'}$ for some pair $j,j'$ with $j\neq j'$. By a {\em homogeneous component} of a groupoid $\Gg$ we shall mean the disjoint union of all its connected components having a given automorphism group $\Gsf$ (up to isomorphism), and we shall call $\Gsf$ the {\em base group} of the homogeneous component.

It follows that any groupoid $\Gg$ canonically decomposes as a disjoint union of its homogeneous components $\{\Gg_i\}_{i\in I}$. If the corresponding base groups are $\{\Gsf_i\}_{i\in I}$ we have
\begin{equation*}
\Gg=\coprod_{i\in I}\Gg_i
\end{equation*}
with
$$
\Gg_i\simeq\coprod_{j\in J_i}\Gsf_i.
$$
When $|I|=1$ (only one homogeneous component), we shall say that $\Gg$ is a {\em homogeneous groupoid}.

\begin{ex}{\rm
The underlying groupoid of any 2-group is homogeneous with an abelian base group (see \S~\ref{invariants_homotopics}).}
\end{ex}

The groupoids and the associated permutation 2-groups we are mostly interested in are then the following.

\begin{defn}
A groupoid is called of {\em finite type} when each homogeneous component has a finite number of connected components (i.e. $J_i$ is a finite set for all $i\in I$). A {\em finite type permutation 2-group} is a 2-group equivalent to the permutation 2-group of a finite type groupoid.
\end{defn}
Notice that we put no restriction on the cardinality of the set of homogeneous components nor on the cardinality of the base groups $\Gsf_i$.

\begin{ex}{\rm 
For any finite set $X$ the discrete groupoid $X[0]$ is of finite type. It has one homogeneous component with trivial base group decomposing into $|X|$ connected components.}
\end{ex}

\begin{ex}{\rm 
For any group $\Gsf$ the one-object groupoid $\Gsf$ is of finite type. These are in fact the homogeneous connected groupoids.}
\end{ex}

\begin{ex}{\rm 
The groupoid of all finite sets and bijections between them is of finite type. It has a countable set of homogeneous components, each of them connected, and with respective base groups the symmetric groups $\Ssf_n$, $n\geq 0$.}
\end{ex}

\begin{ex}{\rm 
The groupoid of all finite-dimensional $\Fsf$-vector spaces for any field $\Fsf$ and the isomorphisms between them is of finite type. It also has a countable set of homogeneous components, each of them connected, but now with respective base groups the general linear groups $\Gsf\Lsf_n(\Fsf)$, $n\geq 0$.}
\end{ex}

\begin{ex}{\rm 
The fundamental groupoid $\Pi_1(X)$ of any topological space $X$ with a finite number of path-connected components is of finite type. More generally, the same is true for any $X$ with an arbitrary number of path-connected components if the number of path-connected components with any given group as fundamental group (up to isomorphism) is finite. }
\end{ex}

\begin{ex}{\rm
The action groupoid $\Gg_X$ of any $\Gsf$-set $X$ having a finite number of $\Gsf$-orbits is of finite type. More generally, the same is true when the number of $\Gsf$-orbits with a given stabilizer subgroup is finite. }
\end{ex}

Let $(n,\Gsf)$ be any pair consisting of a positive integer $n\geq 1$ and a group $\Gsf$. We shall denote by $\Gg_{n,\Gsf}$ the skeletal groupoid given by the coproduct of $n$ copies of $\Gsf$, and by $\SSS_{n,\Gsf}$ the associated permutation 2-group $\SSS ym(\Gg_{n,\Gsf})$. We shall just write $\SSS_\Gsf$ when $n=1$, and $\SSS_n$ when $\Gsf=\mathsf{1}$.

Notice that $\SSS_{n}\cong\Ssf_n[0]$ for any $n\geq 1$. Because of this, from now on we shall write $\SSS_n\wr\wr\ \GG$ to denote the wreath 2-product $\Ssf_n\wr\wr\ \GG$.

\begin{prop}\label{S_G_trivial}
$\SSS_\Gsf\simeq\mathbbm{1}$ if and only if the group $\Gsf$ is complete, i.e. centerless and such that all its automorpisms are inner.
\end{prop}
\begin{proof}
We know from Proposition~\ref{exemple_2-grup_split_bis} that the first and second homotopy groups of $\SSS_\Gsf$ are $\mathsf{Out}(\Gsf)$ and $\Zsf(\Gsf)$, respectively. The result follows then from Sinh's theorem. 
\end{proof}

\begin{prop}\label{Sim(A)_bis}
If $\Asf$ is an abelian group, $\SSS_{\Asf}\simeq\Asf[1]\rtimes\mathsf{Aut}(\Asf)[0]$. 
\end{prop}
\begin{proof}
This is Corollary~\ref{Sim(A)}.
\end{proof}

\begin{prop}\label{2-grups_permutacions_grups_simetrics}
The permutation 2-groups associated to the symmetric groups $\Ssf_n$ are given by
$$
\SSS_{\Ssf_n}\simeq\left\{\begin{array}{ll} \ZZ_2[1], & \mbox{if $n=2$} \\ \ZZ_2[0], & \mbox{if $n=6$} \\ \mathbbm{1}, & \mbox{if $n\neq 2,6$.} \end{array}\right.
$$
\end{prop}
\begin{proof}
$\Ssf_n$ is complete for all $n\neq 2,6$ (see \cite{jjR95}). Hence the case $n\neq 2,6$ follows from Proposition~\ref{S_G_trivial}. The cases $n=2$ and $n=6$ follow from Proposition~\ref{exemple_2-grup_split_bis}. Indeed, we have $\Ssf_2\cong\ZZ_2$. Hence its center is $\ZZ_2$ and its outer automorphism group trivial. On the other hand, $\Ssf_6$ is centerless and has exactly one non-trivial outer automorphism (see also \cite{jjR95}). 
\end{proof}

It is worth emphasizing that non-isomorphic groups $\Gsf$ may correspond to equivalent permutation 2-groups. For instance, this is so for any pair of non-isomorphic complete groups, and for all the symmetric groups $\Ssf_n$ with $n\neq 2,6$. An example where the permutation 2-group is non-trivial is the following. 

\begin{ex}\label{S_D4_S_D6}{\rm
The permutation 2-groups $\SSS_{\Dsf_4}$ and $\SSS_{\Dsf_6}$ are both equivalent to the elementary  2-group $\ZZ_2[1]\rtimes\ZZ_2[0]$, with $\ZZ_2$ acting on itself trivially (see Example~\ref{splits_no_trivials_2}). }
\end{ex}

Let us further remark that the 2-group of permutations of the fundamental groupoid $\Pi_1(X)$ of a path-connected space $X$ is
$$
\SSS ym({\Pi_1(X)})\simeq\SSS_{\pi_1(X)}.
$$
Moreover, any group $\Gsf$ can be realized as the fundamental group of a path-connected CW-complex of dimension $\geq 2$. This means that all permutation 2-groups $\SSS_\Gsf$ are in fact permutation 2-groups of fundamental groupoids.

More generally, let $\{(n_i,\Gsf_i)\}_{i\in I}$ be any family of pairs as before, with $\Gsf_i\ncong\Gsf_{i'}$ for $i\neq i'$. We shall denote by $\Gg_{\{(n_i,\Gsf_i)\}_{i\in I}}$ the skeletal groupoid
$$
\Gg_{\{(n_i,\Gsf_i)\}_{i\in I}}:=\coprod_{i\in I}\Gg_{n_i,\Gsf_i},
$$
and by $\mathbb{S}_{\{(n_i,\Gsf_i)\}_{i\in I}}$ the associated permutation 2-group. When $|I|=k$ is finite, we shall just write $\Gg_{n_1,\Gsf_1;\ldots;n_k,\Gsf_k}$ and $\SSS_{n_1,\Gsf_1;\ldots;n_k,\Gsf_k}$, respectively. 

Observe that all 2-groups $\SSS_{\{(n_i,\Gsf_i)\}_{i\in I}}$ are strict. Thus they are strict as monoidal groupoids because $\mathbf{Gpd}$ is a strict 2-category, and any self-equivalence is bijective on objects and hence, strictly invertible because $\Gg_{\{(n_i,\Gsf_i)\}_{i\in I}}$ is skeletal. However, the underlying groupoid of $\SSS_{\{(n_i,\Gsf_i)\}_{i\in I}}$ is non skeletal. In fact, we shall see below that $\SSS_{\{(n_i,\Gsf_i)\}_{i\in I}}$ can be non-split. 

Equivalent objects in a 2-category have equivalent 2-groups of self-equivalences (see Example~\ref{equivalencia_Equiv(X)_Equiv(Y)}). Therefore, when dealing with finite type permutation 2-groups we shall assume without loss of generality that the groupoid is $\Gg_{\{(n_i,\Gsf_i)\}_{i\in I}}$. The goal of this section is to investigate the structure of the corresponding permutation 2-group $\mathbb{S}_{\{(n_i,\Gsf_i)\}_{i\in I}}$ and to compute its homotopy invariants. We shall see that $\mathbb{S}_{\{(n_i,\Gsf_i)\}_{i\in I}}$ can be constructed using as building blocks the permutation 2-groups $\SSS_{n}$ and $\SSS_\Gsf$, and how the homotopy invariants of $\mathbb{S}_{\{(n_i,\Gsf_i)\}_{i\in I}}$ can be computed in terms of the homotopy invariants of these basic blocks.

\subsection{Structure theorem}

Let $\{\Gg_i\}_{i\in I}$ be the homogeneous components of a groupoid $\Gg$. Then an endofunctor $\Ff:\Gg\To\Gg$ is given by a family of functors $\Ff_i:\Gg_i\To\Gg$, $i\in I$. When $\Ff$ is an equivalence, any object of $\Gg$ necessarily gets mapped by $\Ff$ to a (possibly non-isomorphic) object having the same group of automorphisms (up to isomorphism). Hence each $\Ff_i$ necessarily maps the homogeneous component $\Gg_i$ into itself. Moreover, it is a self-equivalence of $\Gg_i$ if the whole functor $\Ff$ is an equivalence. Therefore any self-equivalence $\Ee$ of $\Gg$ actually amounts to a family of self-equivalences $\{\Ee_i:\Gg_i\To\Gg_i\}_{i\in I}$. We shall call them the {\it homogeneous components} of $\Ee$.

Similarly, for any two self-equivalences $\Ee,\Ee'$ of $\Gg$ a natural isomorphism $\tau:\Ee\Rightarrow\Ee'$ is given by natural isomorphisms $\tau_i:\Ee_i\Rightarrow\Ee'_i$ between the respective homogeneous components, the {\it homogeneous components} of $\tau$. Moreover, the vertical composition is given componentwise. It follows that we have an isomorphism (not just an equivalence) of groupoids
$$
\Dd:\Ss ym(\Gg)\stackrel{\cong}{\longrightarrow}\prod_{i\in I}\Ss ym({\Gg_i})
$$
given on objects and morphisms by $\Ee\mapsto(\Ee_i)_{i\in I}$ and $\tau\mapsto(\tau_i)_{i\in I}$. 

\begin{lem}
The functor $\Dd:\Ss ym(\Gg)\To\prod_{i\in I}\Ss ym({\Gg_i})$ is a strict monoidal functor between $\SSS ym(\Gg)$ and the product 2-group $\prod_{i\in I}\SSS ym(\Gg_i)$.
\end{lem}
\begin{proof}
The tensor product in $\Ss ym(\Gg)$ is given by the composition of self-equivalences and the horizontal composition of natural isomorphisms. In terms of the homogeneous components, both operations are computed componentwise. On the other hand, the tensor product in $\prod_{i\in I}\SSS ym(\Gg_i)$ is also defined componentwise (see \S~\ref{2-categoria_2Grp}). Thus we have
$$
\Dd(\Ee'\circ\Ee)=\Dd(\Ee')\otimes\Dd(\Ee)
$$
$$
\Dd(\tau'\circ\tau)=\Dd(\tau')\otimes\Dd(\tau)
$$
for any objects $\Ee,\Ee'$ and morphisms $\tau,\tau'$ in $\Ss ym(\Gg)$. Furthermore, unit objects are strictly preserved because the homogeneous components of $id_\Gg$ are the identities $id_{\Gg_i}$, $i\in I$. Hence taking $\mu$ trivial the coherence axiom (\ref{axioma_coherencia}) holds because both $\SSS ym(\Gg)$ and $\prod_{i\in I}\SSS ym(\Gg_i)$ are strict as monoidal groupoids.
\end{proof}

In summary, we have shown the following:

\begin{thm}[Primary decomposition]\label{descomposicio_primaria_prop}
Let $\{\Gg_i\}_{i\in I}$ be the homogeneous components of a groupoid $\Gg$. Then the corresponding permutation 2-group decomposes in the form
$$
\SSS ym(\Gg)\cong\prod_{i\in I}\SSS ym({\Gg_i}).
$$
\end{thm}

\begin{cor}\label{corolari_descomposicio_primaria}
Let $\{(n_i,\Gsf_i)\}_{i\in I}$ be any family of pairs consisting of a positive integer $n_i\geq 1$ and a group $\Gsf_i$, with $\Gsf_i\ncong\Gsf_{i'}$ for $i\neq i'$. Then
$$
\SSS_{\{(n_i,\Gsf_i)\}_{i\in I}}\cong\prod_{i\in I}\SSS_{n_i,\Gsf_i}.
$$
\end{cor}

\noindent
Observe that this primary decomposition only exists in the higher dimensional setting. Indeed, the discrete groupoids are homogeneous.

\begin{cor}
Let $\Gg_{\rm FinSets}$ be the groupoid of all finite sets and bijections between them. Then the corresponding finite type permutation 2-group is
$$
\SSS ym({\Gg_{\rm FinSets}})\simeq\ZZ_2[1]\times\ZZ_2[0].
$$
\end{cor}
\begin{proof}
We know that $\Gg_{\rm FinSets}\simeq\coprod_{n\geq 0} \Ssf_n[1]$. Hence
$$
\SSS ym({\Gg_{\rm FinSets}})\simeq\prod_{n\geq 0}\SSS_{\Ssf_n}.
$$
The result follows then from Proposition~\ref{2-grups_permutacions_grups_simetrics}. 
\end{proof}

Let us now consider the permutation 2-group $\SSS_{n,\Gsf}$ for an arbitrary pair $(n,\Gsf)$ as before. The case $n=1$ has already been considered in Proposition~\ref{2-grup_autoequivalencies_grupoide}. Our purpose is to prove Proposition~\ref{S_n,G} below, according to which $\SSS_{n,\Gsf}$ for $n>1$ is isomorphic to the wreath 2-product $\SSS_n\wr\wr\ \SSS_\Gsf$.

Recall that the objects and morphisms in $\SSS_{n,\Gsf}$ are the self-equivalences of the groupoid $\coprod^n\Gsf$ and the natural isomorphisms between them. They have the following explicit descriptions, which generalize to arbitrary $n\geq 1$ what we encountered in Proposition~\ref{2-grup_autoequivalencies_grupoide} for $n=1$.

\begin{lem}\label{lema1}
An equivalence of groupoids $\Ee:\coprod^n\Gsf\To\coprod^n\Gsf$ amounts to a pair $(\sigma,\Phi)$, with $\sigma\in S_n$ and $\Phi=(\phi_1,\ldots,\phi_n)\in Aut(\Gsf)^n$.
\end{lem}
\begin{proof}
Let us denote by $\ast_1,\ldots,\ast_n$ the objects of $\coprod^n\Gsf$. Then a functor $\Ff:\coprod^n\Gsf\To\coprod^n\Gsf$ amounts to $n$ functors $\Ff_i:\Gsf\To\Gsf$, $i=1,\ldots,n$, and each such functor is completely given by the object $\ast_{f(i)}$ (the image of $\ast_i$ by $\Ff_i$) together a functor $\hat{\Ff}_i:\Gsf\To\Gsf$ and hence, a homomorphism of groups $\phi_i:\Gsf\To\Gsf$. In other words, $\Ff$ amounts to a map $f:\{1,\ldots,n\}\To\{1,\ldots,n\}$, and $n$ endomorphisms $\phi_1,\ldots,\phi_n$ of $\Gsf$. Since $\coprod^n\Gsf$ is skeletal, the functor $\Ff$ so defined is an equivalence if and only if $f$ is a permutation $\sigma\in S_n$ and all $\phi_1,\ldots,\phi_n$ are automorphisms.
\end{proof}

\espai
\noindent
We shall denote by $\Ee(\sigma,\Phi)$ the self-equivalence of $\coprod^n\Gsf$ defined by the pair $(\sigma,\Phi)$. Thus
\begin{equation}
\Ee(\sigma,\Phi)(g,\ast_i)=(\phi_i(g),\ast_{\sigma(i)})\label{E_sigma_phi}
\end{equation}
for any morphism $(g,\ast_i)$ in $\coprod^n\Gsf$.

\begin{lem}\label{lema3}
Let $(\sigma,\Phi),(\tilde{\sigma},\tilde{\Phi})\in S_n\times Aut(\Gsf)^n$. Then:
\begin{itemize}
\item[(1)] If $\sigma\neq\tilde{\sigma}$, there is no natural isomorphism between $\Ee(\sigma,\Phi)$ and $\Ee(\tilde{\sigma},\tilde{\Phi})$.
\item[(2)] If $\sigma=\tilde{\sigma}$, a natural isomorphism $\tau:\Ee(\sigma,\Phi)\Rightarrow\Ee(\sigma,\tilde{\Phi})$ is an element $\mathbf{g}=(g_1,\ldots,g_n)\in G^n$ such that
\begin{equation}\label{naturalitat_tau}
\tilde{\phi}_i=c_{g_i}\circ\phi_i,\qquad i=1,\ldots,n.
\end{equation} 
\end{itemize}
\end{lem}
\begin{proof}
If $\sigma\neq\tilde{\sigma}$, there exists at least one object $\ast_i$ in $\coprod^n\Gsf$ that gets mapped to different objects $\ast_{\sigma(i)}$ and $\ast_{\tilde{\sigma}(i)}$ by $\Ee(\sigma,\Phi)$ and $\Ee(\tilde{\sigma},\tilde{\Phi})$, respectively, and hence to two different connected components. However, this is impossible. This proves the first statement. Let us now suppose that $\sigma=\tilde{\sigma}$. Then a natural isomorphism $\tau:\Ee(\sigma,\Phi)\Rightarrow\Ee(\sigma,\tilde{\Phi})$ is given by $n$ elements $g_1,\ldots,g_n\in G$, so that $\tau_{\ast_i}=(g_i,\ast_{\sigma(i)})$, and (\ref{naturalitat_tau}) is nothing but the naturality in $\ast_i$.
\end{proof}

\espai
\noindent
We shall denote by $\tau(\mathbf{g};\sigma,\Phi,\tilde{\Phi}):\Ee(\sigma,\Phi)\Rightarrow\Ee(\sigma,\tilde{\Phi})$ the natural isomorphism defined by any $\mathbf{g}\in G^n$ satisfying (\ref{naturalitat_tau}). Thus
\begin{equation}\label{tau_g's}
\tau(\mathbf{g};\sigma,\Phi,\tilde{\Phi})_{\ast_i}:=(g_i,\ast_{\sigma(i)}),\qquad i=1,\ldots,n.
\end{equation}

\noindent
As for the tensor product and the composition of morphisms in $\SSS_{n,\Gsf}$, they have the following explicit descriptions, which also generalize what we encountered in the case $n=1$.

\begin{lem}\label{lema2}
For any pairs $(\sigma,\Phi),(\sigma',\Phi')\in S_n\times Aut(\Gsf)^n$ we have
$$
\Ee(\sigma,\Phi)\otimes\Ee(\sigma',\Phi')=\Ee(\sigma\sigma',(\Phi\rhd\sigma')\circ\Phi'),
$$
where $\rhd$ denotes the usual (right) wreath action of $S_n$ on the direct product group $\mathsf{Aut}(\Gsf)^n$. In other words, the assignments $(\sigma,\Phi)\mapsto\Ee(\sigma,\Phi)$ define an isomorphism of groups between the wreath product $\Ssf_n\wr\mathsf{Aut}(\Gsf)$ and the group of objects in $\SSS_{n,\Gsf}$.
\end{lem}

\begin{proof}
The tensor product of two objects corresponds to their composition as self-equivalences of $\coprod^n\Gsf$. Now, it follows from (\ref{E_sigma_phi}) that
$$
(\Ee(\sigma,\Phi)\circ\Ee(\sigma',\Phi'))(g,\ast_i)=\Ee(\sigma,\Phi)(\phi'_i(g),\ast_{\sigma'(i)})=(\phi_{\sigma'(i)}(\phi'_i(g)),\ast_{\sigma(\sigma'(i))}).
$$
But $\phi_{\sigma'(i)}\circ\phi'_i$ is the $i^{th}$-component of $(\Phi\rhd\sigma')\circ\Phi'$. Hence
$$
(\Ee(\sigma,\Phi)\circ\Ee(\sigma',\Phi'))(g,\ast_i)=\Ee(\sigma\sigma',(\Phi\rhd\sigma')\circ\Phi')(g,\ast_i),
$$
and this is true for any object $(g,\ast_i)\in\coprod^n\Gsf$.
\end{proof}

\espai
\noindent
To have a compact description of the composition and tensor product of morphisms let us write
\begin{align*}
  \tilde{\mathbf{g}}\cdot\mathbf{g}&:=(\tilde{g}_1g_1,\ldots,\tilde{g}_ng_n), \\ \Phi(\mathbf{g})&:=(\phi_1(g_1),\ldots,\phi_n(g_n))
\end{align*}
for any $\mathbf{g},\mathbf{g}'\in G^n$ and any $\Phi\in Aut(\Gsf)^n$. Then composition and tensor product are as follows.

\begin{lem}\label{lema4}
For any morphisms $\tau(\mathbf{g};\sigma,\Phi,\tilde{\Phi})$, $\tau(\tilde{\mathbf{g}};\sigma,\tilde{\Phi},\tilde{\tilde{\Phi}})$, $\tau(\mathbf{g}';\sigma',\Phi',\tilde{\Phi}')$, with $\mathbf{g}=(g_1,\ldots,g_n)$ satisfying (\ref{naturalitat_tau}), and $\tilde{\mathbf{g}}=(\tilde{g}_1,\ldots,\tilde{g}_n)$ and $\mathbf{g}'=(g'_1,\ldots,g'_n)$ satisfying the analogous conditions for the corresponding collections of automorphisms of $\Gsf$, we have:
\begin{itemize}
\item[(1)] $\tau(\tilde{\mathbf{g}};\sigma,\tilde{\Phi},\tilde{\tilde{\Phi}})\circ\tau(\mathbf{g};\sigma,\Phi,\tilde{\Phi})=\tau(\tilde{\mathbf{g}}\cdot\mathbf{g};\sigma,\Phi,\tilde{\tilde{\Phi}})$.
\item[(2)] $\tau(\mathbf{g};\sigma,\Phi,\tilde{\Phi})\otimes\tau(\mathbf{g}';\sigma',\Phi',\tilde{\Phi}')=\tau((\mathbf{g}\rhd\sigma')\cdot(\Phi\rhd\sigma')(\mathbf{g}');\sigma\sigma',(\Phi\rhd\sigma')\circ\Phi',(\tilde{\Phi}\rhd\sigma')\circ\tilde{\Phi}')$.
\end{itemize}
\end{lem}
\begin{proof}
Compositon corresponds to the vertical composition of natural isomorpyhisms. Now, it follows from (\ref{tau_g's}) that
\begin{align*}
(\tau(\tilde{\mathbf{g}};\sigma,\tilde{\Phi},\tilde{\tilde{\Phi}})\cdot\tau(\mathbf{g};\sigma,\Phi,\tilde{\Phi}))_{\ast_i}&=\tau(\tilde{\mathbf{g}};\sigma,\tilde{\Phi},\tilde{\tilde{\Phi}})_{\ast_i}\circ\tau(\mathbf{g};\sigma,\Phi,\tilde{\Phi})_{\ast_i} \\ &=(\tilde{g}_i,\ast_{\sigma(i)})\circ(g_i,\ast_{\sigma(i)}) \\ &=(\tilde{g}_ig_i,\ast_{\sigma(i)}) \\ &=\tau(\tilde{\mathbf{g}}\cdot\mathbf{g};\sigma,\Phi,\tilde{\tilde{\Phi}})_{\ast_i}.
\end{align*}
Similarly, the tensor product is given by the horizontal composition, and it follows from (\ref{E_sigma_phi}) and (\ref{tau_g's}) that
\begin{align*}
(\tau(\mathbf{g};\sigma,\Phi,\tilde{\Phi})\circ\tau(\mathbf{g}';\sigma',\Phi',\tilde{\Phi}'))_{\ast_i}&=\tau(\mathbf{g};\sigma,\Phi,\tilde{\Phi})_{\Ee(\sigma',\Phi')(\ast_i)}\circ\Ee(\sigma,\Phi)(\tau(\mathbf{g}';\sigma',\Phi',\tilde{\Phi}')_{\ast_i}) \\ &=\tau(\mathbf{g};\sigma,\Phi,\tilde{\Phi})_{\ast_{\sigma'(i)}}\circ\Ee(\sigma,\Phi)(g'_i,\ast_{\sigma'(i)})) \\ &=(g_{\sigma'(i)},\ast_{(\sigma\sigma')(i)})\circ(\phi'_{\sigma'(i)}(g'_i),\ast_{(\sigma\sigma')(i)}) \\ &=(g_{\sigma'(i)}\phi'_{\sigma'(i)}(g'_i),\ast_{(\sigma\sigma')(i)}) \\ &=(((\mathbf{g}\rhd\sigma')\cdot(\Phi\rhd\sigma')(\mathbf{g}'))_i,\ast_{(\sigma\sigma')(i)}) \\ &=\tau((\mathbf{g}\rhd\sigma')\cdot(\Phi\rhd\sigma')(\mathbf{g}');\sigma\sigma',(\Phi\rhd\sigma')\circ\Phi',(\tilde{\Phi}\rhd\sigma')\circ\tilde{\Phi}')_{\ast_i}.
\end{align*}
\end{proof}

\noindent
Before proving our claim about the structure of $\SSS_{n,\Gsf}$, we still need some more notation concerning the product 2-group $\SSS_\Gsf^n$. Its objects will be denoted by
$$
\Ee(\Phi)=(\Ee(\phi_1),\ldots,\Ee(\phi_n)),
$$
for any $\Phi=(\phi_1,\ldots,\phi_n)\in Aut(\Gsf)^n$, where $\Ee(\phi)$ denotes the automorphism of $\Gsf$ associated to $\phi$. It follows from Proposition~\ref{2-grup_autoequivalencies_grupoide} that
\begin{equation}\label{otimes_objectes_S_n,G}
\Ee(\Phi)\otimes\Ee(\Phi')=\Ee(\Phi\circ\Phi').
\end{equation}
On the other hand, any $\mathbf{g}\in G^n$ satisfying (\ref{naturalitat_tau}) determines a morphism in $\Ss_\Gsf^n$ that we shall denote by
$$
\tau(\mathbf{g};\Phi,\tilde{\Phi}):\Ee(\Phi)\To\Ee(\tilde{\Phi}).
$$
In the notations of Proposition~\ref{2-grup_autoequivalencies_grupoide}, it is the morphism with components
\begin{equation}\label{components_g_phi_tildephi}
\tau(\mathbf{g};\Phi,\tilde{\Phi})_i=\tau(g_i;\phi_i,\tilde{\phi}_i):\Ee(\phi_i)\Rightarrow\Ee(\tilde{\phi}_i),\qquad i=1,\ldots,n.
\end{equation}
Since $\Ss_\Gsf^n$ is a product, the composite of $\tau(\mathbf{g};\Phi,\tilde{\Phi})$ with a morphism $\tau(\tilde{\mathbf{g}};\tilde{\Phi},\tilde{\tilde{\Phi}}):\Ee(\tilde{\Phi})\To\Ee(\tilde{\tilde{\Phi}})$ is computed componentwise. Proposition~\ref{2-grup_autoequivalencies_grupoide} then gives
\begin{equation}\label{composicio_S_G^n}
\tau(\tilde{\mathbf{g}};\tilde{\Phi},\tilde{\tilde{\Phi}})\circ\tau(\mathbf{g};\Phi,\tilde{\Phi})=\tau(\tilde{\mathbf{g}}\cdot\mathbf{g};\Phi,\tilde{\tilde{\Phi}}),
\end{equation}
which generalizes the formula we had for $n=1$ to arbitrary $n\geq 1$. Similarly, the tensor product of $\tau(\mathbf{g};\Phi,\tilde{\Phi})$ with a morphism $\tau(\mathbf{g}';\Phi',\tilde{\Phi}'):\Ee(\Phi')\To\Ee(\tilde{\Phi}')$ is also computed componentwise, and Proposition~\ref{2-grup_autoequivalencies_grupoide} gives
\begin{equation}\label{otimesS_G^n}
\tau(\mathbf{g};\Phi,\tilde{\Phi})\otimes\tau(\mathbf{g}';\Phi',\tilde{\Phi}')=\tau(\mathbf{g}\cdot\Phi(\mathbf{g}');\Phi\circ\Phi',\tilde{\Phi}\circ\tilde{\Phi}').
\end{equation}
This is again a generalization to arbitrary $n\geq 1$ of the formula we had for $n=1$.

Let us now consider the functor $\Tt_{n,\Gsf}:\Ssf_n[0]\times\Ss_\Gsf^n\To\Ss_{n,\Gsf}$ given on objects $(\sigma,\Ee(\Phi))$ and morphisms $(id_\sigma,\tau(\mathbf{g};\Phi,\tilde{\Phi}))$ by
\begin{align*}
\Tt_{n,\Gsf}(\sigma,\Ee(\Phi))&:=\Ee(\sigma,\Phi),\\ \Tt_{n,\Gsf}(id_\sigma,\tau(\mathbf{g};\Phi,\tilde{\Phi}))&:=\tau(\mathbf{g};\sigma,\Phi,\tilde{\Phi}).
\end{align*}
It follows from (\ref{composicio_S_G^n}) and item (1) in Lemma~\ref{lema4} that these assignments are functorial. Indeed, we have
\begin{align*}
\Tt_{n,\Gsf}((id_\sigma,\tau(\tilde{\mathbf{g}};\tilde{\Phi},\tilde{\tilde{\Phi}}))\circ(id_\sigma,\tau(\mathbf{g};\Phi,\tilde{\Phi})))&=\Tt_{n,\Gsf}(id_\sigma,\tau(\tilde{\mathbf{g}}\cdot\mathbf{g};\Phi,\tilde{\tilde{\Phi}})) \\ &=\tau(\tilde{\mathbf{g}}\cdot\mathbf{g};\sigma,\Phi,\tilde{\tilde{\Phi}}) \\ &=\tau(\tilde{\mathbf{g}};\sigma,\tilde{\Phi},\tilde{\tilde{\Phi}})\cdot\tau(\mathbf{g};\sigma,\Phi,\tilde{\Phi}) \\ &=\Tt_{n,\Gsf}(id_\sigma,\tau(\tilde{\mathbf{g}};\tilde{\Phi},\tilde{\tilde{\Phi}}))\circ\Tt_{n,\Gsf}(id_\sigma,\tau(\mathbf{g};\Phi,\tilde{\Phi})).
\end{align*}
Then we have the following.

\begin{prop}\label{S_n,G}
For any $n\geq 1$ and any group $\Gsf$, the above functor $\Tt_{n,\Gsf}$ gives a strict isomorphism of 2-groups
$$
\TT_{n,\Gsf}:\SSS_n\wr\wr\ \SSS_\Gsf\longrightarrow\SSS_{n,\Gsf}.
$$ 
\end{prop}
\begin{proof}
It follows from Lemmas~\ref{lema1} and \ref{lema3} that $\Tt_{n,\Gsf}$ is an isomorphism of groupoids. It remains to see that it preserves strictly the monoidal structures in $\SSS_n\wr\wr\ \SSS_\Gsf$ and $\SSS_{n,\Gsf}$. Let us consider any objects $(\sigma,\Ee(\Phi)),(\sigma',\Ee(\Phi'))\in\Ssf_n[0]\times\Ss_\Gsf^n$. It follows from Proposition~\ref{S_n_wreath_GG} and (\ref{otimes_objectes_S_n,G}) that their tensor product is given
$$
(\sigma,\Ee(\Phi))\otimes(\sigma',\Ee(\Phi'))=(\sigma\sigma',\Ee((\Phi\rhd\sigma')\circ\Phi')).
$$
Hence by Lemma~\ref{lema2} we have
\begin{align*}
\Tt_{n,\Gsf}(\sigma,\Ee(\Phi))\otimes\Tt_{n,\Gsf}(\sigma',\Ee(\Phi'))&=\Ee(\sigma,\Phi)\otimes\Ee(\sigma',\Phi')\\ &=\Ee(\sigma\sigma',(\Phi\rhd\sigma')\circ\Phi') \\ &=\Tt_{n,\Gsf}(\sigma\sigma',\Ee((\Phi\rhd\sigma')\circ\Phi'))  \\ &=\Tt_{n,\Gsf}((\sigma,\Ee(\Phi))\otimes(\sigma',\Ee(\Phi'))).
\end{align*}
On the other hand, for any morphisms $(id_\sigma,\tau(\mathbf{g};\Phi,\tilde{\Phi})),(id_{\sigma'},\tau(\mathbf{g}';\Phi',\tilde{\Phi}'))$ in $\Ssf_n[0]\times\Ss_\Gsf^n$, their tensor product is given by Proposition~\ref{S_n_wreath_GG}, which together with (\ref{components_g_phi_tildephi}) and (\ref{otimesS_G^n}) gives
\begin{align*}
(id_\sigma,\tau(\mathbf{g};\Phi,\tilde{\Phi}))\otimes(id_{\sigma'},\tau(\mathbf{g}';&\Phi',\tilde{\Phi}')) \\ &\ =\ (id_{\sigma\sigma'},(\tau(\mathbf{g};\Phi,\tilde{\Phi})\rhd\sigma')\otimes\tau(\mathbf{g}';\Phi',\tilde{\Phi}')) \\ &\stackrel{(\ref{components_g_phi_tildephi})}{=}(id_{\sigma\sigma'},(\tau(\mathbf{g}\rhd\sigma';\Phi\rhd\sigma',\tilde{\Phi}\rhd\sigma')\otimes\tau(\mathbf{g}';\Phi',\tilde{\Phi}')) \\ &\stackrel{(\ref{otimesS_G^n})}{=}(id_{\sigma\sigma'},\tau((\mathbf{g}\rhd\sigma')\cdot(\Phi\rhd\sigma')(\mathbf{g}');(\Phi\rhd\sigma')\circ\Phi',(\tilde{\Phi}\rhd\sigma')\circ\tilde{\Phi}')) .
\end{align*}
Thus by item (2) in Lemma~\ref{lema4} we have
\begin{align*}
\Tt_{n,\Gsf}(id_\sigma,\tau(\mathbf{g};\Phi,\tilde{\Phi}))\otimes&\Tt_{n,\Gsf}(id_{\sigma'},\tau(\mathbf{g}';\Phi',\tilde{\Phi}')) \\ &=\tau(\mathbf{g};\sigma,\Phi,\tilde{\Phi})\otimes\tau(\mathbf{g}';\sigma',\Phi',\tilde{\Phi}') \\ &=\tau((\mathbf{g}\rhd\sigma')\cdot(\Phi\rhd\sigma')(\mathbf{g}');\sigma\sigma',(\Phi\rhd\sigma')\circ\Phi',(\tilde{\Phi}\rhd\sigma')\circ\tilde{\Phi}') \\ &=\Tt_{n,\Gsf}(id_{\sigma\sigma'},\tau((\mathbf{g}\rhd\sigma')\cdot(\Phi\rhd\sigma')(\mathbf{g}');(\Phi\rhd\sigma')\circ\Phi',(\tilde{\Phi}\rhd\sigma')\circ\tilde{\Phi}')) \\ &=\Tt_{n,\Gsf}((id_\sigma,\tau(\mathbf{g};\Phi,\tilde{\Phi}))\otimes(id_{\sigma'},\tau(\mathbf{g}';\Phi',\tilde{\Phi}'))).
\end{align*}
This proves that the two functors $\otimes\circ(\Tt_{n,\Gsf}\times\Tt_{n,\Gsf})$ and $\Tt_{n,\Gsf}\circ\otimes$ are actually the same and hence, that we can take $\mu$ equal to the identity. Finally, both $\SSS_n\wr\wr\ \SSS_\Gsf$ and $\SSS_{n,\Gsf}$ are strict 2-groups, the first one because $\SSS_\Gsf$ is strict. Hence (\ref{axioma_coherencia}) commutes when $\mu=1$.
\end{proof}

Together with Corollary~\ref{corolari_descomposicio_primaria}, Proposition~\ref{S_n,G} implies the following general structure theorem for the permutation 2-groups of finite type.

\begin{thm}[Structure Theorem]\label{teorema_estructura}
Let $\{(n_i,\Gsf_i)\}_{i\in I}$ be any family of pairs consisting of a positive integer $n_i\geq 1$ and a group $\Gsf_i$, with $\Gsf_i\ncong\Gsf_{i'}$ for $i\neq i'$. Then
$$
\SSS_{\{(n_i,\Gsf_i)\}_{i\in I}}\cong\prod_{i\in I}\SSS_{n_i}\wr\wr\ \SSS_{\Gsf_i}.
$$
\end{thm}

As an example, let us consider the permutation 2-groups $\SSS_{n,\Asf}$ for $n\geq 1$ and $\Asf$ an abelian group. These will be called {\em Cayley 2-groups}. Indeed, these are the permutation 2-groups which appear in the analog for 2-groups of Cayley's theorem. Thus we saw in \S~\ref{invariants_homotopics} that the underlying groupoid of any 2-group $\GG$ with a finite number $n$ of isomorphism classes of objects is equivalent to the groupoid $\coprod^n\pi_1(\GG)[1]$, with $\pi_1(\GG)$ an abelian group. Hence the corresponding permutation 2-group is equivalent to the Cayley 2-group $\SSS_{n,\pi_1(\GG)}$. The higher version of Cayley's theorem will be discussed in a sequel to this paper. Here let us just observe the following.

\begin{prop}\label{2-grup_permutacions_grupoide_subjacent_2-grup}
For any $n\geq 1$ and any abelian group $\Asf$ we have
\begin{equation*}
\SSS_{n,\Asf}\simeq\Asf^n[1]\rtimes(\Ssf_n\wr\mathsf{Aut}(\Asf))[0],
\end{equation*}
with $\Ssf_n\wr\mathsf{Aut}(\Asf)$ acting on $\Asf^n$ by
$$
(\sigma,\Phi)\lhd(a_1,\ldots,a_n)=(\phi_{\sigma^{-1}(1)}(a_{\sigma^{-1}(1)}),\ldots,\phi_{\sigma^{-1}(n)}(a_{\sigma^{-1}(n)})).
$$
\end{prop}
\begin{proof}
It follows from Theorem~\ref{teorema_estructura} and Proposition~\ref{Sim(A)_bis} that
\begin{equation*}
\SSS_{n,\Asf}\simeq\SSS_n\wr\wr\ \SSS_\Asf\simeq\SSS_n\wr\wr\ (\Asf[1]\rtimes\mathsf{Aut}(\Asf)[0]).
\end{equation*}
Moreover, since $\Asf[1]\rtimes\mathsf{Aut}(\Asf)[0]$ is split, the wreath 2-product is also split (see Proposition~\ref{caracter_split_producte_wreath}). Then the statement follows from Corollary~\ref{corolari} and the computation in \S~\ref{splitness_producte_wreath_2} of the homotopy groups of the wreath 2-product.
\end{proof}

\subsection{Homotopy invariants}

According to Theorem~\ref{teorema_estructura}, in order to classify the finite type permutation 2-groups $\SSS_{\{(n_i,\Gsf_i)\}_{i\in I}}$ it is enough to establish the relationship between the homotopy invariants of a product of 2-groups and those of the factors, and to compute the homotopy invariants of a wreath 2-product $\SSS_n\wr\wr\ \GG$. Let us start by computing the homotopy invariants of a product of 2-groups.

Recall that the cohomology groups $\Hsf^\bullet(\Gsf,\Asf)$ are functorial in the pair $(\Gsf,\Asf)$, contravariantly in $\Gsf$ and covariantly in $\Asf$. In particular, for any groups $\{\Gsf_i\}_{i\in I}$ and any abelian groups $\{\Asf_i\}_{i\in I}$, with $\Gsf_i$ acting on $\Asf_i$, we have canonical homomorphisms
$$
\textstyle{\Hsf^\bullet(\pi_j,\iota_j):\Hsf^\bullet(\Gsf_j,\Asf_j)\To \Hsf^\bullet(\prod_{i\in I}\Gsf_i,\prod_{i\in I}\Asf_i)}
$$
induced by the canonical projections $\pi_j:\prod_{i\in I}\Gsf_i\To\Gsf_j$ and injections $\iota_j:\Asf_j\To\prod_{i\in I}\Asf_i$. Here $\prod_{i\in I}\Asf_i$ is thought of as a $\prod_{i\in I}\Gsf_i$-module in the obvious way. They are given by
$$
\Hsf^n(\pi_j,\iota_j)([z_j]):=[\iota_j\circ z_j\circ\pi_j^n],\qquad j\in I,\quad n\geq 0.
$$
Together, they define a canonical homomorphism of groups
$$
\textstyle{\zeta:\bigoplus_{i\in I}\Hsf^\bullet(\Gsf_i,\Asf_i)\To\Hsf^\bullet(\prod_{i\in I}\Gsf_i,\prod_{i\in I}\Asf_i)}.
$$
It is the homomorphism mapping $([z_i])_{i\in I}$ to the cohomology class of the n-cocycle on $\prod_{i\in I}\Gsf_i$ defined by the composite map
\begin{equation}\label{3-cocicle_producte}
\xymatrix{
(\prod_{i\in I}G_i)^n\ar[r]^{\cong} & \prod_{i\in I}G_i^n\ar[rr]^{\ \ \ \ \prod_{i\in I}z_i\ \ } && \prod_{i\in I} A_i. }
\end{equation}

\begin{lem}\label{zeta_mono}
The map $\zeta:\bigoplus_{i\in I}\Hsf^\bullet(\Gsf_i,\Asf_i)\To\Hsf^\bullet(\prod_{i\in I}\Gsf_i,\prod_{i\in I}\Asf_i)$ is a monomorphism.
\end{lem}
\begin{proof}
Let $(z_i)_{i\in I}$, with $z_i\in Z^n(\Gsf_i,\Asf_i)$, be such that $\zeta([z_i])_{i\in I}=0$. Let $c\in C^{n-1}(\prod_{i\in I}G_i,\prod_{i\in I}A_i)$ be an (n-1)-cochain whose coboundary is equal to the n-cocycle (\ref{3-cocicle_producte}). Then if $(c_i)_{i\in I}$ are the components of $c$, it easily follows from $z=\partial c$ that $z_i=\partial c_i$ for all $i\in I$.
\end{proof}

The relationship between the homotopy invariants of a product 2-group and those of the factors is then as follows.
\begin{prop}\label{invariants_homotopics_producte}
Let $\{\mathbb{G}_i\}_{i\in I}$ be an arbitrary family of 2-groups. Then:
\begin{itemize}
\item[(1)] $\pi_0(\prod_{i\in I}\mathbb{G}_i)=\prod_{i\in I}\pi_0(\mathbb{G}_i)$.
\item[(2)] $\pi_1(\prod_{i\in I}\mathbb{G}_i)=\prod_{i\in I}\pi_1(\mathbb{G}_i)$.
\item[(3)] The canonical left action $\lhd_I$ of $\pi_0(\prod_{i\in I}\mathbb{G}_i)$ on $\pi_1(\prod_{i\in I}\mathbb{G}_i)$ is given componentwise, i.e.
\begin{equation*}
([x_i])_{i\in I}\lhd_I(u_i)_{i\in I}=([x_i]\lhd_i u_i)_{i\in I}.
\end{equation*}
\item[(4)] $\alpha(\prod_{i\in I}\mathbb{G}_i)=\zeta(\alpha(\mathbb{G}_i)_{i\in I})$.
\end{itemize}
\end{prop}
\begin{proof}
Items (1) and (2) follow readily from the definition of the product of 2-groups. As for item (3), let us first observe that the canonical homomorphisms $\gamma_\mathbf{x},\delta_\mathbf{x}:\mathsf{Aut}(\mathbf{e})\To\mathsf{Aut}(\mathbf{x})$ of the product are given by
\begin{align*}
\gamma_{\mathbf{x}}&=(\gamma_{x_i})_{i\in I}, \\ \delta_{\mathbf{x}}&=(\delta_{x_i})_{i\in I}
\end{align*}
for any object $\mathbf{x}=(x_i)_{i\in I}\in\prod_{i\in I}\Gg_i$. This follows from the expressions for the unitors of the product 2-group (see \S~\ref{2-categoria_2Grp}). Therefore the action of $[\mathbf{x}]=([x_i])_{i\in I}$ on $\mathbf{u}=(u_i)_{i\in I}$ is given by
\begin{align*}
[\mathbf{x}]\lhd_I \mathbf{u}&=\gamma^{-1}_{\mathbf{x}}(\delta_{\mathbf{x}}(\mathbf{u})) \\ &=(\gamma^{-1}_{x_i}(\delta_{x_i}(u_i)))_{i\in I} \\ &=([x_i]\lhd_i u_i)_{i\in I}.
\end{align*}
To prove (4), it is enough to check that (\ref{3-cocicle_producte}) is a classifying 3-cocycle of the product 2-group when $z_i$ is a classifying 3-cocycle of $\GG_i$ for each $i\in I$. Indeed, we saw in \S~\ref{invariants_homotopics} that a classifying 3-cocycle of the product is obtained by chosing an {\em \'epinglage} $(s,\theta)$ of $\prod_{i\in I}\mathbb{G}_i$. Now, such an {\em \'epinglage} is clearly given by
\begin{align*}
s([\mathbf{x}])&:=(s_i[x_i])_{i\in I}, \\ \theta_{(x_i)_{i\in I}}&:=(\theta_{i;x_i})_{i\in I}
\end{align*}
for any {\em \'epinglages} $(s_i,\theta_i)$ of $\mathbb{G}_i$, $i\in I$. Then if $z_i=\mathsf{z}_i(\GG_i)$ is the classifying 3-cocycle of $\mathbb{G}_i$ obtained from the {\it \'epinglage} $(s_i,\theta_i)$ by the procedure described in \S~\ref{invariants_homotopics}, it is easy to check that (\ref{3-cocicle_producte}) is precisely the classifying 3-cocycle of $\prod_{i\in I}\mathbb{G}_i$ obtained from the {\it \'epinglage} $(s,\theta)$. 
\end{proof}

\begin{cor}\label{producte_2-grups_escindits}
Let $\{\mathbb{G}_i\}_{i\in I}$ be any family of 2-groups. Then the product 2-group $\prod_{i\in I}\mathbb{G}_i$ is split if and only if $\mathbb{G}_i$ is split for all $i\in I$. In particular, the product of any family of elementary 2-groups $\{\Asf_i[1]\rtimes\Gsf_i[0]\}_{i\in I}$ is the elementary 2-group
$$
\prod_{i\in I}\Asf_i[1]\rtimes\Gsf_i[0]=(\prod_{i\in I}\Asf_i)[1]\rtimes(\prod_{i\in I}\Gsf_i)[0],
$$
with $\prod_{i\in I}\Gsf_i$ acting on $\prod_{i\in I}\Asf_i$ componentwise.
\end{cor}
\begin{proof}
If $\mathbb{G}_i$ is split for all $i\in I$, it follows from item (4) in Proposition~\ref{invariants_homotopics_producte} together with Theorem~\ref{criteri_general} that the product is also split. Conversely, if the product is split, each $\GG_i$ for $i\in I$ is split again by item (4) together now with Lemma~\ref{zeta_mono}.
\end{proof}

We know from \S~\ref{splitness_producte_wreath_2} that the homotopy groups of a wreath 2-product $\SSS_n\wr\wr\ \GG$ are
\begin{align*}
\pi_0(\SSS_n\wr\wr\ \GG)&\cong\Ssf_n\wr\pi_0(\GG), \\ \pi_1(\SSS_n\wr\wr\ \GG)&\cong\pi_1(\GG)^n,
\end{align*}
with the first acting on the second according to
\begin{equation}\label{accio_Sn_wreath_G}
(\sigma,[\mathbf{x}])\lhd\mathbf{u}=([\mathbf{x}]\lhd_n\mathbf{u})\rhd\sigma^{-1}.
\end{equation}
Here $\lhd_n$ denotes the componentwise action induced by the action of $\pi_0(\GG)$ on $\pi_1(\GG)$. When $\GG$ is strict, this reduces to (\ref{accio_producte_wreath}). When $\GG=\SSS_\Gsf$, this gives
\begin{align*}
\pi_0(\SSS_n\wr\wr\ \SSS_\Gsf)&\cong\Ssf_n\wr\mathsf{Out}(\Gsf), \\ \pi_1(\SSS_n\wr\wr\ \SSS_\Gsf)&\cong\Zsf(\Gsf)^n,
\end{align*}
with the first acting on the second according to
\begin{equation}\label{accio_pi0_S_n_wr_S_G_sobre_pi1}
(\sigma,[\Phi])\lhd(u_1,\ldots,u_n)=(\phi_{\sigma^{-1}(1)}(u_{\sigma^{-1}(1)}),\ldots,\phi_{\sigma^{-1}(n)}(u_{\sigma^{-1}(n)})).
\end{equation}
As for the Postnikov invariant, we know that $\SSS_n\wr\wr\ \GG$ is split if and only if $\GG$ is itself split (see Proposition~\ref{caracter_split_producte_wreath}). Now, we shall see in \S~\ref{seccio_2-grups_permutacions_no_escindits} that there exists groups $\Gsf$ such that $\SSS_\Gsf$ is non-split. Hence the 2-group $\SSS_n\wr\wr\ \SSS_\Gsf$ can be non-split. Therefore to completely classify $\SSS_n\wr\wr\ \SSS_\Gsf$ and more generally, $\SSS_n\wr\wr\ \GG$ we need to compute its Postnikov invariant in terms of the Postnikov invariant of $\GG$.

Let us turn our attention for a moment to the general setting of an arbitrary group $\Gsf$ and an arbitrary left $\Gsf$-module $\Asf$. For any $n\geq 1$, the left action $\lhd$ of $\Gsf$ on $\Asf$ induces an obvious left action $\lhd_n$ of $\Gsf^n$ on $\Asf^n$ given componentwise. This in turn induces a left action $\tilde{\lhd}_n$ of $\Ssf_n\wr\Gsf$ on $\Asf^n$ generalizing (\ref{accio_Sn_wreath_G}), i.e.
\begin{equation*}
(\sigma,\mathbf{g})\tilde{\lhd}_n\mathbf{a}:=(\mathbf{g}\lhd_n\mathbf{a})\rhd\sigma^{-1},
\end{equation*}
where $\rhd$ denotes the usual (right) wreath action. Given $n\geq 1$ and $k\geq 0$, let
$$
f^{(n)}_k:C^k(\Gsf,\Asf)\To C^k(\Ssf_n\wr\Gsf,\Asf^n)
$$
be the map defined as follows. For any $c\in C^k(\Gsf,\Asf)$ and any $j=1,\ldots,n$ the $j^{th}$-component of
$$
f^{(n)}_k(c):(S_n\wr G)\times\stackrel{k)}{\cdots}\times(S_n\wr G)\To A^n
$$
is given by
\begin{equation*}
f^{(n)}_k(c)((\sigma_1,\mathbf{g}_1),\ldots,(\sigma_k,\mathbf{g}_k))_j:=c(g_{1,\sigma_1^{-1}(j)},g_{2,(\sigma_1\sigma_2)^{-1}(j)},\ldots,g_{k,(\sigma_1\cdots\sigma_k)^{-1}(j)})
\end{equation*}
when $\mathbf{g}_i=(g_{i,1},\ldots,g_{i,n})$ for each $i=1,\ldots,k$.

\begin{lem}
The maps $\{f^{(n)}_k\}_{k\geq 0}$ give the components of a morphism of cochain complexes $f^{(n)}:C^\bullet(\Gsf,\Asf)\To C^\bullet(\Ssf_n\wr\Gsf,\Asf^n)$. 
\end{lem}
\begin{proof}
Left to the reader.
\end{proof}

\medskip
We shall denote by $\xi_n:H^\bullet(\Gsf,\Asf)\To H^\bullet(\Ssf_n\wr\Gsf,\Asf^n)$ the homomorphism induced in cohomology by this morphism of cochain complexes.

\begin{lem}\label{lema_naturalitat_xi_n}
The homomorphism $\xi_n:H^\bullet(\Gsf,\Asf)\To H^\bullet(\Ssf_n\wr\Gsf,\Asf^n)$ is natural with respect to isomorphisms of pairs $(\Gsf,\Asf)$, i.e. for any isomorphism of groups $\rho:\Gsf\To\Gsf'$ and any isomorphism of $\Gsf$-modules $\beta:\Asf\To\Asf'_\rho$, the induced diagram
$$
\xymatrix{
H^\bullet(\Gsf,\Asf)\ar[rr]^{\Hsf^\bullet(\rho^{-1},\beta)}\ar[d]_{\xi_n} && H^\bullet(\Gsf',\Asf')\ar[d]^{\xi'_n} \\ H^\bullet(\Ssf_n\wr\Gsf,\Asf^n)\ar[rr]_{\Hsf^\bullet(\Ssf_n\wr\rho^{-1},\ \beta^n)} && H^\bullet(\Ssf_n\wr\Gsf',(\Asf')^n)
}
$$
commutes, where $\Ssf_n\wr\rho^{-1}:\Ssf_n\wr\Gsf\To\Ssf_n\wr\Gsf'$ is the obvious isomorphism induced by $\rho^{-1}$.
\end{lem}
\begin{proof}
Left to the reader.
\end{proof}

Let us now take $\Gsf=\pi_0(\GG)$ and $\Asf=\pi_1(\GG)$. It follows that the canonical action of $\pi_0(\GG)$ on $\pi_1(\GG)$ induces a morphism of cochain complexes 
$$
f^{(n)}:C^\bullet(\pi_0(\GG),\pi_1(\GG))\To C^\bullet(\Ssf_n\wr\pi_0(\GG),\pi_1(\GG)^n).
$$
The homomomorphism in cohomology
$$
\xi_n:H^\bullet(\pi_0(\GG),\pi_1(\GG))\To H^\bullet(\Ssf_n\wr\pi_0(\GG),\pi_1(\GG)^n)
$$
is natural in $(\pi_0(\GG),\pi_1(\GG))$ in the sense of Lemma~\ref{lema_naturalitat_xi_n}. Then we have the following.

\begin{prop}
For any 2-group $\GG$ and any $n\geq 1$, the Postnikov invariant of $\SSS_n\wr\wr\ \GG$ is the image by $\xi_n:H^3(\pi_0(\GG),\pi_1(\GG))\To H^3(\Ssf_n\wr\pi_0(\GG),\pi_1(\GG)^n)$ of the Postnikov invariant of $\GG$.
\end{prop}
 \begin{proof}
We can assume without loss of generality that $\GG$ is strict. Indeed, if $\GG$ is non-strict there exists an equivalent strict 2-group $\GG_s$, and any equivalence $\EE:\GG\To\GG_s$ induces a commutative diagram as before with $\rho=\pi_0(\EE)$ and $\beta=\pi_1(\EE)$. It follows that the claim is true for $\GG$ when it is true for $\GG_s$.

Let $(s,\theta)$ be any {\it \'epinglage} of $\GG$, and let $\mathsf{z}(\GG)$ be the classifying 3-cocycle of $\GG$ obtained from it by the method described in \S~\ref{invariants_homotopics}. Now, associated to $(s,\theta)$ there is a canonical {\em \'epinglage} $(s^{(n)},\theta^{(n)})$ of $\SSS_n\wr\wr\ \GG$ consisting of the map
$$
s^{(n)}:=id_{S_n}\times s\ \times\stackrel{n)}{\cdots}\times\ s:S_n\wr\pi_0(\GG)\To S_n\wr G_0,
$$
together with the isomorphisms in $S_n[0]\times\Gg^n$ given by
$$
\theta^{(n)}_{(\sigma,\mathbf{x})}:=(id_\sigma,(\theta_{x_1},\ldots,\theta_{x_n})):(\sigma,(s[x_1],\ldots,s[x_n]))\To(\sigma,(x_1,\ldots,x_n))
$$
for all $(\sigma,\mathbf{x})\in S_n\wr G_0$. Then it is a tedious but straightforward computation checking that the classifying 3-cocycle of $\SSS_n\wr\wr\ \GG$ obtained from the {\em \'epinglage} $(s^{(n)},\theta^{(n)})$ by the usual procedure is precisely the image of $\mathsf{z}(\GG)$ by the above morphism of cochain complexes $f^{(n)}$. The computation basically consists in making explicit the commutative diagram (\ref{definicio_3-cocicle}) that defines $\mathsf{z}(\SSS_n\wr\wr\ \GG)$. This is a diagram in the product groupoid $S_n[0]\times\Gg^n$. Hence it consists of a diagram in $S_n[0]$, which trivially commutes, and a diagram in $\Gg^n$ with lots of tensor products of objects and morphisms. Now, the tensor product in $\GG$ is given componentwise, so that this diagram in $\Gg^n$ actually amounts to $n$ diagrams in $\Gg$ with lots of tensor products now in $\Gg$. For any elements $(\sigma,[\mathbf{x}]),(\sigma',[\mathbf{x}']),(\sigma'',[\mathbf{x}''])\in S_n\wr\pi_0(\GG)$ let us denote by
$$
u_i((\sigma,[\mathbf{x}]),(\sigma',[\mathbf{x}']),(\sigma'',[\mathbf{x}'']))\in\pi_1(\GG)
$$
the $i^{th}$-component of
$$
\mathsf{z}(\SSS_n\wr\wr\ \GG)((\sigma,[\mathbf{x}]),(\sigma',[\mathbf{x}']),(\sigma'',[\mathbf{x}'']))\in\pi_1(\GG)^n
$$
for any $i=1,\ldots,n$. Then the point is that the required commutativity condition for each $u_i$ is the same for all of them, and it is exactly the commutative diagram that defines $\mathsf{z}(\GG)$ avaluated on the appropriate triple. The details are left to the reader.
\end{proof}  

\medskip
\noindent
In summary, we have shown that the homotopy invariants of a wreath 2-product are as follows:

\begin{prop}\label{invariants_producte_wreath}
For any 2-group $\GG$ we have:
\begin{itemize}
\item[(1)] $\pi_0(\SSS_n\wr\wr\ \GG)\cong\Ssf_n\wr\pi_0(\GG)$.
\item[(2)] $\pi_1(\SSS_n\wr\wr\ \GG)\cong\pi_1(\GG)^n$, with the action of the previous group on it given by (\ref{accio_Sn_wreath_G}). 
\item[(3)] $\alpha(\SSS_n\wr\wr\ \GG)=\xi_n(\alpha(\GG))$.
\end{itemize}
\end{prop}

Next result is now an immediate consequence of Proposition~\ref{caracter_split_producte_wreath}, Theorem~\ref{teorema_estructura}, Propositions~\ref{invariants_homotopics_producte} and \ref{invariants_producte_wreath} and the above computation of the homotopy groups of $\SSS_\Gsf$.

\begin{thm}\label{invariants_homotopics_2-grup_permutacions_general}
Let $\{(n_i,\Gsf_i)\}_{i\in I}$ be any family of pairs consisting of a positive integer $n_i\geq 1$ and a group $\Gsf_i$, with $\Gsf_i\ncong\Gsf_{i'}$ for $i\neq i'$. Then:
\begin{itemize}
\item[(1)] $\pi_0(\SSS_{\{(n_i,\Gsf_i)\}_{i\in I}})\cong\prod_{i\in I}\Ssf_{n_i}\wr\mathsf{Out}(\Gsf_i)$.
\item[(2)] $\pi_1(\SSS_{\{(n_i,\Gsf_i)\}_{i\in I}})\cong\prod_{i\in I}\Zsf(\Gsf_i)^{n_i}$ equipped with the $\pi_0(\SSS_{\{(n_i,\Gsf_i)\}_{i\in I}})$-module structure given componentwise by (\ref{accio_pi0_S_n_wr_S_G_sobre_pi1}).
\item[(3)] $\alpha(\SSS_{\{(n_i,\Gsf_i)\}_{i\in I}})=\zeta(\xi_{n_i}(\alpha_i)_{i\in I})$, where $\alpha_i=\alpha(\SSS_{\Gsf_i})$ for $i\in I$. Moreover, $\SSS_{\{(n_i,\Gsf_i)\}_{i\in I}}$ is split if and only if $\SSS_{\Gsf_i}$ is split for all $i\in I$.
\end{itemize}
\end{thm}

\subsection{Non-split permutation 2-groups}
\label{seccio_2-grups_permutacions_no_escindits}

It is still open the question whether all finite type permutation 2-groups are split. By our previous discussion, the problem is to determine whether all 2-groups $\SSS_\Gsf$ are split for all groups $\Gsf$. 

The split character of $\SSS_\Gsf$ is just a property of the group $\Gsf$. This suggests introducing the following definition.

\begin{defn}
A group $\Gsf$ is {\em permutationally split} if the corresponding permutation 2-group $\SSS_\Gsf$ is split (hence, equivalent to the 2-group $\Zsf(\Gsf)[1]\rtimes\mathsf{Out}(\Gsf)[0]$ by Corollary~\ref{corolari}).
\end{defn}

Two families of permutationally split groups are the following.
\begin{prop}
All abelian and all centerless groups are permutationally split.
\end{prop}
\begin{proof}
It is a direct consequence of Corollaries~\ref{Sim(A)} and \ref{splits_no_trivials_1}.
\end{proof} 

Examples of permutationally split groups which are neither abelian nor centerless are the dihedral groups $\Dsf_4$ and $\Dsf_6$ (see Example~\ref{splits_no_trivials_2}).

In fact, all of the examples of permutation 2-groups $\SSS_\Gsf$ we have given until now are split, and one might be tempted to think that this is always the case. However, this is false. To prove it, we shall make use of the next result, which gives a characterization of the permutationally split groups in more elementary terms. This follows immediately by applying Theorem~\ref{criteri_split_2-grups_estrictes} to $\SSS_\Gsf$.

\begin{prop}\label{caracteritzacio_grup_permutacionalment_escindit}
Let $\Gsf$ be any group, and let us consider the associated exact sequence
\begin{equation}\label{4-successio_S_G}
\xymatrix{
0\ar[r] & \Zsf(\Gsf)\ar[r] & \Gsf\ar[r]^{c\ \ \ \ \ } & \mathsf{Aut}(\Gsf)\ar[r]^p & \mathsf{Out}(\Gsf)\ar[r] & \mathsf{1}\ ,}
\end{equation}
with $c$ the map sending $g\in G$ to the corresponding inner automorphism $c_g$. Then $\Gsf$ is permutationally split if and only if there exists a normalized set theoretic section $s:Out(\Gsf)\To Aut(\Gsf)$ of $p$ such that the map $Out(\Gsf)\times Out(\Gsf)\To Inn(\Gsf)$ defined by 
$$
([\phi],[\phi'])\mapsto s[\phi]\circ s[\phi]\circ s[\phi\circ\phi']^{-1}
$$
has a normalized lifting $\psi_s:Out(\Gsf)\times Out(\Gsf)\To G$ satisfying the ``2-cocycle condition''
\begin{equation}\label{condicio_2-cocicle_G}
\psi_s([\phi],[\phi'])\ \psi_s([\phi\circ\phi'],[\phi''])=s[\phi](\psi_s([\phi'],[\phi'']))\ \psi_s([\phi],[\phi'\circ\phi'']).
\end{equation}
for all $[\phi],[\phi'],[\phi'']\in Out(\Gsf)$.
\end{prop}
\begin{proof}
As shown in Example~\ref{s.e._Equiv(G[1])}, (\ref{4-successio_S_G}) is indeed the exact 4-sequence of $\SSS_\Gsf$. Therefore, the only point which needs to be checked is that the term $s[x]\lhd_s\psi_s([x'],[x''])$ in (\ref{condicio_2-cocicle}) indeed reduces in this case to $s[\phi](\psi_s([\phi'],[\phi'']))$, and this is left to the reader.
\end{proof}

\begin{cor}\label{condicio_suficient_split_S_G}
Let $\Gsf$ be any group having a non-trivial outer automorphism $[\varphi]\in\mathsf{Out}(\Gsf)$ such that
\begin{itemize}
\item[(1)] $[\varphi]^2=[id_\Gsf]$ in $\mathsf{Out}(\Gsf)$, and
\item[(2)] for any automorphism $\phi\in[\varphi]$ and any $g\in G$ such that $\phi^2=c_g$, we have $\phi(g)\neq g$.
\end{itemize}
Then $\Gsf$ is non permutationally split.
\end{cor}
\begin{proof}
Let us assume that there exists $[\varphi]\in Out(\Gsf)$ as in the statement. Let $s$ be any normalized set theoretic section of $p$, and $\psi_s$ any normalized lifting of the induced map $\hat{s}$. Since $\psi_s$ is normalized, we have
$$
\psi_s([id_G],[\phi])=\psi_s([\phi],[id_G])=e
$$
for all $[\phi]\in Out(\Gsf)$. We claim that no such $\psi_s$ can satisfy (\ref{condicio_2-cocicle_G}) for all $[\phi],[\phi'],[\phi'']\in Out(\Gsf)$. Indeed, it is enough to take $([\phi],[\phi'],[\phi''])$ equal to $([\varphi],[\varphi],[\varphi])$. Since $[\varphi]^2=[id_G]$, (\ref{condicio_2-cocicle_G}) takes the form
$$
\psi_s([\varphi],[\varphi])\ \psi_s([id_G],[\varphi])=s[\varphi](\psi_s([\varphi],[\varphi]))\ \psi_s([\varphi],[id_G]),
$$
i.e.
\begin{equation}\label{condicio_sobre_psi_s}
\psi_s([\varphi],[\varphi])=s[\varphi](\psi_s([\varphi],[\varphi])).
\end{equation}
Now, since $s$ is normalized, item (1) in the statement implies that $\psi_s([\varphi],[\varphi])$ is an element $g\in G$ such that $s[\varphi]^2=c_g$. Hence, (\ref{condicio_sobre_psi_s}) requires the existence of some automorphism $\phi\in[\varphi]$ which leaves invariant at least one of the elements $g\in G$ such that $\phi^2=c_g$, and no such $\phi$ exists by item (2). Hence, $\Gsf$ can not be permutationally split. 
\end{proof}

Besides the abelian and the centerless groups, it readily follows from Proposition~\ref{caracteritzacio_grup_permutacionalment_escindit} that all groups with only inner automorphisms are also permutationally split. Hence any attempt to find a non permutationally split group requires looking at non-abelian groups with a non-trivial center and at least one non-trivial outer automorphism. Among the simplest examples of such groups we have the dihedral groups $\Dsf_n$ for $n\geq 4$ even (otherwise, $\Dsf_n$ is abelian or centerless). We already know that $\Dsf_4$ and $\Dsf_6$ are both permutationally split (in fact, both lead to the same permutation 2-group up to equivalence). However, $\Dsf_8$ is no longer permutationally split. In fact, we have the following.

\begin{prop}\label{2-grups_permutacions_no_escindits}
For any $k\geq 1$, the dihedral group $\Dsf_{8k}$ is non permutationally split and consequently, the permutation 2-group $\SSS ym(\Dsf_{8k})$ is non split.
\end{prop}
\begin{proof}
Let us see that $\Dsf_{8k}$ has a non-trivial outer automorphism satisfying the conditions in Corollary~\ref{condicio_suficient_split_S_G}. Recall that the group $\mathsf{Aut}(\Dsf_n)$ is isomorphic to $\ZZ_n\rtimes(\ZZ_n)^\ast$, with $(\ZZ_n)^\ast$ acting on $\ZZ_n$ by multiplication. An element $(\overline{p},\overline{q})\in\ZZ_n\rtimes(\ZZ_n)^\ast$ has to be identified with the automorphism mapping the generators $r,s$ of $\Dsf_n$ to $r^q$ and $sr^p$, respectively. We shall denote this automorphism by $\phi(p,q)$, with $p,q$ always thought mod $n$. For even $n$, the inner automorphisms are then the automorphisms $\phi(p,\pm 1)$ for all $p\in\{0,2,\ldots,n-2\}$. In fact, we have
$$
c_{r^l}=\phi(-2l,1),\qquad  c_{sr^l}=\phi(2l,-1)
$$
for any $l\in\{0,1,\ldots,n-1\}$. Moreover, the center is $\Zsf(\Dsf_n)=\{e,r^{n/2}\}$, so that for each inner automorphism there are exactly two elements in $\Dsf_n$ giving rise to it.

Let us now consider the case $n=8k$ for any $k\geq 1$. Clearly $(8k,4k-1)=1$, so that $\phi(1,4k-1)$ is a non inner automorphisms of $\Dsf_{8k}$. We claim that the outer automorphism $[\phi(1,4k-1)]$ satisfies the conditions in Corollary~\ref{condicio_suficient_split_S_G}. Indeed, its square is $\phi(4k,1)$, which is an inner automorphism. Moreover, one easily checks that $[\phi(1,4k-1)]$ consists of the automorphisms
$$
[\phi(1,4k-1)]=\{\phi(1,\pm(4k-1)),\phi(3,\pm(4k-1)),\ldots,\phi(8k-1,\pm(4k-1))\},
$$  
whose respective squares are given by
\begin{align*}
\phi(1+2i,4k-1)^2&=\phi(4k,1), \\ \phi(1+2i,1-4k)^2&=\phi(2+4i-4k,1)
\end{align*}
for $i=0,1,\ldots,4k-1$. The automorphism $\phi(4k,1)$ is both conjugation by $r^{2k}$ and by $r^{6k}$, while $\phi(2+4i-4k,1)$ is conjugation by $r^{2k-1-2i}$ and by $r^{6k-1-2i}$. Now, we have
\begin{align*}
&\phi(1+2i,4k-1)(r^{2k})=r^{6k}\neq r^{2k}, \\ &\phi(1+2i,4k-1)(r^{6k})=r^{2k}\neq r^{6k}, \\ &\phi(1+2i,1-4k)(r^{2k-1-2i})=r^{6k-1-2i}\neq r^{2k-1-2i}, \\ &\phi(1+2i,1-4k)(r^{6k-1-2i})=r^{2k-1-2i}\neq r^{6k-1-2i}.
\end{align*} 
Hence, by Corollary~\ref{condicio_suficient_split_S_G}, $\Dsf_{8k}$ is non permutationally split. 
\end{proof}

\bibliographystyle{plain}
\bibliography{2-grups_permutacions_I_versio_revisada}

\end{document}